\documentclass[a4paper,reqno,11pt]{amsart}
\usepackage{amsfonts,amsmath,amsthm,amssymb,stmaryrd}
\usepackage{mathrsfs}

\allowdisplaybreaks

\newtheorem{lem}{Lemma}[section]

\newtheorem{prop}[lem]{Proposition}
\newtheorem{thm}[lem]{Theorem}
\newtheorem{remark}[lem]{Remark}

\numberwithin{equation}{section}
\usepackage{color}

\usepackage[left=1 in, right=1 in,top=1 in, bottom=1 in]{geometry}

\providecommand{\abs}[1]{\left\vert#1\right\vert}
\providecommand{\norm}[1]{\left\Vert#1\right\Vert}

\providecommand{\Rn}[1]{\mathbb{R}^{#1}}

\providecommand{\snormspace}[3]{\left\Vert#1\right\Vert_{H^{#2}({#3})}}
\providecommand{\sd}[1]{\mathcal{D}_{#1}}
\providecommand{\se}[1]{\mathcal{E}_{#1}}
\providecommand{\sdb}[1]{\bar{\mathcal{D}}_{#1}}

\providecommand{\snormspace}[3]{\left\Vert#1\right\Vert_{H^{#2}({#3})}}
\providecommand{\ns}[1]{\norm{#1}^2}

\def\Lbrack{\left \llbracket}
\def\Rbrack{\right \rrbracket}

\def\({\left (}
\def\){\right )}
\DeclareMathOperator{\trace}{tr}
\DeclareMathOperator{\diverge}{div}

\providecommand{\Rn}[1]{\mathbb{R}^{#1}}
\providecommand{\norm}[1]{\left\Vert#1\right\Vert}
\def\ls{\lesssim}

\def\dt{\partial_t}

\def\H{{}_0H^1}

\def\na{\nabla}
\def\pa{\partial}

\def\rj{\Lbrack \bar{\rho} \Rbrack}

\providecommand{\jump}[1]{\left\llbracket #1 \right\rrbracket }

\def\RRvert2{\right \vert\! \right\vert}
\def\Lvert3{\left \vert\!\left\vert\!\left\vert}
\def\Rvert3{\right \vert\!\right\vert\!\right\vert}

\def\nab{\nabla}
\def\al{\alpha}
\def\dt{\partial_t}
\def\dtt{ \frac{d}{dt}}
\def\hal{\frac{1}{2}}

\def\ls{\lesssim}
\def\p{\partial}
\def\sg{\mathbb{D}}
\def\sgz{\mathbb{D}^0}

\def\da{\Delta_{\mathcal{A}}}
\def\naba{\nab_{\mathcal{A}}}
\def\diva{\diverge_{\mathcal{A}}}

\def\pal{\p^\alpha}

\def\a{\mathcal{A}}
\def\f{\mathcal{F}}

\def\fj1{\mathcal{J}^{-1}}

\def\n{\mathcal{N}}
\def\w{\mathcal{W}}

\def\y{\mathcal{Y}}
\def\z{\mathcal{Z}}
\def\q{q}
\def\S{\mathbb{S}}

\def\Ef{\mathfrak{E}}

\def\Ef{\mathfrak{E}_{2N}^\sigma}
\def\Df{\mathfrak{D}_{2N}}
\def\Af{\mathfrak{A}_{2N}^{j,k}}
\def\Bf{\mathfrak{B}_{2N}^{j,k}}
\def\Hf{\mathfrak{H}_{2N}^{j,k}}

\title[Compressible viscous surface-internal waves]{The compressible viscous surface-internal wave problem: nonlinear Rayleigh-Taylor instability}

\author{Juhi Jang}
\address{
Department of Mathematics\\
University of California, Riverside\\
Riverside, CA 92521, USA
}
\email[J. Jang]{juhijang@math.ucr.edu}
\thanks{J. Jang was supported in part by NSF grants DMS-1212142 and DMS-1351898.}

\author{Ian Tice}
\address{
Department of Mathematical Sciences\\
Carnegie Mellon University\\
Pittsburgh, PA 15213, USA
}
\email[I. Tice]{iantice@andrew.cmu.edu}

\author{Yanjin Wang}
\address{
School of Mathematical Sciences\\
Xiamen University\\
Xiamen, Fujian 361005, China}
\email[Y. J. Wang]{yanjin$\_$wang@xmu.edu.cn}
\thanks{Y. J. Wang was supported by the National Natural Science Foundation of China (No. 11201389)}


\begin{document}

\begin{abstract}
This paper concerns the dynamics of two layers of compressible, barotropic, viscous fluid lying atop one another. The lower fluid is bounded below by a rigid bottom, and the upper fluid is bounded above by a trivial fluid of constant pressure.  This is a free boundary problem: the interfaces between the fluids and above the upper fluid are free to move. The fluids are acted on by gravity in the bulk, and at the free interfaces we consider both the case of surface tension and the case of no surface forces.  We are concerned with the Rayleigh-Taylor instability when the upper fluid is heavier than the lower fluid along the equilibrium interface. When the surface tension at the free internal interface is below the critical value, we prove that the problem is nonlinear unstable.

\end{abstract}

\maketitle


\section{Introduction}


The Rayleigh-Taylor instability, one of the classic examples of hydrodynamic instability, is an interfacial instability between two fluids of different densities that occurs when a heavy fluid initially lies above a lighter one in a gravitational field. The instability is well-known since the classical work of Rayleigh \cite{3R} and of Taylor \cite{3T}, and it is one of the fundamental problems in fluid dynamics. A general discussion of the physics related to this topic can be found, for example, in \cite{3K}.

The Rayleigh-Taylor instability problem has received a lot of attention in the mathematics community due both to its physical importance and to the mathematical challenges it offers.  The linear Rayleigh-Taylor instability is well understood (see, for instance, Chandrasekhar's book \cite{3C}). However, there is no general theory that guarantees the passage from linear instability to nonlinear instability for PDEs, so the question of nonlinear instability is not immediately resolved by the linear analysis.

In this paper, we are concerned with the nonlinear dynamical Rayleigh-Taylor instability of viscous compressible two-fluids having different densities along a free interface, when the upper fluid is
heavier than the lower fluid along the equilibrium interface.  This is the final paper in a trio \cite{JTW_LWP, JTW_GWP} that establishes a sharp stability criterion for the compressible viscous surface-internal wave problem.

\subsection{Governing equations in Eulerian coordinates}

We consider two distinct,
immiscible, viscous, compressible, barotropic fluids evolving in a moving domain $\Omega(t)=\Omega_+(t)\cup \Omega_-(t)$ for time $t\ge0$. One fluid $(+)$, called the ``upper fluid'', fills the upper domain
\begin{equation}\label{omega_plus}
\Omega_+(t)=\{y\in  \mathrm{T}^2\times \mathbb{R}\mid \eta_-(y_1,y_2,t)<y_3< \ell +\eta_+(y_1,y_2,t)\},
\end{equation}
and the other fluid $(-)$, called the ``lower fluid'', fills the lower domain
\begin{equation}\label{omega_minus}
\Omega_-(t)=\{y\in  \mathrm{T}^2\times \mathbb{R}\mid  -b <y_3<\eta_-(y_1,y_2,t)\}.
\end{equation}
Here we assume the domains are horizontally periodic by setting $\mathrm{T}^2=(2\pi L_1\mathbb{T}) \times (2\pi L_2\mathbb{T})$ for $\mathbb{T} = \mathbb{R}/\mathbb{Z}$ the usual 1--torus and $L_1,L_2>0$ the periodicity lengths.  We assume that $\ell,b >0$ are two fixed and given constants, but the two surface functions $\eta_\pm$ are free and unknown.  The surface $\Gamma_+(t) = \{y_3= \ell  + \eta_+(y_1,y_2,t)\}$ is the moving upper boundary of $\Omega_+(t)$ where the upper fluid is in contact with the atmosphere, $\Gamma_-(t) = \{y_3=\eta_-(y_1,y_2,t)\}$ is the moving internal interface between the two fluids, and $\Sigma_b = \{y_3=-b \}$ is the fixed lower boundary of $\Omega_-(t)$.

The two fluids are described by their density and velocity functions, which are given for each $t\ge0$ by $\tilde{\rho}_\pm
(\cdot,t):\Omega_\pm (t)\rightarrow \mathbb{R}^+$ and $\tilde{u}_\pm (\cdot,t):\Omega_\pm (t)\rightarrow \mathbb{R}^3$, respectively.  In each fluid the pressure is a function of density: $P_\pm =P_\pm(\tilde{\rho}_\pm)>0$, and the pressure function is assumed to be smooth, positive, and strictly increasing.  For a vector function $u\in \Rn{3}$ we define the symmetric gradient by $(\sg u)_{ij} =   \p_i u_j + \p_j u_i$ for $i,j=1,2,3$;  its deviatoric (trace-free) part  is then
\begin{equation}\label{deviatoric_def}
 \sgz u = \sg u - \frac{2}{3} \diverge{u} I,
\end{equation}
where $I$ is the $3 \times 3$ identity matrix.  The viscous stress tensor in each fluid is then given by
\begin{equation}
\S_\pm(\tilde{u}_\pm) := \mu_\pm \sgz \tilde u_\pm +\mu'_\pm \diverge\tilde{u}_\pm I,
\end{equation}
where $\mu_\pm$ is the shear viscosity and  $\mu'_\pm$ is the bulk viscosity; we assume these satisfy the usual physical conditions
\begin{equation}\label{viscosity}
\mu_\pm>0,\quad \mu_\pm'\ge 0.
\end{equation}
The tensor $P_\pm(\tilde{\rho}_\pm) I-\S_\pm(\tilde{u}_\pm)$ is known as the stress tensor.  The divergence of a symmetric tensor $\mathbb{M}$ is defined to be the vector with components $(\diverge \mathbb{M})_i = \p_j \mathbb{M}_{ij}$.  Note then that
\begin{equation}
 \diverge\left( P_\pm(\tilde{\rho}_\pm) I-\S_\pm(\tilde{u}_\pm) \right) = \nab P_\pm(\tilde{\rho}_\pm) - \mu_\pm \Delta \tilde{u}_\pm - \left(\frac{\mu_\pm}{3} + \mu_\pm' \right)\nab \diverge{\tilde{u}_\pm}.
\end{equation}

For each $t>0$ we require that $(\tilde{\rho}_\pm,\tilde{u}_\pm, \eta_\pm)$ satisfy the following equations:
\begin{equation}\label{ns_euler}
\begin{cases}
\partial_t\tilde{\rho}_\pm+\diverge (\tilde{\rho}_\pm \tilde{u}_\pm)=0 & \text{in }\Omega_\pm(t)
\\\tilde{\rho}_\pm  (\partial_t\tilde{u}_\pm  +   \tilde{u}_\pm \cdot \nabla \tilde{u}_\pm ) +\nab P_\pm(\tilde{\rho}_\pm) -  \diverge \S_\pm(\tilde{u}_\pm) =-g\tilde{\rho}_\pm e_3 & \text{in } \Omega_\pm(t)
\\\partial_t\eta_\pm=\tilde{u}_{3,\pm}-\tilde{u}_{1,\pm}\partial_{y_1}\eta_\pm-\tilde{u}_{2,\pm}\partial_{y_2}\eta_\pm &\hbox{on } \Gamma_\pm(t)
\\(P_+(\tilde{\rho}_+)I-\S_+(\tilde{u}_+))n_+=p_{atm}n_+-\sigma_+ \mathcal{H}_+ n_+ &\hbox{on }\Gamma_+(t)
\\ (P_+(\tilde{\rho}_+)I-\S_+( \tilde{u}_+))n_-=(P_-(\tilde{\rho}_-)I-\S_-( \tilde{u}_-))n_-+ \sigma_- \mathcal{H}_-n_- &\hbox{on }\Gamma_-(t) \\\tilde{u}_+=\tilde{u}_-  &\hbox{on }\Gamma_-(t) \\\tilde{u}_-=0 &\hbox{on }\Sigma_b.
\end{cases}
\end{equation}
In the equations $-g \tilde{\rho}_\pm e_3$ is the gravitational force with the constant $g>0$ the acceleration of gravity and $e_3$ the vertical unit vector. The constant $p_{atm}>0$ is the atmospheric pressure, and we take $\sigma_\pm\ge 0$ to be the constant coefficients of surface tension. In this paper, we let $\nabla_\ast$ denote the horizontal gradient, $\diverge_\ast$ denote the horizontal divergence and $\Delta_\ast$ denote the horizontal Laplace operator. Then the upward-pointing unit normal of $\Gamma_\pm(t)$, $n_\pm$,  is given by
\begin{equation}
n_\pm=\frac{(-\nabla_\ast\eta_\pm,1)}
{\sqrt{1+|\nabla_\ast\eta_\pm|^2}},
\end{equation}
and  $\mathcal{H}_\pm$, twice the mean curvature of the surface $\Gamma_\pm(t)$, is given by the formula
\begin{equation}
\mathcal{H}_\pm=\diverge_\ast\left(\frac{\nabla_\ast\eta_\pm}
{\sqrt{1+|\nabla_\ast\eta_\pm|^2}}\right).
\end{equation}
The third equation in \eqref{ns_euler} is called the kinematic boundary condition since it implies that the free surfaces are advected with the fluids. The  boundary equations in \eqref{ns_euler} involving the stress tensor are called the dynamic boundary conditions. Notice that on $\Gamma_-(t)$, the continuity of velocity, $\tilde{u}_+ = \tilde{u}_-$,  means that it is the common value of $\tilde{u}_\pm$ that advects the interface. For a more physical description of the equations \eqref{ns_euler} and the boundary conditions in \eqref{ns_euler}, we refer to \cite{3WL}.

To complete the statement of the problem, we must specify the initial conditions. We assume that the initial surfaces $\Gamma_\pm(0)$ are given by the graphs of the functions $\eta_\pm(0)$, which yield the open sets $\Omega_\pm(0)$ on which we specify the initial data for the density, $\tilde{\rho}_\pm(0): \Omega_\pm(0) \rightarrow  \mathbb{R}^+$, and the velocity, $\tilde{u}_\pm(0): \Omega_\pm(0) \rightarrow  \mathbb{R}^3$. We will assume that $\ell+\eta_+(0)>\eta_-(0)>-b $ on $\mathrm{T}^2$, which means that at the initial time the boundaries do not intersect with each other.

\subsection{Equilibria}

We now seek a steady-state equilibrium solution to \eqref{ns_euler} with $\tilde{u}_\pm=0, \eta_\pm =0$, and the equilibrium domains given by the slabs
\begin{equation}
\Omega_+=\{y\in   \mathrm{T}^2\times \mathbb{R}\mid 0 < y_3< \ell \} \text{ and }
\Omega_-=\{y\in  \mathrm{T}^2\times \mathbb{R}\mid  -b <y_3< 0 \}.
\end{equation}
Then the system \eqref{ns_euler} reduces to the ODEs for the equilibrium densities  $\tilde\rho_\pm = \bar{\rho}_\pm(y_3)$:
\begin{equation}\label{steady}
\begin{cases}
\displaystyle\frac{d(P_+ (\bar{\rho}_+ ))}{dy_3} = -g\bar{\rho}_+, & \text{for }y_3 \in (0,\ell), \\
\displaystyle\frac{d(P_- (\bar{\rho}_- ))}{dy_3} = -g\bar{\rho}_-, & \text{for } y_3 \in (-b,0), \\
P_+(\bar{\rho}_+(\ell)) = p_{atm}, \\
P_+(\bar{\rho}_+(0))  =P_-(\bar{\rho}_-(0)).
\end{cases}
\end{equation}
The system \eqref{steady} admits a solution $\bar{\rho}_\pm >0$ if and only if the equilibrium heights $b,\ell>0$, the pressure laws $P_\pm$, and the atmospheric pressure $p_{atm}$ fulfill a collection of admissibility conditions.  These are enumerated in detail in our companion paper \cite{JTW_GWP}.  For the sake of brevity we will not repeat them here, but we will assume that those admissibility conditions are satisfied so that an equilibrium exists. We remark that the equilibrium densities $\bar{\rho}$ are strictly positive and smooth when restricted to  $[-b,0]$ and $[0,\ell]$.

We 
denote the equilibrium density at the fluid interfaces by:
\begin{equation}
 \bar{\rho}_1 = \bar{\rho}_+(\ell), \;  \bar{\rho}^+ = \bar{\rho}_+(0),\;  \bar{\rho}^- = \bar{\rho}_-(0).
\end{equation}
Notice in particular that the equilibrium density can jump across the internal interface.   The jump in the equilibrium density, which we denote by
\begin{equation}\label{rho+-}
\rj := \bar{\rho}_+(0)-\bar{\rho}_-(0)= \bar{\rho}^+ - \bar{\rho}^-,
\end{equation}
is of fundamental importance in the analysis of solutions to \eqref{ns_euler} near the equilibrium. Since we are interested in the Rayleigh-Taylor instability, we assume $\rj > 0$, that is, the upper fluid is heavier than the lower fluid along the equilibrium interface.  We refer to our companion paper \cite{JTW_GWP} for the analysis of the stable regime $\rj \le 0$.

In studying perturbations of the equilibrium density it will be useful 
to employ the enthalpy functions.  These are defined in terms of the pressure laws $P_\pm$ and the equilibrium density values via
\begin{equation}\label{h'}
 h_+(z) = \int_{\bar{\rho}_1}^z \frac{P'_+(r)}{r}dr \text{ and }  h_-(z) = \int_{\bar{\rho}^-}^z \frac{P'_-(r)}{r}dr.
\end{equation}

\subsection{Reformulation in flattened coordinates}\label{sec1.3}

The movement of the free surfaces $\Gamma_\pm(t)$ and the subsequent change of the domains $\Omega_\pm(t)$ create numerous mathematical difficulties. To circumvent these, we will switch to a  coordinate system in which  the boundaries and the domains stay fixed in time.  In order to be consistent with our study of the nonlinear stability of the equilibrium state in \cite{JTW_GWP}, we will use the equilibrium domain as the fixed domain. We will not use a Lagrangian coordinate transformation, but rather utilize a special flattening coordinate transformation motivated by Beale \cite{B2}.

To this end, we define the fixed domain
\begin{equation}
\Omega = \Omega_+\cup\Omega_-\text{ with }\Omega_+:=\{0<x_3<\ell \} \text{ and } \Omega_-:=\{-b<x_3<0\},
\end{equation}
for which we have written the coordinates as $x\in \Omega$. We shall write $\Sigma_+:=\{x_3= \ell\}$ for the upper boundary, $\Sigma_-:=\{x_3=0\}$ for the internal interface and $\Sigma_b:=\{x_3=-b\}$ for the lower boundary.  Throughout the paper we will write $\Sigma = \Sigma_+ \cup \Sigma_-$.   We think of $\eta_\pm$ as a function on $\Sigma_\pm$ according to $\eta_+: (\mathrm{T}^2\times\{\ell\}) \times \mathbb{R}^{+} \rightarrow\mathbb{R}$ and $\eta_-:(\mathrm{T}^2\times\{0\}) \times \mathbb{R}^{+} \rightarrow \mathbb{R}$, respectively. We will transform the free boundary problem in $\Omega(t)$ to one in the fixed domain $\Omega $ by using the unknown free surface functions $\eta_\pm$. For this we define
\begin{equation}
\bar{\eta}_+:=\mathcal{P}_+\eta_+=\text{Poisson extension of }\eta_+ \text{ into }\mathrm{T}^2 \times \{x_3\le \ell\}
\end{equation}
and
\begin{equation}
\bar{\eta}_-:=\mathcal{P}_-\eta_-=\text{specialized Poisson extension of }\eta_-\text{ into }\mathrm{T}^2 \times \mathbb{R},
\end{equation}
where $\mathcal{P}_\pm$ are defined in the appendix by \eqref{P+def} and \eqref{P-def}. The Poisson extensions $\bar{\eta}_\pm$ allow us to flatten the coordinate domains via the following special coordinate transformation:
\begin{equation}\label{cotr}
\Omega_\pm \ni x\mapsto(x_1,x_2, x_3+ \tilde{b}_1\bar{\eta}_++\tilde{b}_2\bar{\eta}_-):=\Theta (x,t)=(y_1,y_2,y_3)\in\Omega_\pm(t),
\end{equation}
where we have chosen $\tilde{b}_1=\tilde{b}_1(x_3), \tilde{b}_2=\tilde{b}_2(x_3)$ to be two smooth functions in $\mathbb{R}$ that satisfy
\begin{equation}\label{b function}
\tilde{b}_1(0)=\tilde{b}_1(-b)=0, \tilde{b}_1(\ell)=1\text{ and }\tilde{b}_2(\ell)=\tilde{b}_2(-b)=0, \tilde{b}_2(0)=1.
\end{equation}
Note that $\Theta(\Sigma_+,t)=\Gamma_+(t),\ \Theta (\Sigma_-,t)=\Gamma_-(t)$ and $\Theta(\cdot,t) \mid_{\Sigma_b} = Id \mid_{\Sigma_b}$.

If $\eta $ is sufficiently small (in an appropriate Sobolev space), then the mapping $\Theta $ is a diffeomorphism.  This allows us to transform the problem \eqref{ns_euler} to one in the fixed spatial domain $\Omega$ for each $t\ge 0$.
In order to write down the equations in the new coordinate system, we compute
\begin{equation}\label{A_def}
\begin{array}{ll} \nabla\Theta  =\left(\begin{array}{ccc}1&0&0\\0&1&0\\A  &B  &J  \end{array}\right)
\text{ and }\mathcal{A}  := \left(\nabla\Theta
^{-1}\right)^T=\left(\begin{array}{ccc}1&0&-A   K  \\0&1&-B   K  \\0&0&K
\end{array}\right)\end{array}.
\end{equation}
Here the components in the matrix are
\begin{equation}\label{ABJ_def}
A  =\p_1\theta ,\
B  =\p_2\theta,\
J = 1 + \p_3\theta,\  K  =J^{-1},
\end{equation}
where we have written
\begin{equation}\label{theta}
\theta:=\tilde{b}_1\bar{\eta}_++\tilde{b}_2\bar{\eta}_-.
\end{equation}
Notice that $J={\rm det}\, \nabla\Theta $ is the Jacobian of the coordinate transformation. It is straightforward to check that, because of how we have defined $\bar{\eta}_-$ and $\Theta $, the matrix $\mathcal{A}$ is regular across the interface $\Sigma_-$.

We now define the density $\rho_\pm$ and the velocity $u_\pm$ on $\Omega_\pm$ by the compositions $\rho_\pm(x,t)=\tilde \rho_\pm(\Theta_\pm(x,t),t)$ and $  u_\pm(x,t)=\tilde u_\pm(\Theta_\pm(x,t),t)$. Since the domains $\Omega_\pm$ and the boundaries $\Sigma_\pm$ are now fixed, we henceforth consolidate notation by writing $f$ to refer to $f_\pm$ except when necessary to distinguish the two; when we write an equation for $f$ we assume that the equation holds with the subscripts added on the domains $\Omega_\pm$ or $\Sigma_\pm$. To write the jump conditions on $\Sigma_-$, for a quantity $f=f_\pm$, we define the interfacial jump as
\begin{equation}
\jump{f} := f_+ \vert_{\{x_3=0\}} - f_- \vert_{\{x_3=0\}}.
\end{equation}
In the new coordinates, the PDEs \eqref{ns_euler} become the following system for $(\rho, u,\eta)$:
\begin{equation}\label{ns_geometric}
\begin{cases}
\partial_t \rho-K\p_t\theta\p_3\rho +\diverge_\a (  {\rho}   u)=0 & \text{in }
\Omega  \\
\rho (\partial_t    u -K\p_t\theta\p_3 u+u\cdot\nabla_\a u  ) + \nabla_\a P ( {\rho} )    -\diva \S_\a (u) =- g\rho e_3 & \text{in }
\Omega
\\ \partial_t \eta = u\cdot \n &
\text{on }\Sigma
\\ (P ( {\rho} ) I- \S_{\a}(u))\n
=p_{atm}\n -\sigma_+  \mathcal{H} \n  &\hbox{on }\Sigma_+
 \\ \jump{P ( {\rho} ) I- \S_\a(u)}\n
= \sigma_-  \mathcal{H} \n  &\hbox{on }\Sigma_-
 \\\jump{u}=0   &\hbox{on }\Sigma_-\\  {u}_- =0 &\text{on }\Sigma_b.
\end{cases}
\end{equation}
Here we have written the differential operators $\naba$, $\diva$, and $\da$ with their actions given by
\begin{equation}
 (\naba f)_i := \a_{ij} \p_j f,\; \diva X := \a_{ij}\p_j X_i, \text{ and }\da f := \diva \naba f
\end{equation}
for appropriate $f$ and $X$.  We have also written
\begin{equation}\label{n_def}
\n := (-\p_1 \eta, - \p_2 \eta,1)
\end{equation}
for the non-unit normal vector to $\Sigma(t)$, and we have written
\begin{multline}\label{deviatoric_a_def}
(\sg_{\a} u)_{ij} =  \a_{ik} \p_k u_j + \a_{jk} \p_k u_i, \qquad \sgz_{\a} u = \sg_{\a} u - \frac{2}{3} \diva u I,\\
 \text{and } \S_{\a,\pm}(u): =\mu_\pm \sgz_{\a} u+ \mu_\pm' \diva u I.
\end{multline}
Note that if we extend $\diva$ to act on symmetric tensors in the natural way, then $\diva \S_{\a} u =\mu\Delta_\a u+(\mu/3+\mu')\nabla_\a \diverge_\a u$. Recall that $\a$ is determined by $\eta$ through \eqref{A_def}. This means that all of the differential operators in \eqref{ns_geometric} are connected to $\eta$, and hence to the geometry of the free surfaces.


\subsection{Perturbation equations}

We will now rephrase the PDEs \eqref{ns_geometric} in a perturbation formulation around the steady-state solution 
$(\bar\rho, 0,0)$. We define a special density perturbation by
\begin{equation}\label{q_def}
 \q=\rho-\bar\rho- \p_3\bar\rho\theta.
\end{equation}
For the pressure term $P(\rho)=P(\bar\rho+\q+ \p_3\bar\rho\theta)$, we expand it via the Taylor expansion: by \eqref{steady} we have
\begin{equation}\label{R1}
P (\bar\rho+\q+\p_3\bar\rho\theta)=P (\bar{\rho} )+P '(\bar{\rho} )(\q+\p_3\bar\rho\theta)+\mathcal{R}=P (\bar{\rho} )+P '(\bar{\rho} ) \q -g\bar\rho\theta+\mathcal{R},
\end{equation}
where the remainder term is given by
\begin{equation}\label{R_def}
\mathcal{R} =\int_{\bar{\rho} }^{\bar{\rho} +\q+\p_3\bar\rho\theta}(\bar{\rho} +\q+\p_3\bar\rho\theta-z)  P ^{\prime\prime}(z)\,dz.
\end{equation}
Recalling \eqref{rho+-}, \eqref{b function}, and \eqref{theta}, we find that
\begin{equation}
 -g\bar\rho_+\theta=-\bar{\rho}_1  g\eta_+\text{ on }\Sigma_+,\text{ and } \jump{-g\bar\rho\theta}  =-\rj g \eta_-\text{ on }\Sigma_-.
\end{equation}

The equations \eqref{ns_geometric} can be written as the following system when perturbed around the equilibrium $(\bar{\rho},0,0)$:
\begin{equation}\label{geometric}
\begin{cases}
\partial_t \q +\diverge_\a((\bar{\rho} + \q+\p_3\bar\rho\theta) u ) - \p_3^2\bar\rho K \theta \p_t\theta - K\p_t\theta \pa_3  \q =0 & \text{in } \Omega
\\
( \bar{\rho} +  \q+\p_3\bar\rho\theta)\partial_t    u
+( \bar{\rho} +  \q+\p_3\bar\rho\theta) (-K\p_t\theta \pa_3  u  +   u \cdot \nab_\a  u )
+ \bar{\rho}\nabla_\a \left(h'(\bar{\rho})\q\right)  \\
\quad -\diva \S_{\a} u =- \nabla_\a\mathcal{R}-g( \q+\p_3\bar\rho\theta ) \nabla_\a \theta & \text{in }
\Omega
\\
\partial_t \eta = u\cdot \n & \text{on }\Sigma
\\
(  P'(\bar\rho)\q I- \S_{\a}(  u))\n  =  \bar{\rho}_1  g \eta \n-\sigma_+ \mathcal{H}   \n
- \mathcal{R}_+ \n
 & \text{on } \Sigma_+
 \\
 \jump{P'(\bar\rho)\q I- \S_\a(u)}\n = \rj g\eta\n +\sigma_- \mathcal{H} \n  - \jump{ \mathcal{R} }\n
 &\text{on }\Sigma_-
\\
 \jump{u}=0 &\text{on } \Sigma_-
\\
u_-=0 &\hbox{on }\Sigma_b.
\end{cases}
\end{equation}

\begin{remark}
The special density perturbation $q$ given by \eqref{q_def} and the subsequent perturbation equations of the form \eqref{geometric} are crucial for our study in the stability regime in \cite{JTW_GWP}. However, it is not essential for the instability regime in this paper. Indeed, we could consider $\rho-\bar\rho$ directly. We choose here to consider $q$ in order to be consistent with the study in \cite{JTW_GWP}.
\end{remark}

\subsection{Main result}



For a given jump value in the equilibrium density {$\jump{\bar\rho}>0$}, we define the critical surface tension value by
\begin{equation}\label{sigma_c}
\sigma_c:= {\jump{\bar\rho}}g\max\{L_1^2,L_2^2\}.
\end{equation}
 In our companion paper \cite{JTW_GWP}, we have proved the global existence of solutions decaying to the equilibrium state $(0,0,0)$ in the problem \eqref{geometric} when $\sigma_- > \sigma_c$. The goal of this paper is to show that when $\sigma_-<\sigma_c$, the equilibrium state $(0,0,0)$ is unstable in the compressible viscous surface-internal wave problem \eqref{geometric}.

Our main result can be stated as follows:

\begin{thm}\label{maintheorem}
Assume that $\rj >0$ and $\sigma_-<\sigma_c$, where $\sigma_c$ is defined by \eqref{sigma_c}. Let the triple norm $\Lvert3   \cdot   \Rvert3_{00}$ be defined by \eqref{norm3} (with $N\ge 3$ an integer). There exist constants $\theta_0>0$  and $C>0$ such that for any sufficiently small $0< \iota<\theta_0$ there  exist solutions $(q^\iota(t), u^\iota(t) , \eta^\iota (t))$ to \eqref{geometric} such that
 \begin{equation}
  \Lvert3   ( q^\iota(0) , u^\iota(0) , \eta^\iota(0) )   \Rvert3_{00}\le C\iota,\hbox{ but }  \norm{ \eta_-^\iota(T^\iota)}_{L^2}\ge \frac{\theta_0}{2}.
 \end{equation}
Here the escape time $T^\iota>0$ is
 \begin{equation}\label{escape_time}
T^\iota:=\frac{1}{\lambda}\log\frac{\theta_0}{\iota},
\end{equation}
 where $\frac{\Lambda}{2}<\lambda\le \Lambda$ with $\Lambda$ the sharp linear growth rate defined by \eqref{Lambda}.
\end{thm}

\begin{remark}
Theorem \ref{maintheorem} shows that the instability occurs in the $L^2$ norm of $\eta_-$.   This highlights the fact that the instability occurs at the internal interface. This also means that although our instability analysis works in a framework with some degree of regularity, the onset of instability occurs at the lowest level of regularity.
\end{remark}

\begin{remark}
Our results can be readily applied to the compressible viscous internal wave problem,  i.e. the problem posed with a rigid top in place of the upper free surface.
\end{remark}

To our best knowledge, this work is the first rigorous result to address the nonlinear Rayleigh-Taylor instability for compressible viscous fluids with or without surface tension; for the compressible viscous internal wave problem, the linear instability was shown by Guo and Tice \cite{3GT2}. Theorem \ref{maintheorem} together with our results in \cite{JTW_GWP} establish  sharp nonlinear stability criteria for the equilibrium state in the compressible viscous surface-internal wave problem.  We summarize these and the rates of decay to equilibrium in the following table.

\begin{displaymath}
\begin{array}{| c | c | c |  c |}
\hline
 & \rj < 0 & \rj =0 & \rj >0  \\ \hline
 \sigma_\pm =0 &
\begin{array}{c}
\textnormal{nonlinearly stable} \\
\textnormal{almost exponential decay}
\end{array} &
\textnormal{locally well-posed} &
\textnormal{nonlinearly unstable}     \\ \hline
\begin{array}{c}
0 < \sigma_+  \\ 0 < \sigma_- < \sigma_c
\end{array} &
\begin{array}{c}
\textnormal{nonlinearly stable} \\
\textnormal{exponential decay}
\end{array} &
\begin{array}{c}
\textnormal{nonlinearly stable} \\
\textnormal{exponential decay}
\end{array} &
\textnormal{nonlinearly unstable}      \\ \hline
\begin{array}{c}
0 < \sigma_+  \\  \sigma_c = \sigma_-
\end{array} &
\begin{array}{c}
\textnormal{nonlinearly stable} \\
\textnormal{exponential decay}
\end{array} &
\begin{array}{c}
\textnormal{nonlinearly stable} \\
\textnormal{exponential decay}
\end{array} &
\textnormal{locally well-posed}
\\ \hline
\begin{array}{c}
0 < \sigma_+  \\  \sigma_c < \sigma_-
\end{array} &
\begin{array}{c}
\textnormal{nonlinearly stable} \\
\textnormal{exponential decay}
\end{array} &
\begin{array}{c}
\textnormal{nonlinearly stable} \\
\textnormal{exponential decay}
\end{array} &
\begin{array}{c}
\textnormal{nonlinearly stable} \\
\textnormal{exponential decay}
\end{array}
\\ \hline
\end{array}
\end{displaymath}
Note that our results do not cover the critical case:  $\sigma_- = \sigma_c$. From \cite{JTW_LWP} we know  that the problem is locally well-posed, but at the moment of writing it is not clear to us what the stability of the system should be.

We mention some previous mathematical results concerning the Rayleigh-Taylor instability. For the inviscid  Rayleigh-Taylor problem without surface tension, Ebin \cite{3E} proved the nonlinear ill-posedness of the problem for incompressible fluids, Guo and Tice \cite{3GT1} showed an analogous result for compressible fluids, and Hwang and Guo \cite{hw_guo} proved the nonlinear instability of the incompressible problem with a continuous density distribution. For the viscous Rayleigh-Taylor problem, Pr$\ddot{\text{u}}$ss and Simonett \cite{PS2} proved the nonlinear instability for incompressible fluids with surface tension in an $L^p$ setting by using Henry's instability theorem, and Wang, Tice, and Kim \cite{WT, WTK} established the sharp nonlinear instability criteria  for the incompressible surface-internal wave problem with or without surface tension.

Since linear instability can be established in the same way as that for the compressible viscous internal wave problem in \cite{3GT2}, the heart of the proof of Theorem \ref{maintheorem} is the passage from linear instability to nonlinear instability.  This is in general a delicate issue for PDEs since the spectrum of the linear part is fairly complicated and the unboundedness of the nonlinear part usually yields a loss in derivatives. Our proof is based on a variant of the bootstrap argument first developed by Guo and Strauss \cite{GS}. The main strategy of Guo-Strauss approach is to show that on the time scale of the instability, higher-regularity norms of the solution are actually bounded by the growth of low-regularity norms (in our case $L^2$).  For our problem, the term $\rj g \eta_-$ along the interface is the cause of the instability; since it is of low order, we are led to use the Guo-Strauss bootstrap framework here.

We encounter a number of mathematical difficulties in analyzing our complicated nonlinear problem, especially due to the fact that our problem is defined in a domain with a boundary.  First, even to guarantee the existence of local-in-time solutions in our energy space, the initial data must satisfy certain nonlinear compatibility conditions that the growing modes to the linearized problem constructed in Section \ref{growing mode} would not satisfy. We employ an argument from Jang and Tice \cite{JT} that uses the linear growing mode to construct initial data for the nonlinear problem.

Second, because the spectrum of the solution operator for the linearized problem is complicated, we can only derive the largest growth rate for the linearized problem by using careful energy estimates as in \cite{3GT2}; this is in the context of strong solutions, which requires the initial data to satisfy the linear compatibility conditions.  Such estimates would not be applicable to the nonlinear problem by directly employing Duhamel's principle.  To get around this issue, we provide the estimates for the growth in time of arbitrary solutions of the linear inhomogeneous equations in Section \ref{growth}; clearly, the estimates can be applied directly to the nonlinear problem.

The last difficulty is to derive the bootstrap energy estimates, a key step in showing the instability of $\eta_-$.   We employ a variant of the energy method elaborated in our companion paper \cite{JTW_GWP} to derive these estimates in Section \ref{energy}. There are new ingredients.  First, we need to weaken the dissipation (see \eqref{p_dissipation_def} for $\sd{2N}^\sigma$) due to the lack of the dissipation estimates of $L^2$ norm of $q$ and $\eta$. The missing terms are controlled by using the energy $\se{2N}^\sigma$ instead.  This is not possible in the global analysis of \cite{JTW_GWP} but is effective in the local-in-time framework of our instability analysis.   Second, the term $\rj g \eta_-$ contributes  negative energy in the unstable regime and thus the estimates of  $\eta_-$ can not be obtained simultaneously with the other terms.  We derive the estimates of $\eta_-$ by making use of the kinematic boundary condition, a transport equation for $\eta_-$.   Our complete bootstrap estimate is recorded in Theorem \ref{energyeses}, which shows that the stronger Sobolev norm $\Lvert3   \cdot   \Rvert3_{00}$ of the solution is actually bounded by the lower-order norm $\norm{ \eta_- }_{L^2}$.  The bootstrap analysis allows us to finally prove Theorem \ref{maintheorem} in Section \ref{proof}.

\subsection{Definitions and terminology} \label{def-ter}

We now mention some of the definitions, bits of notation, and conventions that we will use throughout the paper.

{ \bf Universal constants}

Throughout the paper we will refer to generic constants as ``universal'' if they depend on $N$, $\Omega_\pm$,  the various parameters of the problem (e.g. $g$, $\mu_\pm$, $\mu'_\pm$, $\sigma_\pm$) and the functions $\bar{\rho}_\pm$,  with the caveat that if the constant depends on $\sigma_\pm$, then it remains bounded above as either $\sigma_\pm$ tend to $0$.  For example this allows constants of the form $g\mu_+ + 3 \sigma_-^2 + \sigma_+$ but forbids constants of the form $3 + 1/\sigma_-$. We make this choice in order to be able to handle together all the cases $\sigma_\pm\ge 0$.

We will employ the notation $a \ls b$ to mean that $a \le C b$ for a universal constant $C>0$. Universal constants are allowed to change from line to line. When a constant depends on a quantity $z$ we will write $C = C(z)= C_z$ to indicate this. To indicate some constants in some places so that they can be referred to later, we will denote them in particular by $C_1,C_2$, etc.

{ \bf Norms }

We write $H^k(\Omega_\pm)$ with $k\ge 0$ and $H^s(\Sigma_\pm)$ with $s \in \Rn{}$ for the usual Sobolev spaces.   We will typically write $H^0 = L^2$.  If we write $f \in H^k(\Omega)$, the understanding is that $f$ represents the pair $f_\pm$ defined on $\Omega_\pm$ respectively, and that $f_\pm \in H^k(\Omega_\pm)$. We employ the same convention on $\Sigma_\pm$.  We will refer to the space $\H(\Omega)$ defined as follows:
\begin{equation}
 \H(\Omega) = \{ v \in H^1(\Omega) \; \vert \; \jump{v}=0 \text{ on } \Sigma_- \text{ and } v_- = 0 \text{ on } \Sigma_b\}.
\end{equation}

To avoid notational clutter, we will avoid writing $H^k(\Omega)$ or $H^k(\Sigma)$ in our norms and typically write only $\norm{\cdot}_{k}$, which we actually use to refer to  sums
\begin{equation}
 \ns{f}_k = \ns{f_+}_{H^k(\Omega_+)} + \ns{f_-}_{H^k(\Omega_-)}   \text{ or }  \ns{f}_k = \ns{f_+}_{H^k(\Sigma_+)} + \ns{f_-}_{H^k(\Sigma_-)}.
\end{equation}
Since we will do this for functions defined on both $\Omega$ and $\Sigma$, this presents some ambiguity.  We avoid this by adopting two conventions.  First, we assume that functions have natural spaces on which they ``live.''  For example, the functions $u$, $\rho$, $q$, and $\bar{\eta}$ live on $\Omega$, while $\eta$ lives on $\Sigma$.  As we proceed in our analysis, we will introduce various auxiliary functions; the spaces they live on will always be clear from the context.  Second, whenever the norm of a function is computed on a space different from the one in which it lives, we will explicitly write the space.  This typically arises when computing norms of traces onto $\Sigma_\pm$ of functions that live on $\Omega$.

Occasionally we will need to refer to the product of a norm of $\eta$ and a constant that depends on $\pm$.  To denote this we will write
\begin{equation}
\gamma \ns{\eta}_{k} = \gamma_+ \ns{\eta_+}_{H^k(\Sigma_+)} + \gamma_- \ns{\eta_-}_{H^k(\Sigma_-)}.
\end{equation}

{ \bf Derivatives }

We write $\mathbb{N} = \{ 0,1,2,\dotsc\}$ for the collection of non-negative integers.  When using space-time differential multi-indices, we will write $\mathbb{N}^{1+m} = \{ \alpha = (\alpha_0,\alpha_1,\dotsc,\alpha_m) \}$ to emphasize that the $0-$index term is related to temporal derivatives.  For just spatial derivatives we write $\mathbb{N}^m$.  For $\alpha \in \mathbb{N}^{1+m}$ we write $\pal = \dt^{\alpha_0} \p_1^{\alpha_1}\cdots \p_m^{\alpha_m}.$ We define the parabolic counting of such multi-indices by writing $\abs{\alpha} = 2 \alpha_0 + \alpha_1 + \cdots + \alpha_m.$  We will write $\nab_{\ast}f$ for the horizontal gradient of $f$, i.e. $\nab_{\ast}f = \p_1 f e_1 + \p_2 f e_2$, while $\nab f$ will denote the usual full gradient.

For a \textit{given norm} $\norm{\cdot}$ and an integer $k\ge 0$, we introduce the following notation for sums of spatial derivatives:
\begin{equation}
 \norm{\nab_{\ast}^k f}^2 := \sum_{\substack{\alpha \in \mathbb{N}^2 \\  \abs{ \alpha}\le k} } \norm{\pa^\al  f}^2 \text{ and }
\norm{\nab^k f}^2 := \sum_{\substack{\alpha \in \mathbb{N}^{3} \\   \abs{\alpha}\le k} } \norm{\pa^\al  f}^2.
\end{equation}
The convention we adopt in this notation is that $\nab_{\ast}$ refers to only ``horizontal'' spatial derivatives, while $\nab$ refers to full spatial derivatives.   For space-time derivatives we add bars to our notation:
\begin{equation}\label{barnotation}
 \norm{\bar{\nab}_{\ast}^k f}^2 := \sum_{\substack{\alpha \in \mathbb{N}^{1+2} \\   \abs{\alpha}\le k} } \norm{\pa^\al  f}^2 \text{ and }
\norm{\bar{\nab}^k f}^2 := \sum_{\substack{\alpha \in \mathbb{N}^{1+3} \\   \abs{\alpha}\le k} } \norm{\pa^\al  f}^2.
\end{equation}
We allow for composition of derivatives in this counting scheme in a natural way; for example, we write
\begin{equation}
 \norm{\nab_{\ast} \nab_{\ast}^{k} f}^2 = \norm{ \nab_{\ast}^k \nab_{\ast} f}^2 = \sum_{\substack{\alpha \in \mathbb{N}^{2} \\   \abs{\alpha}\le k} } \norm{\pa^\al  \nab_{\ast} f}^2  = \sum_{\substack{\alpha \in \mathbb{N}^{2} \\  1\le \abs{\alpha}\le k+1} } \norm{\pa^\al   f}^2.
\end{equation}

\section{Growing mode solution to the linearized equations}\label{growing mode}

In this section, we consider the linearization of \eqref{geometric}:
\begin{equation}\label{linear}
\begin{cases}
\partial_t \q +\diverge ( \bar{\rho}   u)=0 & \text{in }
\Omega  \\
 \bar{\rho} \partial_t    u   + \bar{\rho}\nabla \left(h'(\bar{\rho})\q\right)   -\diverge \S(u) =0 & \text{in }
\Omega  \\
\partial_t \eta = u_3 &
\text{on }\Sigma  \\
(  P'(\bar\rho)\q I- \S(  u))e_3  = (\rho_1  g \eta_+ -\sigma_+ \Delta_\ast \eta_+ ) e_3
 & \text{on } \Sigma_+
 \\\jump{P'(\bar\rho)\q I- \S(u)}e_3
=(\rj g\eta_- +\sigma_- \Delta_\ast \eta_-)e_3&\hbox{on }\Sigma_-
\\\jump{u}=0 &\hbox{on }\Sigma_-
\\ u_-=0 &\hbox{on }\Sigma_b.
\end{cases}%
\end{equation}
We seek a growing mode solution to \eqref{linear} of the following form:
\begin{equation}\label{ansatz}
 u(x,t) = w(x) e^{\lambda t},\ \q(x,t)= \tilde{q}(x) e^{\lambda t},\ \eta(x',t) = \tilde{\eta}(x') e^{\lambda t}
\end{equation}
for some $\lambda>0$ (the same in the upper and lower fluids), where $x'=(x_1,x_2)$. Substituting the ansatz \eqref{ansatz} into \eqref{linear}, we find that
\begin{equation}\label{ansatz1}
\tilde{q}={-\lambda^{-1}{\rm div}(\bar\rho w)}\;\text{ and }\; \tilde{\eta}=\lambda^{-1}w_3|_{\Sigma}.
\end{equation}
By using \eqref{ansatz1}, we can eliminate $\tilde{q},\tilde{\eta}$ from \eqref{linear} and arrive at the following time-invariant system for $w$:
\begin{equation}\label{linear2}
\begin{cases}
\lambda^2 \bar{\rho} w -\bar{\rho}\nabla \left(h'(\bar{\rho})\diverge ( \bar{\rho}   w)\right)   -\lambda\diverge \S(w) =0 & \text{in }
\Omega  \\
( -P'(\bar\rho)\diverge ( \bar{\rho}   w) I- \lambda\S( w))e_3  = (\rho_1  g   w_{3} -\sigma_+ \Delta_\ast   w_{3} ) e_3
 & \text{on } \Sigma_+
 \\ \jump{-P'(\bar\rho)\diverge ( \bar{\rho}   w) I- \lambda\S(w)}e_3
=(\rj g  w_{3} +\sigma_- \Delta_\ast w_{3})e_3&\hbox{on }\Sigma_-
\\\jump{w}=0 &\hbox{on }\Sigma_-
\\w_-=0 &\hbox{on }\Sigma_b.
\end{cases}%
\end{equation}

Since the coefficients of the linear problem \eqref{linear2} only depend on the $x_3$ variable, we are free to make the further structural assumption that the $x'$ dependence of $w$ is given as a Fourier mode $e^{i x' \cdot \xi}$ for the spatial frequency $\xi=(\xi_1,\xi_2)\in L_1^{-1}\mathbb{Z}\times L_2^{-1}\mathbb{Z}$.  Together with the growing mode ansatz \eqref{ansatz}, this constitutes a ``normal mode'' ansatz, which is standard in fluid stability analysis \cite{3C}.  We define the new unknowns $\varphi,\theta,\psi:(-b,\ell) \rightarrow \Rn{}$ according to
\begin{equation}\label{ansatz2}
w_1(x)=-i\varphi( x_3){\rm e}^{ix'\cdot\xi},\ w_2(x)=-i\theta( x_3){\rm e}^{ix'\cdot\xi},
\ \text{and }w_3(x)= \psi( x_3){\rm e}^{ix'\cdot\xi}.
\end{equation}
For each fixed $\xi$, and for the new unknowns $\varphi(x_3), \theta(x_3), \psi(x_3)$, and $\lambda$, we obtain the following system of ODEs (here $' = d/dx_3$):
\begin{equation}\label{linear3}
\begin{cases}
-\left(\lambda \mu  \varphi'\right)' + \left[ \lambda^2 \bar\rho   + \lambda \mu  \abs{\xi}^2 + \xi_1^2 \left( \lambda\mu' + \lambda \mu /3+ P'(\bar\rho ) \bar\rho  \right)   \right] \varphi \\ \quad= - \xi_1 \left[ \left(\lambda \mu'  + \lambda \mu /3 \right)\psi' + P'(\bar\rho )(\bar\rho\psi)'\right]
 -\xi_1 \xi_2  \left[ \lambda \mu'  + \lambda \mu /3 + P'(\bar\rho ) \bar\rho   \right] \theta&\hbox{in }
(-b,\ell)\\
 -(\lambda \mu  \theta')' + \left[ \lambda^2 \bar\rho   + \lambda \mu  \abs{\xi}^2 + \xi_2^2 \left( \lambda \mu'  + \lambda \mu /3 +  P'(\bar\rho ) \bar\rho  \right)   \right] \theta \\ \quad= - \xi_2\left[ \left(\lambda \mu'  + \lambda \mu /3 \right)\psi' + P'(\bar\rho )(\bar\rho\psi)'\right]
 -\xi_1 \xi_2  \left[ \lambda \mu'  + \lambda \mu /3 +  P'(\bar\rho ) \bar\rho   \right] \varphi&\hbox{in }
(-b,\ell)\\
  -\left[\left(  4\lambda \mu /3 + \lambda \mu'     \right) \psi'\right]'- \bar\rho\left[h'(\bar\rho )( \bar\rho\psi)'    \right]'
+ \left( \lambda^2 \bar\rho  + \lambda \mu  \abs{\xi}^2  \right) \psi \\
\quad =  \left[  \left( \lambda \mu' + \lambda \mu /3   \right) \left( \xi_1 \varphi + \xi_2 \theta \right)  \right]' +\bar\rho\left[  P'(\bar\rho )  \left( \xi_1 \varphi + \xi_2 \theta \right)  \right]'&\hbox{in }
(-b,\ell)\\
\mu_+\lambda(\xi_1\psi_+-\varphi_+') = \mu_+\lambda(\xi_2\psi_+-\theta_+')=0&\hbox{at }
x_3=\ell
\\ {-(\lambda \mu'+ \lambda \mu/3 ) (\psi' + \xi_1 \varphi + \xi_2 \theta )
-P'(\bar\rho)\left((\bar\rho\psi)'+\bar\rho ( \xi_1 \varphi + \xi_2 \theta )\right)}&
\\\quad{-\lambda \mu \left(\psi' - \xi_1 \varphi - \xi_2 \theta    \right)}
  =(\rho_1g+\sigma_+\abs{\xi}^2) \psi &\hbox{at }
x_3=\ell
\\
\llbracket\varphi\rrbracket=\llbracket\theta\rrbracket=\llbracket\psi\rrbracket=\llbracket\mu\lambda(\xi_1\psi-\varphi') \rrbracket=\llbracket\mu\lambda(\xi_2\psi-\theta') \rrbracket=0&\hbox{at }
x_3=0
\\\jump{(\lambda \mu'+ \lambda \mu/3 ) (\psi' + \xi_1 \varphi + \xi_2 \theta )}
+ \jump{P'(\bar\rho)\left((\bar\rho\psi)'+\bar\rho ( \xi_1 \varphi + \xi_2 \theta )\right)}&
\\\quad+  \jump{\lambda \mu \left(\psi' - \xi_1 \varphi - \xi_2 \theta    \right)  }
  =-(\rj g-\sigma_-\abs{\xi}^2) \psi &\hbox{at }
x_3=0
\\\varphi={\theta}= \psi =0&\hbox{at }
x_3=-b.
\end{cases}
\end{equation}

We can reduce the complexity of the problem by removing the component $\theta$.  To this end, note that if $\varphi,\theta,\psi$ solve the equations \eqref{linear3} for $\xi\in \Rn{2}$ and $\lambda>0$, then for any rotation operator $R\in SO(2)$,  $(\tilde{\varphi},\tilde{\theta}) := R (\varphi,\theta)$ solve the same equations for $\tilde{\xi} := R\xi$ with $\psi, \lambda$ unchanged.  So, by choosing an appropriate rotation, we may assume without loss of generality that $\xi_2=0$ and $\xi_1= \abs{\xi} \ge 0$.  In this setting $\theta$ solves
\begin{equation}
\begin{cases}
 -(\lambda \mu \theta')' + (\lambda^2 \bar\rho + \lambda \mu \abs{\xi}^2) \theta =0 \\
 \theta(-b) = \theta'(\ell) = 0 \\
 \jump{\theta} = \jump{\lambda \mu \theta'} =0,
\end{cases}
\end{equation}
which implies that $\theta=0$ since we assume $\lambda >0$.  Then the equations \eqref{linear3} are reduced to the equations for $\varphi,\psi$:
\begin{equation}\label{linear4}
\begin{cases}
-\left(\lambda \mu  \varphi'\right)' + \left[ \lambda^2 \bar\rho   + \lambda \mu  \abs{\xi}^2 + \abs{\xi}^2 \left( \lambda\mu' + \lambda \mu /3+ P'(\bar\rho ) \bar\rho  \right)   \right] \varphi \\ \quad= - \abs{\xi} \left[ \left(\lambda \mu'  + \lambda \mu /3 \right)\psi' + P'(\bar\rho )(\bar\rho\psi)'\right]&\hbox{in }
(-b,\ell)\\
  -\left[\left(  4\lambda \mu /3 + \lambda \mu'     \right) \psi'\right]'- \bar\rho\left[h'(\bar\rho )( \bar\rho\psi)'    \right]'
+ \left( \lambda^2 \bar\rho  + \lambda \mu  \abs{\xi}^2  \right) \psi \\
\quad =  \left[  \left( \lambda \mu' + \lambda \mu /3   \right) \abs{\xi} \varphi  \right]' +\bar\rho\left[  P'(\bar\rho )  \abs{\xi} \varphi   \right]'&\hbox{in }
(-b,\ell)\\
\mu_+\lambda(\abs{\xi}\psi_+-\varphi_+') =0&\hbox{at }
x_3=\ell
\\ {- (\lambda \mu'+ \lambda \mu/3 ) (\psi' +\abs{\xi}\varphi )
-P'(\bar\rho)\left((\bar\rho\psi)'+\bar\rho \abs{\xi} \varphi \right)}&
\\\quad{-\lambda \mu \left(\psi' - \abs{\xi}\varphi  \right)}
  =(\rho_1g+\sigma_+\abs{\xi}^2) \psi &\hbox{at }
x_3=\ell
\\
\llbracket\varphi\rrbracket =\llbracket\psi\rrbracket=\llbracket\mu\lambda(\abs{\xi}\psi-\varphi') \rrbracket=0&\hbox{at }
x_3=0
\\\jump{(\lambda \mu'+ \lambda \mu/3 ) (\psi' + \abs{\xi}\varphi )}
+ \jump{P'(\bar\rho)\left((\bar\rho\psi)'+\bar\rho \abs{\xi} \varphi \right)}&
\\\quad+  \jump{\lambda \mu \left(\psi' - \abs{\xi} \varphi  \right)  }
  =-(\rj g-\sigma_-\abs{\xi}^2) \psi &\hbox{at }
x_3=0
\\\varphi = \psi =0&\hbox{at }
x_3=-b.
\end{cases}
\end{equation}

Solutions to \eqref{linear4} can be constructed in the same way as that for the compressible viscous internal wave problem in \cite{3GT2}, so we will outline the procedure with minor modifications and refer some of the proofs to \cite{3GT2}. 

It is not trivial at all  to construct solutions by utilizing variational methods since $\lambda$ appears both linearly and quadratically. In order to circumvent this problem and restore the ability to use variational methods, we artificially remove the linear dependence on $\lambda$ in \eqref{linear4} by introducing an arbitrary parameter $s>0$. This results in a family $(s>0)$ of modified problems:
\begin{equation}\label{linear5}
\begin{cases}
-\left(s \mu  \varphi'\right)' + \left[ \lambda^2 \bar\rho   + s \mu  \abs{\xi}^2 + \abs{\xi}^2 \left( s\mu' + s \mu /3+ P'(\bar\rho ) \bar\rho  \right)   \right] \varphi \\ \quad= - \abs{\xi} \left[ \left(s \mu'  + s \mu /3 \right)\psi' + P'(\bar\rho )(\bar\rho\psi)'\right]&\hbox{in }
(-b,\ell)\\
  -\left[\left(  4s \mu /3 + s \mu'     \right) \psi'\right]'- \bar\rho\left[h'(\bar\rho )( \bar\rho\psi)'    \right]'
+ \left( \lambda^2 \bar\rho  + s \mu  \abs{\xi}^2  \right) \psi \\
\quad =  \left[  \left( s \mu' + s \mu /3   \right) \abs{\xi} \varphi  \right]' +\bar\rho\left[  P'(\bar\rho )  \abs{\xi} \varphi   \right]'&\hbox{in }
(-b,\ell)\\
\mu_+s(\abs{\xi}\psi_+-\varphi_+') =0&\hbox{at }
x_3=\ell
\\ {-(s \mu'+ s \mu/3 ) (\psi' +\abs{\xi}\varphi )
-P'(\bar\rho)\left((\bar\rho\psi)'+\bar\rho \abs{\xi} \varphi \right)}&
\\\quad{-s \mu \left(\psi' - \abs{\xi}\varphi  \right)}
  =(\rho_1g+\sigma_+\abs{\xi}^2) \psi &\hbox{at }
x_3=\ell
\\
\llbracket\varphi\rrbracket =\llbracket\psi\rrbracket=\llbracket\mu s(\abs{\xi}\psi-\varphi') \rrbracket=0&\hbox{at }
x_3=0
\\\jump{(s \mu'+ s \mu/3 ) (\psi' + \abs{\xi}\varphi )}
+ \jump{P'(\bar\rho)\left((\bar\rho\psi)'+\bar\rho \abs{\xi} \varphi \right)}&
\\\quad+  \jump{s \mu \left(\psi' - \abs{\xi} \varphi  \right)  }
  =-(\rj g-\sigma_-\abs{\xi}^2) \psi &\hbox{at }
x_3=0
\\\varphi = \psi =0&\hbox{at }
x_3=-b.
\end{cases}
\end{equation}
A solution to the modified problem \eqref{linear5} with $\lambda = s$ corresponds to a solution to the original problem \eqref{linear4}.
Note that for any fixed $s>0$ and $\xi$, \eqref{linear5} is a standard eigenvalue problem for $-\lambda^2$, which has a natural variational structure that allows us to use variational methods to construct solutions. In order to understand $\lambda$ in a variational framework, we consider the energy functional
\begin{equation}\label{E_def}
E(\varphi,\psi;s) = E_0(\varphi,\psi) + s E_1(\varphi,\psi)
\end{equation}
with
\begin{equation}\label{E0_def}
E_0(\varphi,\psi) = \frac{\sigma_- \abs{\xi}^2-\rj g}{2}(\psi(0))^2+\frac{\sigma_+ \abs{\xi}^2+\rho_1 g}{2}(\psi(\ell))^2
+ \hal \int_{-b}^\ell h'(\bar\rho) ((\bar\rho\psi)' + \bar\rho\abs{\xi} \varphi)^2,
\end{equation}
\begin{equation}\label{E1_def}
E_1(\varphi,\psi)= \hal \int_{-b}^\ell  \mu \left(  (\varphi' - \abs{\xi} \psi)^2 + (\psi'- \abs{\xi}\varphi)^2 + \frac{1}{3}(\psi' + \abs{\xi} \varphi)^2\right) + \mu' (\psi' + \abs{\xi} \varphi)^2,
\end{equation}
and
\begin{equation}\label{J_def}
 J(\varphi,\psi) = \hal \int_{-{b}}^\ell \bar\rho  (\varphi^2 + \psi^2),
\end{equation}
which are both well-defined on the space ${}_0H^1((-b,\ell)) \times {}_0H^1((-b,\ell))$ where
\begin{equation}
 \H((-b,\ell)) = \{ \phi \in H^1((-b,\ell)) \; \vert \;  \phi_- = 0 \text{ at } x_3=-b\}.
\end{equation}
Note that functions in this space automatically satisfy the condition $\jump{\phi}=0$ at $x_3=0$.  Consider the admissible set
\begin{equation}
 \mathfrak{S} = \{ (\varphi,\psi)\in {}_0H^1((-b,\ell)) \times {}_0H^1((-b,\ell)) \;\vert\;  J(\varphi,\psi)=1  \}.
\end{equation}
Notice that $E_0(\varphi,\psi)$ is not positive definite for $\jump{\bar\rho} >0$. 


The first proposition asserts that a minimizer of $E$ in \eqref{E_def} over  $\mathfrak{S}$ exists and the minimizer solves \eqref{linear5}.

\begin{prop} Let $\xi$ and $s>0$ be fixed. Then the following hold:
\begin{enumerate}
\item $E$ achieves its infimum over $\mathfrak{S}$.
\item The minimizers are smooth when restricted to $(-b,0)$ or $(0,\ell)$ and solve the equations \eqref{linear5} with $\lambda^2$ given by
\begin{equation}\label{mu_def}
-\lambda^2=\alpha(s) := \inf_{(\varphi,\psi)\in \mathfrak{S}}E(\varphi,\psi;s).
\end{equation}
\end{enumerate}
\end{prop}

\begin{proof}
A completion of the square and the fact that $\bar\rho$ solves \eqref{steady} allow us to write
\begin{equation}
 h'(\bar\rho) \left((\bar\rho\psi)' + \bar\rho\abs{\xi} \varphi\right)^2=P'(\bar\rho)\bar\rho \left(\psi' + \abs{\xi} \varphi\right)^2\\-2 g \bar\rho \abs{\xi} \psi \varphi-2 g  \bar\rho\psi ' \psi  -g\bar\rho'  \psi^2.
\end{equation}
We employ an integration by parts to see that
\begin{equation}\label{n_i_1}
  \int_{-b}^\ell-2 g  \bar\rho\psi ' \psi  -g\bar\rho'  \psi^2=-\int_{-b}^\ell  g\left( \bar\rho\psi^2\right)'
= -{\rho_1 g (\psi(\ell))^2} +{ \rj g (\psi(0))^2}.
\end{equation}
We then obtain another expression for $E_0(\varphi,\psi)$:
\begin{multline}\label{n_i_111}
E_0(\varphi,\psi) = \frac{\sigma_- \abs{\xi}^2 }{2}(\psi(0))^2+\frac{\sigma_+ \abs{\xi}^2}{2}(\psi(\ell))^2
 + \hal \int_{-b}^\ell P'(\bar\rho)\bar\rho \left(\psi' + \abs{\xi} \varphi\right)^2- 2g \bar\rho \abs{\xi} \psi \varphi.
\end{multline}
Notice that by further employing the identity $-2ab=(a-b)^2-(a^2+b^2)$ and the constraint on $J(\varphi,\psi)$, we see from \eqref{n_i_111} that
\begin{equation}
E(\varphi,\psi;s)\ge E_0(\varphi,\psi) \ge -2g\abs{\xi}\int_{-b}^\ell   \bar\rho  \psi \varphi \ge -g\abs{\xi}
\end{equation}
for any $(\varphi,\psi)\in \mathfrak{S}$. This shows that $E$ is bounded below on $\mathfrak{S}$.
The results thus follow from  standard compactness arguments, the variational principle for Euler-Langrange equations, and a bootstrap argument for smoothness.  For more details, we refer to Propositions 3.1 and 3.2 of \cite{3GT2}.
\end{proof}

In order to construct the growing mode solution  to the original problem \eqref{linear} we first need to ensure the negativity of the infimum \eqref{mu_def}. For $\sigma_-\ge\sigma_c$, we always have that $ \sigma_- \abs{\xi}^2-\rj g \ge 0$ for any nonzero frequency $\xi\in L_1^{-1}\mathbb{Z}\times L_2^{-1}\mathbb{Z}$, which implies $E(\varphi,\psi;s)\ge0$ because of \eqref{E0_def}. This then means that $\alpha(s)\ge0$, which suggests that no growing mode solution to \eqref{linear} can be constructed when $\sigma_- \ge \sigma_c$, and in turn indicates that the system is linearly stable. In fact, in \cite{JTW_GWP} we have established the nonlinear stability of the compressible viscous surface-internal wave problem for the case $\sigma_->\sigma_c$.  However, when $0\le\sigma_-<\sigma_c$, for $0<|\xi|< \sqrt{\rj g/\sigma_-}$ (it is interpreted that when $\sigma_-=0$ this means $0<|\xi|< \infty$), $ \sigma_- \abs{\xi}^2-\rj g<0$, and then it is possible for $E(\varphi,\psi;s)$ to be negative. We denote this critical frequency by $ |\xi|_c$:
\begin{equation}
 |\xi|_c:=\sqrt{\frac{\rj g}{\sigma_-}}.
\end{equation}

\begin{lem} Let $0<|\xi|<  |\xi|_c$. 
Then there exists $s_0>0$ depending on $\sigma_\pm,g,\bar\rho,P, b,\ell,\mu_\pm,\mu'_\pm,|\xi|$ such that for $0<s\leq s_0$ it holds that $\alpha(s)<0$.
\end{lem}

\begin{proof}
Since $E$ and $J$ have the same homogeneity, we may reduce to constructing any pair $(\varphi,\psi)\in {}_0H^1((-b,\ell)) \times {}_0H^1((-b,\ell))$ such that $E(\varphi,\psi;s) <0$.  We will take $\psi$ with $\psi(\ell)=0$ and $\varphi = -\psi'/\abs{\xi}$, and we then further reduce to constructing any $\psi \in H_0^2((-b,\ell))$ such that
\begin{equation}
\tilde{E}(\psi;s) := E(-\psi'/\abs{\xi},\psi;s) = \frac{\sigma_- \abs{\xi}^2-\rj g}{2}(\psi(0))^2
+   \hal \int_{-b}^\ell P'(\bar\rho)\bar\rho \psi^2 + sE_1(-\psi'/\abs{\xi},\psi) <0.
\end{equation}

For $\alpha \ge 5$ we define the test function $\psi_\alpha \in H_0^2((-b,\ell))$ according to
\begin{equation}
 \psi_\alpha(x_3) =
\begin{cases}
  \left(1-\frac{x_3^2}{\ell^2} \right)^{\alpha/2}, & x_3 \in [0,\ell) \\
  \left(1-\frac{x_3^2}{b^2} \right)^{\alpha/2}, &  x_3\in (-b,0).
\end{cases}
\end{equation}
Simple calculations then show that
\begin{equation}
 \int_{-m}^\ell (\psi_\alpha)^2 = \frac{\sqrt{\pi}(b+\ell) \Gamma(\alpha +1 )}{2 \Gamma(\alpha +3/2)} = o_\alpha(1),
\end{equation}
where $o_\alpha(1)$ is a quantity that vanishes as $\alpha \rightarrow \infty$. We thus find that
\begin{equation}
\tilde{E}(\psi_\alpha;s) = \frac{\sigma_- \abs{\xi}^2-\rj g}{2}+o_\alpha(1)
+ s C
\end{equation}
for the constant $C=E_1(-\psi_\alpha'/\abs{\xi},\psi_\alpha)$, which depends on
$\alpha,\bar\rho, b,\ell,\mu_\pm,\mu'_\pm,|\xi|$. Since $ \sigma_- \abs{\xi}^2-\rj g <0$, we may then fix $\alpha$ sufficiently large so that the first two terms sum to something strictly negative. Then
there exists $s_0>0$ depending on $\sigma_\pm,g,\bar\rho,P, b,\ell,\mu_\pm,\mu'_\pm,|\xi|$ so that for
$s\le s_0$ it holds that
$\tilde{E}(\psi_\alpha;s)<0$.  Thus  $\alpha(s)<0$ for $s \le s_0$.
\end{proof}

\begin{remark}
For a  minimizer $(\varphi,\psi) \in \mathfrak{S}$ we have
\begin{equation}
 \frac{\sigma_- \abs{\xi}^2 - g{\jump{\bar\rho}}}{2} (\psi(0))^2\le \al(s) <0,
\end{equation}
which in particular requires that $\psi(0)\neq 0$ and $\abs{\xi}^2 < g{\jump{\bar\rho}}/\sigma_-$.
\end{remark}


As in Proposition 3.6 of \cite{3GT2}, one can prove that $\al=\al(s)$ is continuous and strictly increasing and that there exists $s_*\in(0,\infty)$ so that
\begin{equation}
 \mathcal{S} = \al^{-1}((-\infty,0)) = (0,s_*).
\end{equation}
Arguing as in Theorem 3.8 of \cite{3GT2}, we deduce the following.

\begin{lem} For each fixed $0< \abs{\xi} <  |\xi|_c$ there exists a unique $s \in \mathcal{S}$ so that $\lambda(|\xi|,s)=\sqrt{-\al(s)}>0$ and
\begin{equation}\label{inv_0}
 s= \lambda(|\xi|,s).
\end{equation}
\end{lem}

Hence, we may now think of $s = s(\abs{\xi})$ and we may also write $\lambda = \lambda(\abs{\xi})$ from now on. In conclusion, we now have the existence of solutions to the system \eqref{linear3}.

\begin{prop}\label{w_soln_2}
For $\xi\in\Rn{2}$ so that $0 < \abs{\xi} <  \abs{\xi}_c$ there exists a solution $\varphi = \varphi(\xi,x_3)$, $\theta = \theta(\xi,x_3)$, $\psi = \psi(\xi,x_3)$, and $\lambda = \lambda(\abs{\xi})>0$  to \eqref{linear3} so that  $\psi(\xi,0)\neq 0$.
The solutions are smooth when restricted to $(-b,0)$ or $(0,\ell)$, and they are equivariant in $\xi$ in the sense that if $R\in SO(2)$ is a rotation operator, then
\begin{equation}\label{w_s_0}
\begin{pmatrix}
\varphi(R \xi,x_3) \\ \theta(R \xi,x_3) \\ \psi(R \xi,x_3)
\end{pmatrix}
=
\begin{pmatrix}
R_{11} & R_{12} & 0 \\
R_{21} & R_{22} & 0 \\
0      & 0      & 1
\end{pmatrix}
\begin{pmatrix}
\varphi(\xi,x_3) \\ \theta(\xi,x_3) \\ \psi(\xi,x_3)
\end{pmatrix}.
\end{equation}
\end{prop}
\begin{proof}
We may find a rotation operator $R \in SO(2)$ so that $R \xi = (\abs{\xi},0)$.   For $\lambda = \lambda(\abs{\xi})$ given in \eqref{inv_0}, we define $(\varphi(\xi,x_3),\theta(\xi,x_3)) = R^{-1} (\varphi(\abs{\xi},x_3),0)$ and $\psi(\xi,x_3) = \psi(\abs{\xi},x_3)$, where the functions $\varphi(\abs{\xi},x_3)$ and $\psi(\abs{\xi},x_3)$ are the minimizer of \eqref{mu_def}, which solves the equations \eqref{linear5}, with $s=\lambda$.  This gives a solution to \eqref{linear3}.  The equivariance in $\xi$ follows from the definition.
\end{proof}

To obtain a largest growth rate, we next show the boundedness of $\lambda(\abs{\xi})$.

\begin{prop}\label{prop3.5} For any $0< \abs{\xi}<  \abs{\xi}_c$, $\lambda(\abs{\xi})$ satisfies the bound
\begin{equation}\label{bound}
\lambda(|\xi|)\le \frac{b g \rj }{\mu_-}.
 \end{equation}
\end{prop}

\begin{proof} For given $|\xi|\in(0,|\xi|_c)$, let $(\varphi, \psi)$ be the corresponding minimizer of $E$ so that
$-\lambda^2=E(\varphi, \psi; \lambda)$. From \eqref{E_def} and \eqref{E0_def}, we have $E=E_0+\lambda E_1$ and $E_0\geq -{ \rj g}(\psi(0))^2/2$. Hence,
\begin{equation}\label{3.27}
\lambda E_1 \leq -E_0 \leq \frac{ \rj g}{2}(\psi(0))^2
\end{equation}
On the other hand, since $\psi(-b)=0$, $\psi(0)=\int_{-b}^0 \psi' dx_3 \leq \sqrt{b}(\int_{-b}^0 (\psi')^2 dx_3 )^{1/2}$, and thus
\begin{equation}
(\psi(0))^2\leq b \int_{-b}^0 (\psi')^2 dx_3 =  b \int_{-b}^0 \left(\frac{\psi'-|\xi|\varphi}{2} + \frac{\psi'+|\xi|\varphi}{2} \right)^2 dx_3.
\end{equation}
By further using the inequality: $(A+B)^2\leq 4(A^2+B^2/3)$ for all $A,B\in \mathbb{R}$, we have
\begin{equation}
\begin{split}
(\psi(0))^2&\leq b  \int_{-b}^0 (\psi'-|\xi|\varphi)^2 +\frac13(\psi'+|\xi|\varphi)^2 dx_3\\
&= \frac{2b}{\mu_-}\,\frac12  \int_{-b}^0 \mu_-\left( (\psi'-|\xi|\varphi)^2 +\frac13(\psi'+|\xi|\varphi)^2 \right)dx_3 \leq \frac{2b}{\mu_-} E_1.
\end{split}
\end{equation}
Combining this with \eqref{3.27}, we deduce \eqref{bound}.
\end{proof}

Proposition \ref{prop3.5} then allows us to define
\begin{equation} \label{Lambda}
0<\Lambda:= \sup_{0<|\xi|< |\xi|_c} {\lambda(|\xi|)}<\infty .
\end{equation}
For $\sigma_->0$, only a finite number of spatial frequencies $\xi\in (L_1^{-1}\mathbb{Z})\times (L_2^{-1}\mathbb{Z})$ satisfy $|\xi|<|\xi|_c$, so the the largest growth rate $\Lambda$  must be achieved when $0<\sigma_- <\sigma_c$. For $\sigma_-=0$ it is not clear whether $\Lambda$ is achieved.  However, we can achieve a growth rate that is arbitrarily close to $\Lambda$, and so in particular $\Lambda_\ast$ is achieved, where
\begin{equation} \label{littlelambda}
0<\Lambda/2<\Lambda_\ast\le \Lambda.
\end{equation}

We may now construct a growing mode solution to the linearized problem \eqref{linear}.

\begin{thm}\label{growingmode}
Let $\Lambda$ be defined by \eqref{Lambda} and $\Lambda_\ast$ be defined by \eqref{littlelambda}.  Then the following hold.
\begin{enumerate}
 \item Let $0< \sigma_- < \sigma_c$.  Then there is a growing mode solution to \eqref{linear} so that
\begin{equation}\label{gro1}
\|q(t)\|_{k}=e^{\Lambda t}\|q(0)\|_{k},\;\|u(t)\|_{k}=e^{\Lambda t}\|u(0)\|_{k}, \;\|\eta(t)\|_{k}=e^{\Lambda t}\|\eta(0)\|_{k}
\end{equation}
for any $k\geq0$.
 \item  Let $\sigma_-=0$. Then there is a growing mode solution to \eqref{linear} so that
\begin{equation}\label{gro2}
\|q(t)\|_{k}=e^{\Lambda_\ast t}\|q(0)\|_{k},\; \|u(t)\|_{k}=e^{\Lambda_\ast t}\|u(0)\|_{k},\; \|\eta(t)\|_{k}=e^{\Lambda_\ast t}\|\eta(0)\|_{k}
\end{equation}
for any $k\geq0$.
\end{enumerate}
\end{thm}

\begin{proof}
Let $|\xi|>0$ be so that $\lambda(|\xi|)=\Lambda$ for $\sigma_->0$ or $\lambda(|\xi|)=\Lambda_\ast$ for $\sigma_-=0$. Let $\varphi = \varphi(\xi,x_3)$, $\theta = \theta(\xi,x_3)$, $\psi = \psi(\xi,x_3)$ be the solution to \eqref{linear3} with $\lambda(|\xi|)$ as stated in Proposition \ref{w_soln_2}. We then define $q$, $u$, and $\eta$ according to \eqref{ansatz}, \eqref{ansatz1}, and \eqref{ansatz2}. Then we have that $q \in {H}^k(\Omega)$, $u\in {}_0H^1(\Omega)\cap {H}^k(\Omega)$, and $\eta\in H^k(\Sigma)$  for any $k \ge 0$ and $(q,u,\eta)$ solve the linearized problem \eqref{linear}.  Moreover, $q,u,\eta$ satisfy \eqref{gro1} or \eqref{gro2}.
\end{proof}

\section{Growth of solutions to the linear inhomogeneous equations}\label{growth}

In this section, we will show that $\Lambda$ defined by \eqref{Lambda} is the sharp growth rate of arbitrary solutions to the linearized problem \eqref{linear}. Since the spectrum of the linear operator is  complicated, it is hard to obtain the largest growth rate of the solution operator in ``$L^2\rightarrow L^2$'' in the usual way. Instead, motivated by \cite{3GT2}, we can use careful energy estimates to show that $e^{\Lambda t}$ is the sharp growth rate in a slightly weaker sense, say, for instance ``$H^2\rightarrow L^2$''. However, this will be done for strong solutions to the problem, and it may be difficult to apply directly to the nonlinear problem due to the issue of compatibility conditions of the initial and boundary data since the problem is defined in a domain with boundary. We overcome this obstacle by proving the estimates for the growth in time of arbitrary solutions to the linear inhomogeneous equations:
\begin{equation}\label{linear ho}
\begin{cases}
\partial_t \q +\diverge ( \bar{\rho}   u)=G^1 & \text{in }
\Omega  \\
 \bar{\rho} \partial_t    u   + \bar{\rho}\nabla \left(h'(\bar{\rho})\q\right)   -\diverge \S(u) =G^2 & \text{in }
\Omega  \\
\partial_t \eta = u_3+G^4 &
\text{on }\Sigma  \\
(  P'(\bar\rho)\q I- \S(  u))e_3  = (\rho_1  g \eta_+ -\sigma_+ \Delta_\ast \eta_+ ) e_3 +G_+^3
 & \text{on } \Sigma_+
 \\ \jump{P'(\bar\rho)\q I- \S(u)}e_3
=(\rj g\eta_- +\sigma_- \Delta_\ast \eta_-)e_3-G_-^3&\hbox{on }\Sigma_-
\\\jump{u}=0 &\hbox{on }\Sigma_-
\\ u_-=0 &\hbox{on }\Sigma_b,
\end{cases}%
\end{equation}
where $G^i$'s are given functions.

It will be convenient to work with a second-order formulation of the equations \eqref{linear ho}. To arrive at this, we differentiate the second equation in time and eliminate the $\q$ and
$\eta$ terms using the other equations. The resulting equations for $u$ read as
\begin{equation}\label{second_order}
\begin{cases}
 \bar{\rho} \p_{tt}    u  -\bar{\rho}\nabla \left(h'(\bar{\rho})\diverge ( \bar{\rho}   u)\right)   -\diverge \S(\p_t u) =\mathfrak{F} & \text{in }
\Omega  \\
( -P'(\bar\rho)\diverge ( \bar{\rho}   u) I- \S( \p_t u))e_3  = (\rho_1  g   u_{3} -\sigma_+ \Delta_\ast   u_{3} ) e_3 +\mathfrak{G}_+
 & \text{on } \Sigma_+
 \\  \jump{-P'(\bar\rho)\diverge ( \bar{\rho}   u) I- \S(\p_t u)}e_3
=(\rj g  u_{3} +\sigma_- \Delta_\ast u_{3})e_3-\mathfrak{G}_-&\hbox{on }\Sigma_-
\\\jump{\p_t u}=0 &\hbox{on }\Sigma_-
\\\p_t u_-=0 &\hbox{on }\Sigma_b,
\end{cases}%
\end{equation}
where
\begin{equation}\label{F_def}
 \mathfrak{F} :=-\bar{\rho}\nabla \left(h'(\bar{\rho})G^1\right)+\p_t G^2,
\end{equation}
\begin{equation}\label{G+_def}
 \mathfrak{G}_+ :=-P_+'(\bar\rho_+) G^1 e_3-\p_t G_+^3+(\rho_1  g  G^4_+ -\sigma_+ \Delta_\ast   G^4_+ ) e_3,
\end{equation}
and
\begin{equation}\label{G-_def}
- \mathfrak{G}_- :=-\jump{P_+'(\bar\rho_+) G^1} e_3-\p_t G_-^3+(\rj  g  G^4_- +\sigma_- \Delta_\ast   G^4_- ) e_3.
\end{equation}

Our first result gives an energy and its evolution equation for solutions to the second-order problem \eqref{second_order}.
\begin{lem}\label{lin_en_evolve le}
Let $u$ solve \eqref{second_order}. Then
\begin{align}\label{energyidentity}
&\nonumber\frac{d}{dt}\left(\int_\Omega \frac{\bar\rho}{2} \abs{\dt u}^2+ \frac{h'(\bar\rho)}{2}\abs{ \diverge{(\bar\rho u)}}^2 + \int_\Sigma \frac{\sigma}{2}\abs{\nab_\ast u_3}^2+\int_{\Sigma_-} \frac{g \rho_1}{2} \abs{u_3}^2 +\int_{\Sigma_-} -\frac{g \jump{\rho}}{2} \abs{u_3}^2\right)
 \\&\quad+ \int_\Omega \frac{\mu}{2} \abs{D \dt u + D \dt u^T - \frac{2}{3} (\diverge{\dt u})I  }^2+\int_\Omega \mu' \abs{\diverge{\dt u}}^2
= \int_\Omega \mathfrak{F}\cdot\dt u-\int_\Sigma \mathfrak{G}\cdot \dt u.
\end{align}
\end{lem}
\begin{proof}
We multiply the first equation of $\eqref{second_order}$ by $\partial_{t} u$ and  then integrate by parts over $\Omega$. By using the boundary conditions in $\eqref{second_order}$, we obtain \eqref{energyidentity}.
\end{proof}

The variational characterization of $\Lambda$, which was given by \eqref{Lambda}, gives rise to the next result.
\begin{lem}\label{lin_en_bound}
Let $u\in {}_0H^1(\Omega)\cap {H}^2(\Omega)$. Then we have the inequality
\begin{align}
&\nonumber\int_\Omega \frac{h'(\bar\rho)}{2}\abs{ \diverge{(\bar\rho u)}}^2 + \int_\Sigma \frac{\sigma}{2}\abs{\nab_\ast u_3}^2+\int_{\Sigma_-} \frac{g \rho_1}{2} \abs{u_3}^2 +\int_{\Sigma_-} -\frac{g \jump{\rho}}{2} \abs{u_3}^2
\\ &\quad\ge -\frac{\Lambda^2}{2}\int_\Omega \bar\rho \abs{u}^2
- \frac{\Lambda}{2} \int_\Omega \frac{\mu}{2} \abs{Du + Du^T - \frac{2}{3}(\diverge{u})I  }^2 + \mu' \abs{\diverge{u}}^2 .
\end{align}
 \end{lem}
 \begin{proof}
We take the horizontal Fourier transform to see that
\begin{align}\label{sum}
&\nonumber4\pi^2\left(\int_\Omega \frac{h'(\bar\rho)}{2}\abs{ \diverge{(\bar\rho u)}}^2+ \int_\Sigma \frac{\sigma}{2}\abs{\nab_\ast u_3}^2+\int_{\Sigma_-} \frac{g \rho_1}{2} \abs{u_3}^2 +\int_{\Sigma_-} -\frac{g \jump{\rho}}{2} \abs{u_3}^2\right)
\\ &\nonumber\quad=\sum_{\xi\in L_1^{-1}\mathbb{Z}\times
L_2^{-1}\mathbb{Z}}\left\{\int_{-b}^\ell \frac{h'(\bar\rho) }{2}\abs{ i\xi_1 \bar\rho\hat{u}_1 + i\xi_2\bar\rho \hat{u}_2 + \partial_3 (\bar\rho\hat{u}_3)}^2 dx_3\right.
\\ &\qquad\left.+  \frac{\sigma_+|\xi|^2
+\rho_1g }{2}|\hat{u}_3(\ell)|^2+  \frac{\sigma_-|\xi|^2
-\rj g}{2} |\hat{u}_3(0)|^2 \right\}.
\end{align}

For $\xi=0$, the term in the sum of \eqref{sum} is
\begin{equation}
\frac{  \rho_1g}{2}\abs{\hat{u}_3(\ell)}^2-\frac{g \jump{\rho} }{2} \abs{\hat{u}_3(0)}^2
+\int_{-b}^\ell \frac{h'(\bar\rho) }{2}\abs{ \partial_3 (\bar\rho\hat{v}_3)}^2 dx_3.
\end{equation}
We expand the derivative term in the integral and integrate by parts to get
\begin{align}
\nonumber\int_{-b}^\ell \frac{h'(\bar\rho) }{2}\abs{ \partial_3 (\bar\rho\hat{u}_3)}^2 dx_3& =\int_{-b}^\ell \frac{h'(\bar\rho) }{2}\abs{ \partial_3 \bar\rho\hat{u}_3+\bar\rho\partial_3 \hat{u}_3 }^2 dx_3  \\ \nonumber&= \int_{-b}^\ell \hal P'(\bar\rho) \bar\rho \abs{\partial_3 \hat{u}_3}^2 -g\bar\rho \partial_3 \hat{u}_3 \hat{u}_3-\hal g\p_3\bar\rho\abs{\hat{u}_3}^2 dx_3
\\ &= \int_{-b}^\ell \hal P'(\bar\rho) \bar\rho \abs{\partial_3 \hat{u}_3}^2 -\frac{  \rho_1g}{2}\abs{\hat{u}_3(\ell)}^2+\frac{g \jump{\bar\rho} }{2} \abs{\hat{u}_3(0)}^2.
\end{align}
This implies
\begin{equation}\label{sum0}
\frac{  \rho_1g}{2}\abs{\hat{u}_3(\ell)}^2-\frac{g \jump{\bar\rho} }{2} \abs{\hat{u}_3(0)}^2
+\int_{-b}^\ell \frac{h'(\bar\rho) }{2}\abs{ \partial_3 (\bar\rho\hat{u}_3)}^2 dx_3 =\hal \int_{-b}^\ell P'(\bar\rho) \bar\rho \abs{\partial_3 \hat{u}_3}^2\ge 0.
\end{equation}

Consider now the sum of \eqref{sum} for fixed $\xi\neq 0$, writing $\varphi(x_3)= i \hat{u}_1(\xi,x_3)$, $\theta(x_3)= i \hat{u}_2(\xi,x_3)$, $\psi(x_3) = \hat{u}_3(\xi,x_3)$.  That is, define
\begin{align}
\nonumber Z(\varphi,\theta,\psi;\xi) =& \frac{\sigma_+ \abs{\xi}^2+\rho_1 g}{2}(\psi(\ell))^2+ \frac{\sigma_- \abs{\xi}^2-\rj g}{2}(\psi(0))^2\\&+\int_{-b}^\ell \frac{h'(\bar\rho) }{2}\abs{  \xi_1 \bar\rho\varphi+  \xi_2\bar\rho \theta + (\bar\rho\psi)'}^2 dx_3,
\end{align}
where $' = \partial_3$. The expression for $Z$ is invariant under simultaneous rotations of $\xi$ and $(\varphi,\theta)$, so without loss of generality we may assume that $\xi = (\abs{\xi},0)$ with $\abs{\xi} > 0$ and $\theta =0$.   If $\sigma_->0$ then we assume that $\abs{\xi} < \abs{\xi}_c$ as well.  Then, using \eqref{E_def} with $ s=\lambda(\abs{\xi})$, we may rewrite
\begin{align}
\nonumber Z(\varphi,\theta,\psi;\xi) =& E(\varphi,\psi;\lambda(\abs{\xi}))- \frac{\lambda(\abs{\xi})}{2} \int_{-b}^\ell \mu'  \abs{\psi'+\abs{\xi} \varphi}^2 \\&
-\frac{\lambda(\abs{\xi})}{2} \int_{-b}^\ell \mu  \left(  \abs{\varphi' - \abs{\xi} \psi}^2 + \abs{\psi' -\abs{\xi} \varphi}^2 + \frac{1}{3} \abs{\psi' + \abs{\xi} \varphi}^2\right)
\end{align}
and hence
\begin{align}\label{ininin}
\nonumber Z(\varphi,\theta,\psi;\xi) \ge &-\frac{\Lambda^2}{2} \int_{-b}^\ell \bar\rho  (\abs{\varphi}^2  +\abs{\psi}^2)-\frac{\Lambda}{2} \int_{-b}^\ell \mu' \abs{\psi'+\abs{\xi} \varphi}^2 \\
&-\frac{\Lambda}{2} \int_{-b}^\ell \mu  \left(  \abs{\varphi' - \abs{\xi} \psi}^2 + \abs{\psi' -\abs{\xi} \varphi}^2 + \frac{1}{3} \abs{\psi' + \abs{\xi} \varphi}^2\right).
\end{align}
Here in the inequality above we have used the following variational characterization for $\Lambda$, which follows directly from the definitions \eqref{mu_def} and \eqref{Lambda},
\begin{equation}
E(\varphi,\psi;\lambda(\abs{\xi}))\ge -\lambda(\abs{\xi})^2
J(\varphi,\psi)\ge -\frac{\Lambda^2}{2} \int_{-b}^\ell \bar\rho  (\abs{\varphi}^2  +\abs{\psi}^2).
\end{equation}
For $\abs{\xi} \ge \xi_c$ the expression for $Z$ is non-negative, so the inequality \eqref{ininin} holds trivially, and so we deduce that it holds for all $\abs{\xi}>0$.

Translating the inequality \eqref{ininin} back to the original notation for fixed $\xi$, we find
\begin{align}
\nonumber &\frac{\sigma_+|\xi|^2
+\rho_1g }{2}|\hat{u}_3(\ell)|^2+  \frac{\sigma_-|\xi|^2
-\rj g}{2} |\hat{u}_3(0)|^2
+ \int_{-b}^\ell \frac{h'(\bar\rho) }{2}\abs{ i\xi_1 \bar\rho\hat{u}_1 + i\xi_2\bar\rho \hat{u}_2 + \partial_3 (\bar\rho\hat{u}_3)}^2 dx_3 \\ &\quad
\ge- \frac{\Lambda^2}{2} \int_{-b}^\ell \bar\rho  \abs{\hat{u}}^2
- \frac{\Lambda}{2} \int_{-b}^\ell \mu'  \abs{i\xi_1 \hat{u}_1 + i\xi_2 \hat{u}_2 + \partial_3 \hat{u}_3 }^2 + \frac{\mu }{2} \abs{\hat{B}}^2,
\end{align}
where
\begin{equation}
 B = Du+ Du^T - \frac{2}{3}(\diverge{u})I.
\end{equation}
Taking sum of each side of this inequality over all $0\neq\xi\in L_1^{-1}\mathbb{Z}\times
L_2^{-1}\mathbb{Z}$, together with \eqref{sum0}, then proves the result.
\end{proof}

Now we can prove our main result of this section. We write
\begin{equation}
\mathcal{E}_{G}:=\norm{G^1 }_1^2 +\norm{\dt G^2 }_0^2 +\norm{\dt G^3 }_0^2 +\norm{ G^4 }_2^2.
\end{equation}
\begin{thm} \label{lineargrownth}
Let  $(q,u,\eta)$ solve \eqref{linear ho}. Then we have the following estimates for $t \ge 0$:
\begin{align}\label{result1}
\norm{u(t)}_{1}^2  &\ls
e^{2 \Lambda t} (\norm{ u(0)}_2^2 +\norm{\partial_t u(0)}_0^2)
 + \int_0^t  e^{2\Lambda
(t-s)} \sqrt{\mathcal{E}_{G}(s)}\norm{\dt u(s)}_{1}ds,
\\\label{result2}
\norm{\eta(t)}_{0} & \ls
e^{\Lambda t} (\norm{ u(0)}_2 +\norm{\partial_t u(0)}_0+\norm{\eta(0)}_0)+\int_0^t \norm{G^4(s)}_0ds
\\&\nonumber\quad+ \int_0^t\sqrt{\int_0^s  e^{2\Lambda
(s-\tau)}\sqrt{\mathcal{E}_{G}(\tau)}\norm{\dt u(\tau)}_{1}d\tau}ds,
\\\label{result3}
\norm{q(t)}_{0}  &\ls
e^{ \Lambda t} (\norm{ u(0)}_2 +\norm{\partial_t u(0)}_0+\norm{q(0)}_0)+\int_0^t \norm{G^1(s)}_0ds
\\&\nonumber\quad+ \int_0^t\sqrt{\int_0^s  e^{2\Lambda
(s-\tau)}\sqrt{\mathcal{E}_{G}(\tau)}\norm{\dt u(\tau)}_{1}d\tau}ds,
\end{align}
\end{thm}

\begin{proof}
Integrating the result of Lemma \ref{lin_en_evolve le} in time from $0$ to $t$, and then applying Lemma \ref{lin_en_bound}, we find that
\begin{align}
\nonumber & \int_\Omega \frac{\bar\rho}{2} \abs{\dt u}^2+ \int_0^t\int_\Omega \frac{\mu}{2} \abs{D \dt u + D \dt u^T - \frac{2}{3} (\diverge{\dt u})I  }^2+\int_\Omega \mu' \abs{\diverge{\dt u}}^2
\\\nonumber&\quad\le K_0+\int_0^t \int_\Omega \mathfrak{F}\cdot\dt u-\int_0^t\int_\Sigma \mathfrak{G}\cdot \dt u
\\\nonumber&\qquad-\left(\int_\Omega \frac{h'(\bar\rho)}{2}\abs{ \diverge{(\bar\rho u)}}^2 + \int_\Sigma \frac{\sigma}{2}\abs{\nab_\ast u_3}^2+\int_{\Sigma_-} \frac{g \rho_1}{2} \abs{u_3}^2 +\int_{\Sigma_-} -\frac{g \jump{\rho}}{2} \abs{u_3}^2\right)
\\\nonumber &\quad\le K_0+ \int_0^t\int_\Omega \mathfrak{F}\cdot\dt u-\int_0^t\int_\Sigma \mathfrak{G}\cdot \dt u
\\&\qquad+\frac{\Lambda^2}{2}\int_\Omega \bar\rho \abs{u}^2
+\frac{\Lambda}{2} \int_\Omega \frac{\mu}{2} \abs{Du + Du^T - \frac{2}{3}(\diverge{u})I  }^2 + \mu' \abs{\diverge{u}}^2,
\end{align}
where
\begin{align}
\nonumber K_0 =& \int_\Omega \frac{\bar\rho}{2} \abs{\dt u(0)}^2 + \frac{h'(\bar\rho)}{2}\abs{ \diverge{(\bar\rho u(0))}}^2 \\&+ \int_\Sigma \frac{\sigma}{2}\abs{\nab_\ast u_3(0)}^2+\int_{\Sigma_-} \frac{g \rho_1}{2} \abs{u_3(0)}^2 +\int_{\Sigma_-} -\frac{g \jump{\rho}}{2} \abs{u_3(0)}^2.
\end{align}
For notational simplicity we introduce the norms
\begin{equation}
\norm{u}_\star^2:=\int_\Omega\bar\rho |  u|^2\hbox{ and }\norm{u}_{\star\star}^2:=\int_\Omega \frac{\mu}{2} \abs{Du + Du^T - \frac{2}{3}(\diverge{u})I  }^2 + \mu' \abs{\diverge{u}}^2
\end{equation}
and the corresponding inner-products given by  $\langle \cdot,\cdot\rangle_{\star}$ and $\langle \cdot,\cdot\rangle_{\star\star}$, respectively.  We may then compactly rewrite the previous inequality as
\begin{equation}\label{j4}
\frac{1}{2}\norm{\partial_t  u(t)}_\star^2 + \int_0^t\norm{\partial_t  u(s)}_{\star\star}^2 ds \le K_0
+\frac{\Lambda^2}{2}\norm{ u(t) }_\star^2 +  \frac{\Lambda}{2}\norm{ u(t)}_{\star\star}^2 + \mathfrak{H}(t)
\end{equation}
where we have written
\begin{equation}
\mathfrak{H}(t)=\int_0^t \int_\Omega \mathfrak{F}\cdot\dt u-\int_0^t\int_\Sigma \mathfrak{G}\cdot \dt u.
\end{equation}

Integrating in time and using Cauchy's inequality, we may bound
\begin{equation}\label{j222}
\begin{split}
\Lambda\norm{u(t)}_{\star\star}^2
&=\Lambda\norm{u(0)}_{\star\star}^2+ \Lambda\int_0^t 2 \langle
u(s),\partial_t u(s) \rangle_{\star\star}\,ds
\\&\le
\Lambda\norm{u(0)}_{\star\star}^2+\int_0^t\norm{\partial_tu(s)}_{\star\star}^2\,ds
+\Lambda^2\int_0^t\norm{u(s)}_{\star\star}^2\,ds.
\end{split}
\end{equation}
On the other hand
\begin{equation}\label{j333}
\Lambda\partial_t\norm{u(t)}_{\star}^2=2 \Lambda \langle
u(t),\partial_t u(t) \rangle_{\star}\le \norm{\partial_tu(t)}_{\star}^2
+\Lambda^2\norm{u(t)}_{\star}^2.
\end{equation}
We may combine these two inequalities with \eqref{j4} to derive the differential inequality
\begin{equation}\label{j5}
\partial_t\norm{u(t)}_\star^2+\norm{u(t)}_{\star\star}^2\le K_1+2\Lambda \left( \norm{u(t)}_\star^2+\int_0^t\norm{u(s)}_{\star\star}^2 \,ds \right)+\frac{2}{\Lambda}\mathfrak{H}(t)
\end{equation}
for $K_1=2K_0/\Lambda+2\norm{u(0)}_{\star\star}^2$. An application of Gronwall's theorem then shows that
\begin{equation}\label{j6}
\norm{u(t)}_\star^2+\int_0^t\norm{u(s)}_{\star\star}^2
\le e^{2\Lambda t}\norm{u(0)}_\star^2+\frac{K_1}{2\Lambda}( e^{2\Lambda
t}-1)+\frac{2}{\Lambda}\int_0^t e^{2\Lambda
(t-s)}\mathfrak{H}(s)ds.
\end{equation}
Now plugging \eqref{j6} and \eqref{j222} into \eqref{j4}, we find that
\begin{align}
\nonumber\frac{1}{\Lambda}\norm{\partial_tu(t)}_\star^2+\norm{u(t)}_{\star\star}^2
&\le K_1+\Lambda \norm{u(t)}_\star^2+2\Lambda \int_0^t
\norm{u(s)}_{\star\star}^2\,ds+\frac{2}{\Lambda}\mathfrak{H}(t)
\\ &\le  e^{2\Lambda
t}(2\Lambda\norm{u(0)}_\star^2+K_1)+\frac{2}{\Lambda}\mathfrak{H}(t)+4\int_0^t e^{2\Lambda
(t-s)}\mathfrak{H}(s)ds.
\end{align}
Notice that
\begin{align}
\nonumber 4\int_0^t e^{2\Lambda
(t-s)}\mathfrak{H}(s)ds
&=-\frac{2}{\Lambda}\int_0^t \p_t\left(e^{2\Lambda
(t-s)}\right)\mathfrak{H}(s)ds\\&=-\frac{2}{\Lambda}\mathfrak{H}(t)+e^{2\Lambda
 t }\mathfrak{H}(0)+\frac{2}{\Lambda}\int_0^t e^{2\Lambda
(t-s)} \p_t \mathfrak{H}(s)ds
\end{align}
and
\begin{equation}
\mathfrak{H}(0)=0 \text{ and }\p_t \mathfrak{H} =\int_\Omega \mathfrak{F}\cdot\dt u-\int_\Sigma \mathfrak{G}\cdot \dt u.
\end{equation}
We then have,
\begin{equation}\label{j7}
\norm{u(t)}_{\star\star}^2
  \le  e^{2\Lambda
t}(2\Lambda\norm{ u(0)}_\star^2+K_1)+\frac{2}{\Lambda}\int_0^t  e^{2\Lambda
(t-s)} \left(\int_\Omega \mathfrak{F}(s)\cdot\dt u (s)-\int_\Sigma \mathfrak{G}(s)\cdot \dt u(s) \right) ds.
\end{equation}
By the trace theorem,
\begin{equation}
K_1\ls
\norm{ u(0)}_2^2+\norm{\partial_t u(0)}_0^2 .
\end{equation}
On the other hand,
\begin{equation}
\norm{\mathfrak{F}}_0^2+\norm{\mathfrak{G}}_0^2\ls \norm{G^1}_1^2+\norm{\dt G^2}_0^2+\norm{\dt G^3}_0^2+\norm{ G^4}_2^2= \mathcal{E}_{G}.
\end{equation}
So by Korn's inequality  (Proposition \ref{layer_korn}) and the trace theorem, \eqref{j7} implies \eqref{result1}.

Next we use the kinematic boundary condition $\dt \eta = u_3+G^4$ and the trace theorem to estimate
\begin{equation}\label{j8}
\norm{ \partial_t\eta(t)}_{0}
 \le  \norm{u_3}_{H^{0}(\Sigma)}+\norm{G^4}_{0}
 \lesssim  \norm{ u_3}_{1}+\norm{ G^4}_{0}.
\end{equation}
This and \eqref{result1} allow us to estimate
 \begin{align}\label{j9}
\nonumber\norm{ \eta(t)}_{0}
& \le  \norm{ \eta(0)}_{0} +\int_0^t \norm{\partial_t\eta(s)}_0 \,ds \lesssim  \norm{ \eta(0)}_{0}+\int_0^t \norm{G^4(s)}_0ds+\int_0^t \norm{u_3(s)}_{1}ds
\\&\nonumber\lesssim  \norm{ \eta(0)}_{0}+\int_0^t \norm{G^4(s)}_0ds+\int_0^t e^{ \Lambda s} (\norm{ u(0)}_2  +\norm{\partial_t u(0)}_0 )
\\  &\quad+ \int_0^t\sqrt{\int_0^s  e^{2\Lambda
(s-\tau)}\sqrt{\mathcal{E}_{G}(\tau)}\norm{\dt u(\tau)}_{1}d\tau}ds,
\end{align}
which implies \eqref{result2}.

Similarly, we use the continuity equation $\partial_t \q =-\diverge ( \bar{\rho}   u)+G^1$ to estimate
\begin{equation}\label{j10}
\| \partial_tq \|_{0}\le \|{\rm div}(\bar\rho u) \|_0+\norm{G^1}_0
 \lesssim  \| u\|_{1}+\norm{G^1}_0.
\end{equation}
We then deduce \eqref{result3} as that for \eqref{result2}.
\end{proof}

\section{Nonlinear energy estimates}\label{energy}

This section, the most technical part of the paper, is devoted to the nonlinear energy estimates for the system \eqref{geometric}.  Our analysis here is similar to that found in our companion paper on the stable regime \cite{JTW_GWP}.  The primary difference is that we will use  slightly different versions of the energy and dissipation functionals in order to handle the fact that the internal interface makes a negative contribution to the energy and  dissipation.

For any integer $N\ge 3$, we define the energy as
\begin{align}\label{p_energy_def}
 \nonumber\se{2N}^\sigma := &\sum_{j=0}^{2N}  \ns{\dt^j u}_{4N-2j} + \ns{\q}_{4N}  + \sum_{j=1}^{2N} \ns{ \dt^j \q}_{4N-2j+1}
 \\&+\sigma\ns{ \nabla_\ast\eta}_{4N} +\ns{ \eta}_{4N}+ \sum_{j=1}^{2N} \ns{\dt^j \eta}_{4N-2j+3/2},
\end{align}
and the ``dissipation" as
\begin{align}\label{p_dissipation_def}
\nonumber  \sd{2N}^\sigma := & \sum_{j=0}^{2N} \ns{\dt^j u}_{4N-2j+1}  + \ns{\dt \q}_{4N-1} + \sum_{j=2}^{2N+1} \ns{\dt^j \q}_{4N-2j+2}
\\
&+   \sigma^2 \ns{ \nabla_\ast \eta}_{4N+1/2}
+ \sigma^2\ns{\partial_t \eta}_{4N+1/2} +\ns{\partial_t \eta}_{4N-1/2}
 + \sum_{j=2}^{4N+1} \ns{\dt^j \eta}_{4N-2j+5/2}.
\end{align}
We also define
\begin{equation}\label{fff}
\f_{2N}:=  \ns{\eta}_{4N+1/2}.
\end{equation}
The surface tension coefficients $\sigma_\pm$ are included in the definitions \eqref{p_energy_def} and \eqref{p_dissipation_def} so that we will be able to treat the cases with and without surface tension together.
It is noteworthy that the definition \eqref{p_dissipation_def} of $\sd{2N}^\sigma$ is different from the one introduced in \cite{JTW_GWP} for the nonlinear stability analysis: $\ns{\q}_{4N}$ and $\ns{\eta}_{4N-1/2}$ are not included in $\sd{2N}^\sigma$ here. This implies that to control $\ns{\q}_{4N}$ or $\ns{\eta}_{4N-1/2}$, we have to use $\se{2N}^0$; mostly we will replace $\sd{2N}^\sigma$ in the estimates of some nonlinear terms derived in \cite{JTW_GWP} by the sum $\sd{2N}^\sigma+\se{2N}^0$.

Our goal is to derive a priori estimates for solutions $(q,u,\eta)$ to \eqref{geometric} in our functional framework, i.e. for solutions satisfying $\se{2N}^\sigma$, $\sd{2N}^\sigma$, $\f_{2N}<\infty$.  Throughout the rest of this section we will assume that
\begin{equation}
\mathcal{E}^\sigma_{2N}(t)\le \delta^2\le 1
\end{equation}
for some sufficiently small $\delta>0$ and for all $t \in [0,T]$ where $T>0$ is given. This assumption, in particular, will guarantee that the geometric terms introduced in Section \ref{sec1.3} are well-behaved (see Lemma \ref{eta_small}). We will implicitly allow $\delta$ to be made smaller in each result, but we will reiterate the smallness of $\delta$ in our main result. Here is the main result of this section.

\begin{thm}\label{energyeses}
If $\sup_{0 \le t \le T}  \mathcal{E}_{2N}^\sigma(t)\le \delta^2$ for sufficiently small $\delta$, then the following holds.  For any $\varepsilon>0$, there exists $C_\varepsilon>0$ such that
\begin{align}\label{fullenergy}
\nonumber
 {\mathcal{E}}_{2N}^\sigma(t)+\f_{2N}(t) + \int_0^t  {\mathcal{D}}_{2N}^\sigma ds
\le& C_\varepsilon{\mathcal{E}}_{2N}^\sigma(0)+\f_{2N}(0) + C_\varepsilon\int_0^t  \sqrt{\mathcal{E}_{2N}^\sigma }\(  {\mathcal{D}}_{2N}^\sigma+\mathcal{E}_{2N}^\sigma+ \f_{2N}\)ds
\\&+\varepsilon\int_0^t  \(\mathcal{E}_{2N}^\sigma+\f_{2N}\)  ds+  C_\varepsilon\int_0^t\norm{\eta_-}_{0}^2ds
\end{align}
for all $t \in [0,T]$.
\end{thm}

Theorem \ref{energyeses} will be established by a series of energy estimates, elliptic estimates, and comparison results and its final proof will be given at the end of this section. We start with the time differentiated version of problem  \eqref{geometric}.

\subsection{Energy evolution for temporal derivatives in geometric form}\label{stable1}

We will employ the form of the equations \eqref{geometric} primarily for estimating the temporal derivatives of the solutions.  Applying the temporal differential operator $\dt^j$ for $j=0,\dots,2N$ to \eqref{geometric}, we find that
\begin{equation}\label{linear_geometric}
\begin{cases}
\partial_t (\dt^j \q )+\diverge_\a ( \bar{\rho}  \dt^j u)=F^{1,j} & \text{in }
\Omega  \\
( \bar{\rho} +  \q+\p_3\bar\rho\theta)\partial_t   (\dt^j u )  + \bar{\rho}\nabla_\a  (h'(\bar{\rho}) \dt^j\q )   -\diva\S_{\a}(\dt^j u) = F^{2,j}  & \text{in } \Omega \\
  \dt (\dt^j \eta)   = \dt^j u\cdot \n+F^{4,j} & \text{on } \Sigma\\
(  P'(\bar{\rho}) \dt^j \q I- \S_{\a}(\dt^j u))\n  =  \rho_1  g \dt^j \eta_+ \n-\sigma_+ \Delta_\ast(\dt^j \eta_+)  \n +F_+^{3,j}
 & \text{on } \Sigma_+
 \\ \jump{P'(\bar\rho)\dt^j \q I- \S_\a(\dt^j u)}\n
= \rj g \dt^j \eta_-\n +\sigma_- \Delta_\ast(\dt^j \eta_-)\n-F_-^{3,j}&\hbox{on }\Sigma_-
 \\\jump{\dt^j u}=0 &\hbox{on }\Sigma_-\\
 \dt^j u_- =0 & \text{on } \Sigma_b,
\end{cases}
\end{equation}
where
\begin{equation}\label{F1j_def}
 F^{1,j}  =  \dt^j F^1-\sum_{0 < \ell \le j }  C_j^\ell \dt^{ \ell}\mathcal{A}_{lk}\p_k( \bar{\rho}  \dt^{j-\ell} u_l),
 \end{equation}
\begin{align}\label{F2j_def}
\nonumber
 F_i^{2,j}  &=\dt^j F_i^2+\sum_{0 < \ell \le j }  C_j^\ell\left\{\mu\mathcal{A}_{lk} \p_k (\dt^{ \ell}\mathcal{A}_{lm}\dt^{j - \ell}\p_m u_i)
+\mu\dt^{ \ell}\mathcal{A}_{lk} \dt^{j - \ell}\p_k (\mathcal{A}_{lm}\p_m u_i)
  \right.
\\\nonumber
 &\quad+   (\mu/3+\mu')\mathcal{A}_{ik} \p_k (\dt^{ \ell}\mathcal{A}_{lm}\dt^{j - \ell}\p_m u_l)
+(\mu/3+\mu')\dt^{ \ell}\mathcal{A}_{ik} \dt^{j - \ell}\p_k (\mathcal{A}_{lm}\p_m u_l)
\\&\quad\left.-  \bar{\rho} \dt^{\ell} \mathcal{A}_{ik} \p_k
(h'(\bar{\rho})\dt^{j - \ell} \q)-\dt^\ell(\q+\p_3\bar\rho\theta)\partial_t   (\dt^{j-\ell} u ) \right\} ,
 \end{align}
\begin{equation}\label{F3j+_def}
\begin{split}
F_{i,+}^{3,j} &= \dt^jF_{i,+}^{3}+ \sum_{0 < \ell \le j} C_j^\ell \left\{\mu_+\dt^{ \ell} ( \n_{l} \mathcal{A}_{ik} ) \dt^{j - \ell} \p_k u_{l} + \mu_+ \dt^{ \ell} ( \n_{l} \mathcal{A}_{lk} ) \dt^{j - \ell} \p_k u_{i}\right.
\\&\quad \left.+(\mu_+'-2\mu_+/3)\dt^{ \ell} ( \n_{i}\mathcal{A}_{lk} ) \dt^{j - \ell} \p_k u_{l}+ \dt^{\ell} \n_{i} \dt^{j - \ell}(\rho_1 g \eta_+- P'(\bar{\rho})   \q-\sigma_+   \Delta_\ast\eta_+ )\right\},
\end{split}
\end{equation}
\begin{align}\label{F3j-_def}
\nonumber-F_{i,-}^{3,j} &= -\dt^jF_{i,-}^{3}+ \sum_{0 < \ell \le j} C_j^\ell \left\{\dt^{ \ell} ( \n_{l} \mathcal{A}_{ik} ) \dt^{j - \ell} \jump{\mu \p_k u_{l}} + \dt^{ \ell} ( \n_{l} \mathcal{A}_{lk} ) \dt^{j - \ell} \jump{\mu\p_k u_{i}}\right.
\\&\quad\left.+\dt^{ \ell} ( \n_{i}\mathcal{A}_{lk} ) \dt^{j - \ell} \jump{(\mu'-2\mu/3)\p_k u_{l}}+ \dt^{\ell} \n_{i} \dt^{j - \ell}( \rj g\eta_-- \jump{P'(\bar\rho)   \q}+\sigma_-  \Delta_\ast\eta_-) \right\},
\end{align}
for  $i=1,2,3,$ and
\begin{equation}\label{F4j_def}
 F^{4,j} =  \sum_{0 < \ell \le j}  C_j^\ell \dt^{ \ell} \n\cdot \dt^{j - \ell}  u.
\end{equation}
In the above, $F^1$, $F^2$ and $F^3$ are defined by
\begin{equation}
 F^{1}=\p_3^2\bar\rho K \theta \p_t\theta+K\p_t\theta \pa_3  \q -\diverge_\a((\q+\p_3\bar\rho\theta) u ),
\end{equation}
\begin{equation}
\begin{split}
 F^2 =  -( \bar{\rho} +  \q+\p_3\bar\rho\theta)
(-K\p_t\theta \pa_3  u
 +   u \cdot \nab_\a  u )
       - \nabla_\a\mathcal{R}-g( \q+\p_3\bar\rho\theta ) \nabla_\a \theta ,
      \end{split}
\end{equation}
 \begin{equation}
F^3_{+}= - \mathcal{R} \n-\sigma_+\diverge_\ast(((1+|\nab_\ast\eta_+|^2)^{-1/2}-1)\nab_\ast\eta_+) \n,
\end{equation}
and
 \begin{equation}
 -F^3_-= - \jump{ \mathcal{R} }\n
 +\sigma_-\diverge_\ast(((1+|\nab_\ast\eta_-|^2)^{-1/2}-1)\nab_\ast\eta_-) \n.
\end{equation}

We present the estimates of these nonlinear terms $F^{1,j}$, $F^{2,j}$, $F^{3,j}$ and $F^{4,j}$ in the following lemma.

\begin{lem}\label{p_F2N_estimates}
For each $0\le j\le 2N$, we have
\begin{equation}\label{p_F_e_01}
 \ns{F^{1,j} }_{0}+ \ns{  F^{2,j} }_{0}   + \ns{F^{3,j}}_{0} + \norm{F^{4,j}}_{0}^2 \ls   \se{2N}^0 \(\sd{2N}^\sigma+\se{2N}^0\).
\end{equation}
\end{lem}
\begin{proof}
The estimate is restated from Lemma 3.8 of \cite{JTW_GWP}. Note that the appearance of $(\se{2N}^0)^2$ in the estimate is due to the lack of $\ns{\q}_{4N}$ and $\ns{\eta}_{4N-1/2}$ in the definition \eqref{p_dissipation_def} of $\sd{2N}^\sigma$; we use $\se{2N}^0$ to control them here.
\end{proof}

We now estimate the energy evolution of the pure temporal derivatives.
\begin{prop}\label{i_temporal_evolution 2N}
It holds that
\begin{align} \label{tem en 2N}
\nonumber
&\sum_{j=0}^{2N}\(\norm{ \partial_t^j
q(t)}_0^2 +\norm{ \partial_t^j
u(t)}_0^2 +\norm{ \partial_t^j
\eta_+(t)}_0^2 +\sigma\norm{\nabla_\ast \partial_t^j
\eta(t)}_0^2\) + \int_0^t\sum_{j=0}^{2N}\norm{ \partial_t^j
u }_1^2\,ds\\&\quad\lesssim \se{2N}^\sigma(0) +\int_0^t\sqrt{{\mathcal{E}^\sigma
_{2N} }} \(\mathcal{D}^\sigma_{2N}+\mathcal{E}^\sigma_{2N}\) \,ds  +  \int_0^t\norm{
\eta_-  }_{0}^2\,ds .
\end{align}
\end{prop}
\begin{proof}
We take the dot product of the second equation of \eqref{linear_geometric} with $J\dt^j u$ and integrate by parts over the domain $\Omega$; using the other conditions in \eqref{linear_geometric} and some easy geometric identities involving $J,\mathcal{A},$ and $\mathcal{N}$, as in Proposition 3.1 of \cite{JTW_GWP}, we obtain the following energy identity:
\begin{align}\label{identity1}
\nonumber &\frac{1}{2}\frac{d}{dt}\left(\int_\Omega   ( \bar{\rho} +  \q+\p_3\bar\rho\theta)J\abs{\partial_t^j u }^2+h'(\bar{\rho})J \abs{\dt^j\q}^2+\int_{\Sigma_+} \rho_1  g\abs{\dt^j \eta_+}^2+\int_\Sigma \sigma\abs{\nab_\ast \dt^j \eta}^2\right)
\\ \nonumber&\quad +\int_{\Omega}\frac{\mu}{2}J\abs{\sgz_\mathcal{A}\partial_t^j u}^2+\mu'J\abs{\diverge_\mathcal{A}\partial_t^j u}^2
\\ \nonumber&\quad =\frac{1}{2} \int_\Omega \dt(J( \bar{\rho} +  \q+\p_3\bar\rho\theta)) \abs{\partial_t^j u }^2+h'(\bar{\rho})\dt J \abs{\dt^j\q}^2+\int_\Omega J(h'(\bar{\rho}) \dt^j\q   F^{1,j}+\partial_t^j u\cdot F^{2,j})
\\  &\qquad+\int_{\Sigma}-\partial_t^j u\cdot F^{3,j}+\int_{\Sigma_+} \rho_1  g\dt^j \eta_+F_+^{4,j} -\int_\Sigma\sigma\Delta_\ast (\dt^j \eta) F^{4,j} + \int_{\Sigma_-}\rj g\dt^j \eta_-\n\cdot \dt^j u.
\end{align}

First, we may argue as in Proposition 4.3 of \cite{GT_per}, utilizing Lemma \ref{eta_small}, to estimate
\begin{equation} \label{i_te_1}
 \int_{\Omega}\frac{\mu}{2}
J\abs{\sgz_\mathcal{A}\partial_t^j u}^2+\mu'J\abs{\diverge_\mathcal{A}\partial_t^j u}^2
\ge  \int_{\Omega}\frac{\mu}{2}
 \abs{\sgz \partial_t^j u}^2+\mu' \abs{\diverge \partial_t^j u}^2-C\sqrt{ \mathcal{E}_{2N}^\sigma}\mathcal{D}_{2N}^\sigma.
 \end{equation}
We then estimate the right hand side of \eqref{identity1} for $0\le j\le 2N$. For the first two terms, we may bound  as usual $\norm{\dt J}_{L^\infty}\ls \sqrt{ \mathcal{E}_{2N}^\sigma}$ and $\norm{\dt( J( \bar{\rho} +  \q+\p_3\bar\rho\theta))}_{L^\infty}\ls \sqrt{ \mathcal{E}_{2N}^\sigma}$ to have
 \begin{align}\label{i_te_1'}
\nonumber
 &\frac{1}{2} \int_\Omega  \dt (J( \bar{\rho} +  \q+\p_3\bar\rho\theta))\abs{\partial_t^j u }^2+h'(\bar{\rho})\dt J \abs{\dt^j\q}^2
 \\&\quad\ls \sqrt{ \mathcal{E}_{2N}^\sigma}\left( \ns{\partial_t^j u}_0+ \ns{\partial_t^j \q}_0\right) \ls \sqrt{ \mathcal{E}_{2N}^\sigma}\mathcal{E}_{2N}^\sigma.
 \end{align}
By Lemma \ref{p_F2N_estimates}, we may bound the $F^{1,j}$ and $F^{2,j}$ terms as
\begin{align}\label{i_te_2}
\nonumber\int_\Omega J(h'(\bar{\rho}) \dt^j\q   F^{1,j}+\partial_t^j u\cdot F^{2,j}) &\ls   \norm{\dt^j \q}_{0}   \norm{F^{1,j}}_0+ \norm{\dt^j u}_{0}   \norm{F^{2,j}}_0\\& \ls \sqrt{\se{2N}^\sigma } \sqrt{\se{2N}^\sigma \(\sd{2N}^\sigma+\se{2N}^\sigma\)}.
\end{align}
For the $F^{3,j}$ and $F^{4,j}$ terms, by Lemma \ref{p_F2N_estimates} and   trace theory, we have
\begin{align}\label{i_te_3}
\nonumber
 &\int_{\Sigma}-\partial_t^j u\cdot F^{3,j}+\int_{\Sigma_+} \rho_1  g\dt^j \eta_+F_+^{4,j} -\int_\Sigma\sigma\Delta_\ast (\dt^j \eta) F^{4,j}
 \\\nonumber &\quad\ls  \snormspace{\dt^{j} u}{0}{\Sigma} \norm{F^{3,j}}_{0} +\left( \norm{\dt^{j} \eta_+}_{0}+ \sigma\norm{\Delta_\ast\dt^{j} \eta}_{0} \right) \norm{F^{4,j}}_{0} \\\nonumber&\quad
\ls  \left( \norm{\dt^{j} u}_{1} +\norm{\dt^{j} \eta_+}_{0}+ \sigma\norm{\Delta_\ast\dt^{j} \eta}_{0} \right)  \sqrt{\se{2N}^\sigma\(\sd{2N}^\sigma+\se{2N}^\sigma\)}
 \\&\quad\ls   \sqrt{ \sd{2N}^\sigma+\se{2N}^\sigma } 
 \sqrt{\se{2N}^\sigma \(\sd{2N}^\sigma+\se{2N}^\sigma\)}.
\end{align}
For the last term, by the trace theorem and Cauchy's inequality, we have
 \begin{equation}\label{i_te_4}
   \int_{\Sigma_-}
 \rj g \partial_t^j \eta_- \n\cdot \partial_t^j u
 \ls\norm{
\partial_t^j \eta_- }_{0} \norm{\partial_t^j u }_{H^0(\Sigma_-)}
\ls C_\varepsilon \norm{
\partial_t^j \eta_-}_{0}^2 +\varepsilon\norm{\partial_t^j u}_1^2
 \end{equation}
for any $\varepsilon>0$.

Consequently, employing Korn's inequality from Proposition \ref{layer_korn} in \eqref{i_te_1} together with the estimates \eqref{i_te_1'}--\eqref{i_te_4}, taking  $\varepsilon$  sufficiently small, and integrating \eqref{identity1} from $0$ to $t$, we deduce  that
\begin{equation}\label{tempor5}
\begin{split}
&\norm{ \partial_t^j
q(t)}_0^2 +\norm{ \partial_t^j
u(t)}_0^2 +\norm{ \partial_t^j
\eta_+(t)}_0^2 +\sigma\norm{\nabla_\ast \partial_t^j
\eta(t)}_0^2 + \int_0^t\norm{ \partial_t^j
u}_1^2\,ds\\&\quad\lesssim \se{2N}^\sigma(0) +\int_0^t\sqrt{{\mathcal{E}^\sigma
_{2N}}} \(\sd{2N}^\sigma+\se{2N}^\sigma\)\,ds  +  \int_0^t\norm{
\partial_t^j \eta_-}_{0}^2\,ds .
\end{split}
\end{equation}

Now taking $j=0$ in \eqref{tempor5}, we have
\begin{equation}\label{tempor7}
\begin{split}
  &\norm{
q(t)}_0^2 +\norm{
u(t)}_0^2 +\norm{
\eta_+(t)}_0^2 +\sigma\norm{\nabla_\ast
\eta(t)}_0^2 + \int_0^t\norm{  u }_1^2\,ds \\&\quad\lesssim
\se{2N}^\sigma(0) +\int_0^t\sqrt{{\mathcal{E}^\sigma
_{2N} }} \(\sd{2N}^\sigma+\se{2N}^\sigma\)\,ds  +  \int_0^t\norm{  \eta_-}_{0} ^2\,ds
.
\end{split}
\end{equation}
For $j=1,\dots,2N$, the kinematic boundary condition,  trace theory and the estimates \eqref{p_F_e_01} show that
\begin{equation}\label{tempor8}
\norm{
\partial_t^j \eta_-}_{0}^2 \le\norm{
\partial_t^{j-1}u\cdot \n}_{H^0(\Sigma)}^2 +\norm{F^{4,j-1} }_{0}^2\ls\norm{
\partial_t^{j-1}u }_{1}^2+ { \mathcal{E}_{2N}^\sigma\(\sd{2N}^\sigma+\se{2N}^\sigma\)} .
\end{equation}
Plugging \eqref{tempor8} into \eqref{tempor5}, by using $\sqrt{{\mathcal{E}^\sigma
_{2N} }} \le 1$, we obtain
\begin{equation}\label{tempor9}
\begin{split}
&\norm{ \partial_t^j
q(t)}_0^2 +\norm{ \partial_t^j
u(t)}_0^2 +\norm{ \partial_t^j
\eta_+(t)}_0^2 +\sigma\norm{\nabla_\ast \partial_t^j
\eta(t)}_0^2 + \int_0^t\norm{ \partial_t^j
u }_1^2\,ds\\&\quad\lesssim \se{2N}^\sigma(0) +\int_0^t\sqrt{{\mathcal{E}^\sigma
_{2N} }} \(\sd{2N}^\sigma+\se{2N}^\sigma\)\,ds  +  \int_0^t\norm{
\partial_t^{j-1}u  }_{1}^2\,ds.
\end{split}
\end{equation}
Hence, by chaining together \eqref{tempor7} and \eqref{tempor9}, we get \eqref{tem en 2N}.
\end{proof}

We remark that the energy identity in Proposition 3.1 of \cite{JTW_GWP} is slightly different from \eqref{identity1}. Unlike in Proposition 3.1 of \cite{JTW_GWP}, we do not employ the kinematic boundary condition in treating the last term in \eqref{identity1} because $\rj>0$; if we did this, it would involve a negative term, $-\rj g\norm{\eta_-}_0^2$, in the energy. As a result, for $\sigma_-<\sigma_c$, the energy becomes non-positive definite, which is the cause of the instability.

\subsection{Energy evolution for horizontal space-time derivatives in linear form}\label{stable2}

We now estimate the energy evolution of the mixed horizontal space-time derivatives.
It turns out to be convenient to rewrite the system \eqref{geometric} in a linear form such that the coefficients get fixed and that the elliptic regularity is readily adapted in later sections.
The PDEs \eqref{geometric} can be also rewritten for $(\q,u,\eta)$ as
\begin{equation}\label{ns_perturb}
\begin{cases}
\partial_t \q +\diverge ( \bar{\rho}   u)=G^1 & \text{in }
\Omega  \\
 \bar{\rho} \partial_t    u   + \bar{\rho}\nabla \left(h'(\bar{\rho})\q\right)   -\diverge \S(u) =G^2 & \text{in }
\Omega  \\
\partial_t \eta = u_3+G^4 &
\text{on }\Sigma  \\
(  P'(\bar\rho)\q I- \S(  u))e_3  = (\rho_1  g \eta_+ -\sigma_+ \Delta_\ast \eta_+ ) e_3 +G_+^3
 & \text{on } \Sigma_+
 \\ \jump{P'(\bar\rho)\q I- \S(u)}e_3
=(\rj g\eta_- +\sigma_- \Delta_\ast \eta_-)e_3-G_-^3&\hbox{on }\Sigma_-
\\\jump{u}=0 &\hbox{on }\Sigma_-
\\ u_-=0 &\hbox{on }\Sigma_b,
\end{cases}%
\end{equation}
where  we have written the function $G^1=G^{1,1}+G^{1,2}$ for
\begin{equation}\label{G1_def}
G^{1,1}= K  \p_t\theta\pa_3  \q-  u_l \mathcal{A}_{lk}\pa_k  \q,
\end{equation}
\begin{equation}
G^{1,2}= \p_3^2\bar\rho K \theta \p_t\theta- \q \mathcal{A}_{lk}\pa_k u_l-\mathcal{A}_{lk}\pa_k(\p_3\bar\rho\theta   u_l)-(\mathcal{A}_{lk}-\delta_{lk} )\pa_k ( \bar{\rho}   u_l),
\end{equation}
the vector $G^2$ for
\begin{align}
\nonumber G^2_i=&-( \q+\p_3\bar\rho\theta )\partial_t    u_i+(\bar\rho+\q+\p_3\bar\rho\theta )
(K  \p_t\theta \pa_3  u_i
 -   u_l\mathcal{A}_{lk}\pa_k u_i)
  \\\nonumber& +  \mu \mathcal{A}_{lk} \pa_k\mathcal{A}_{lm}\pa_mu_i  +
  \mu (\mathcal{A}_{lk}\mathcal{A}_{lm}- \delta_{lk}\delta_{lm}) \pa_{km} u_i
   \\\nonumber&
  +(\mu/3+\mu') \mathcal{A}_{ik} \pa_k\mathcal{A}_{lm}\pa_mu_l
   +(\mu/3+\mu')(\mathcal{A}_{ik}\mathcal{A}_{lm}- \delta_{ik}\delta_{lm}) \pa_{km} u_l
   \\& - \bar{\rho}(\mathcal{A}_{il}-\delta_{il})\pa_l ( h'(\bar{\rho})\q)
       - \mathcal{A}_{il}\pa_l  \mathcal{R}-g( \q+\p_3\bar\rho\theta ) \a_{il}\p_l \theta
\end{align}
for $i=1,2,3,$
the vector $G^3_+ = G^{3,1}_+ + \sigma_+ G^{3,2}_+$ for
\begin{align}
\nonumber G^{3,1}_{i,+}=&
  \mu_+  (\mathcal{A}_{il} \partial_l u_k+\mathcal{A}_{kl} \partial_l u_i)(\n_{k }-\delta_{k3})
 + \mu_+  (\mathcal{A}_{il}-\delta_{il}) \partial_l u_3+ \mu (\mathcal{A}_{3l}-\delta_{3l}) \partial_l u_i
 \\\nonumber & +(\mu_+'-2\mu_+/3)  \mathcal{A}_{lk} \partial_k u_l (\n_{i }-\delta_{i3})
+ (\mu_+'-2\mu_+/3)   (\mathcal{A}_{lk}-\delta_{lk}) \partial_k u_l\delta_{i3}
\\
&+  \rho_1 g \eta_+ (\n_i-\delta_{i3}) -\mathcal{R}\n_{i }
 + P'(\bar{\rho}) \q (\delta_{i3} - \n_i)
  \end{align}
and
\begin{equation}
 G^{3,2}_{i,+} = -\Delta_\ast\eta_+ (\n_i-\delta_{i3}) - \diverge_\ast(((1+|\nab_\ast\eta_+|^2)^{-1/2}-1)\nab_\ast\eta_+)\n_i
\end{equation}
for $i=1,2,3,$ and the vector $G^3_- = G^{3,1}_- + \sigma_- G^{3,2}_-$ for
\begin{align}
\nonumber -G^{3,1}_{i,-}=&
  (\mathcal{A}_{il} \jump{\mu  \partial_l u_k}+\mathcal{A}_{kl} \jump{\mu \partial_l u_i})(\n_{k}-\delta_{k3})
+ (\mathcal{A}_{il}-\delta_{il}) \jump{\mu \partial_l u_3}-(\mathcal{A}_{3l}-\delta_{3l}) \jump{\mu \partial_l u_i}
 \\\nonumber&   + \mathcal{A}_{lk} \jump{(\mu'-2\mu/3) \partial_k u_l} (\n_{i }-\delta_{i3})
+ (\mathcal{A}_{lk}-\delta_{lk}) \jump{(\mu'-2\mu/3)\partial_k u_l}\delta_{i3}
\\&  + \rj g\eta_-(\n_i-\delta_{i3}) -\jump{\mathcal{R}}\n_{i }
  + \jump{P'(\bar{\rho})\q} (\delta_{i3} - \n_i)
  \end{align}
and
\begin{equation}
 G^{3,2}_{i,-} =  \Delta_\ast\eta_-(\n_i-\delta_{i3})  + \diverge_\ast(((1+|\nab_\ast\eta_-|^2)^{-1/2}-1)\nab_\ast\eta_-)\n_i
\end{equation}
for  $i=1,2,3,$  and the function $G^4$ for
\begin{eqnarray}\label{G4_def}
 G^4= -u_1\pa_1\eta-u_2\pa_2\eta.
\end{eqnarray}

We now present the estimates of these nonlinear terms $G^1$, $G^2$, $G^3$ and $G^4$. Recall the notation $\bar{\nab}$ for space-time derivatives  in \eqref{barnotation}.
\begin{lem}\label{p_G2N_estimates}
 It holds that
\begin{equation}\label{p_G_e_0}
 \ns{ \bar{\nab}^{4N-2} G^1}_{1}  +  \ns{ \bar{\nab}^{4N-2}  G^2}_{0} +
 \ns{ \bar{\nab}_{\ast}^{4N-2}  G^3}_{1/2}
+ \ns{\bar{\nab}_{\ast }^{ 4N-1} G^4}_{1/2}
 \ls   \se{2N}^0\(\se{2N}^\sigma +   \f_{2N}\),
\end{equation}
and
\begin{align}\label{p_G_e_00}
\nonumber&\ns{ \bar{\nab}^{4N-1}  G^{1,1}}_{0}+\ns{ \bar{\nab}^{4N-2}\dt G^{1,1}}_{0}  + \ns{ \bar{\nab}^{4N}  G^{1,2}}_{0}+  \ns{ \bar{\nab}^{4N-1}  G^2}_{0}
\\\nonumber&\quad+ \ns{ \bar{\nab}_{\ast}^{4N-1} G^3}_{1/2} + \ns{\bar{\nab}_{\ast }^{  4N-1} G^4}_{1/2}
   + \ns{\bar{\nab}_{\ast }^{ 4N-2} \dt G^4}_{1/2}
+\sigma^2\ns{\bar{\nab}_{\ast }^{ 4N} G^4}_{1/2}
\\  &\quad\ls    \se{2N}^0\(\sd{2N}^\sigma +\se{2N}^0 +   \f_{2N}\).
\end{align}
\end{lem}
\begin{proof}
The estimates are restated from Lemma 3.3 of \cite{JTW_GWP}. The reason for the appearance of $(\se{2N}^0)^2$ is the same as in Lemma \ref{p_F2N_estimates}.
\end{proof}

Next we present some variants of these estimates involving integrals of certain products.  First we consider products with derivatives of $G^4$.

\begin{lem}\label{lemma8}
Let $\al\in \mathbb{N}^2$ such that $|\al|=4N$. Then
\begin{equation}\label{eta es}
\left| \int_{\Sigma}   \pa^\al \eta \pa^\al G^4 \right| \lesssim \sqrt{\se{2N}^0} \( \sd{2N}^0+\se{2N}^0\) +\sqrt{  \sd{2N}^0\mathcal{E}_{2N }^0 \f_{2N}}
\end{equation}
and
\begin{equation}\label{eta es2}
\left| \int_{\Sigma}   \sigma \Delta_\ast\pa^\al\eta \pa^\al G^4 \right| \lesssim    \sqrt{\se{2N}^0\sd{2N}^0\sd{2N}^\sigma}  +   \sqrt{  \sd{2N}^\sigma\mathcal{E}_{2N}^0 \f_{2N}} .
\end{equation}
 \end{lem}
 \begin{proof}
The estimates are restated from Lemma 3.5 of \cite{JTW_GWP}.
\end{proof}

Next we consider products with derivatives of  $G^{1,1}$.

\begin{lem}\label{lemma7}
Let $\al\in \mathbb{N}^3$ such that $|\al|=4N$.
Then
\begin{equation}\label{rho es}
\left| \int_{\Omega}   h'(\bar{\rho})\pa^\al \q \pa^\al G^{1,1} \right|\lesssim  \sqrt{\mathcal{D}_{2N}^0+\se{2N}^0}\sqrt{\se{2N}^0  \(\sd{2N}^0+\se{2N}^0 +   \f_{2N}\)} .
\end{equation}

\end{lem}
  \begin{proof}
The estimate is restated from Lemma 3.6 of \cite{JTW_GWP}.
\end{proof}

We also consider a similar estimate involving weights and derivatives of $G^{1,1}$.

\begin{lem}\label{G11_weighted}
Let $\alpha \in \mathbb{N}^{1+3}$ with $\abs{\alpha} \le 4N$ and $\alpha_0 \leq 2N-1$.
Then
\begin{equation}
  \abs{\int_\Omega \left(1+\frac{ 4\mu/3+\mu'  }{h'(\bar\rho)\bar\rho^2}\right) \partial^\alpha(h'(\bar\rho) \q)\partial^\alpha \left(h'(\bar\rho) G^{1,1} \right)}
 \ls   \sqrt{\mathcal{D}_{2N}^0+\se{2N}^0}\sqrt{\se{2N}^0  \(\sd{2N}^0+\se{2N}^0 +   \f_{2N}\)}  .
\end{equation}
\end{lem}
\begin{proof}
The estimate is restated from Lemma 3.7 of \cite{JTW_GWP}.
\end{proof}

We now estimate the energy evolution of the mixed horizontal space-time derivatives.
 \begin{prop}\label{i_spatial_evolution 2N}
It holds that
\begin{align} \label{energy 2N}
 \nonumber&\sum_{\substack{\alpha\in \mathbb{N}^{1+2}\\ |\alpha|\le 4N  \\ \alpha_0\le 2N-1}} \(\norm{\partial^\alpha   q (t)}_0^2+\norm{\partial^\alpha   u (t)}_0^2+\norm{\partial^\alpha  \eta_+ (t)}_0^2+\sigma\norm{\partial^\alpha   \eta (t)}_1^2\)
  +\int_0^t\sum_{\substack{\alpha\in \mathbb{N}^{1+2}\\ |\alpha|\le 4N \\ \alpha_0\le 2N-1}} \norm{\partial^\alpha   u}_1^2\,ds
 \\
&\quad\lesssim  {\mathcal{E}}^\sigma_{2N}(0)+  \int_0^t\sqrt{{\mathcal{E}^\sigma
_{2N} }} \(\mathcal{D}^\sigma_{2N}+\mathcal{E}^\sigma_{2N}+\f_{2N}\) \,ds
 + \int_0^t \norm{ \eta_- }_{4N-1/2}^2\,ds .
 \end{align}
 \end{prop}
\begin{proof}
Since the boundaries of $\Omega_\pm$ are flat we are free to apply time derivatives and horizontal derivatives to the equations \eqref{ns_perturb}.  Let $\alpha\in \mathbb{N}^{1+2}$ such that $\alpha_0\le 2N-1$ and $|\alpha|\le 4N$. We apply $\partial^\alpha$  to \eqref{ns_perturb} and argue as in Proposition \ref{i_temporal_evolution 2N}  to obtain the following energy identity:
\begin{align}\label{es_00}
\nonumber
&\frac{1}{2}\frac{d}{dt}\left(\int_\Omega h'(\bar{\rho})\abs{ \pa^\al \q}^2+\bar{\rho}\abs{\pa^\al u}^2 +\int_{\Sigma_+} \rho_1  g\abs{\pa^\al \eta_+}^2+\int_\Sigma\sigma \abs{\nab_\ast \pa^\al \eta }^2\right)
\\\nonumber&\qquad
 +\int_\Omega \frac{\mu}{2} \abs{ \sgz\pa^\al u}^2+\mu' \abs{{\rm div }\pa^\al u}^2
\\\nonumber&\qquad    = \int_\Omega h'(\bar{\rho})\pa^\al \q \pa^\al G^1 +\pa^\al u\cdot\pa^\al G^2+\int_\Sigma -\pa^\al u\cdot \pa^\al G^3
\\&\qquad\quad+ \int_{\Sigma_+} \rho_1  g \pa^\al \eta_+ \pa^\al G^4_+ -\int_\Sigma\sigma\Delta_\ast (\pa^\al \eta) \pa^\al G^4
+\int_{\Sigma_-} \rj g\pa^\al \eta_- \pa^\al u_3.
\end{align}

We first estimate the $G^2,G^3,G^4$ terms in the right hand side of  \eqref{es_00}. We assume initially that $|\al|\le 4N-1$. Then by the estimates \eqref{p_G_e_00} of Lemma \ref{p_G2N_estimates}, we have
\begin{equation}\label{gg1}
 \left|\int_\Omega  \pa^\al u\cdot\pa^\al G^2\right| \le  \norm{\pa^\al u}_0\norm{\pa^\al G^2}_0 \lesssim \sqrt{\mathcal{D}_{2N}^\sigma}\sqrt{  \mathcal{E}_{2N}^\sigma \(\sd{2N}^\sigma +\mathcal{E}^\sigma_{2N}+   \f_{2N}\)}.
\end{equation}
Similarly,  the estimates \eqref{p_G_e_00} of Lemma \ref{p_G2N_estimates} and trace theory show that
\begin{align}
 \nonumber\left| \int_\Sigma \pa^\al u\cdot \pa^\al G^3 \right|&\le  \norm{\pa^\al u}_{H^0(\Sigma)}  \norm{\pa^\al G^3}_0 \ls  \norm{\pa^\al u}_{1}  \norm{\pa^\al G^3}_0
\\&\lesssim \sqrt{\mathcal{D}_{2N}^\sigma}\sqrt{  \mathcal{E}_{2N}^\sigma\(\sd{2N}^\sigma+\mathcal{E}^\sigma_{2N} +   \f_{2N}\)}
 \end{align}
and
\begin{align}
 \nonumber
\left|\int_{\Sigma_+} \rho_1  g \pa^\al \eta_+ \pa^\al G^4_+  -\int_\Sigma\sigma\Delta_\ast (\pa^\al \eta) \pa^\al G^4\right|&\lesssim\left(\norm{\pa^\al \eta}_0+\sigma\norm{\Delta_\ast \pa^\al \eta}_0\right)\norm{\pa^\al G^4}_0
\\&\lesssim   \sqrt{\mathcal{D}_{2N}^\sigma+\mathcal{E}^\sigma_{2N}}\sqrt{\mathcal{E}_{2N}^\sigma\(\sd{2N}^\sigma +\mathcal{E}^\sigma_{2N}+   \f_{2N}\)}.
\end{align}

Now we assume that $|\al|=4N$. We first estimate the $G^2,G^3$ terms. Since $\al_0\le 2N-1$,  $\pa^\al$ involves at least two spatial derivatives, and so we may write $\al=\beta+(\al-\beta)$ for some $\beta\in\mathbb{N}^2$ with $|\beta|=1$. We then integrate by parts and use the estimates \eqref{p_G_e_00} of Lemma \ref{p_G2N_estimates} to see that
\begin{align}
 \nonumber
 \left|\int_\Omega  \pa^\al u\cdot\pa^\al G^2\right|
&=\left|\int_\Omega  \pa^{\al+\beta} u\cdot\pa^{\al-\beta} G^2\right| \lesssim \norm{\pa^{\al+\beta}u}_{0}\norm{\pa^{\al-\beta} G^2}_{0}
\\&\lesssim \norm{\pa^\al u}_{1}\norm{\bar{\na}_{\ast}^{4N-1} G^2}_{0}\lesssim \sqrt{\mathcal{D}_{2N}^\sigma}\sqrt{  \mathcal{E}_{2N}^\sigma\(\sd{2N}^\sigma +\mathcal{E}^\sigma_{2N}+   \f_{2N}\)}.
\end{align}
Arguing similarly and using  trace theory, we also find that
\begin{align}
 \nonumber
 \left|\int_\Sigma  \pa^\al u\cdot\pa^\al G^3\right|
&=\left|\int_\Sigma  \pa^{\al+\beta}  u\cdot\pa^{\al-\beta} G^3\right|\lesssim \norm{\pa^{\al+\beta} u}_{H^{-1/2}(\Sigma)}\norm{\pa^{\al-\beta} G^3}_{1/2}
\\\nonumber&\lesssim 
\norm{\pa^\al u}_{1}\norm{\bar{\na}_{\ast}^{4N-1} G^3}_{1/2} \lesssim \sqrt{\mathcal{D}_{2N}^\sigma}\sqrt{  \mathcal{E}_{2N}^\sigma\(\sd{2N}^\sigma +\mathcal{E}^\sigma_{2N}+   \f_{2N}\)}.
\end{align}
For the $G^4$ term, we split into two cases: $\al_0\ge 1$ and $\al_0=0$. If $\al_0\ge 1$, then $\pa^\al$ involves at least one temporal derivative, so $\norm{\pa^\al \eta}_{3/2}\le \norm{\dt^{\al_0} \eta}_{4N-2\al_0+3/2}\le \mathcal{D}_{2N}^0$. This together with the estimates \eqref{p_G_e_00} of Lemma \ref{p_G2N_estimates} implies
\begin{align}
 \nonumber
 &\left| \int_{\Sigma_+} \rho_1  g \pa^\al \eta_+ \pa^\al G^4_+
 -\int_\Sigma\sigma\Delta_\ast (\pa^\al \eta) \pa^\al G^4\right| \lesssim\left(\norm{\pa^\al \eta}_{0}+\sigma\norm{\pa^\al \eta}_{3/2}\right)\norm{\pa^\al G^4}_{1/2}
\\&\quad
\lesssim  \sqrt{\mathcal{D}_{2N}^\sigma+\mathcal{E}^\sigma_{2N}}\sqrt{\mathcal{E}_{2N}^\sigma\(\sd{2N}^\sigma+\mathcal{E}^\sigma_{2N} +   \f_{2N}\)}.
\end{align}
If $\al_0=0$, we must resort to the special estimates \eqref{eta es}--\eqref{eta es2} of Lemma \ref{lemma8} to bound, with a use of Cauchy's inequality,
\begin{align}
 &\left| \int_{\Sigma_+} \rho_1  g \pa^\al \eta_+ \pa^\al G^4_+
 -\int_\Sigma\sigma\Delta_\ast \pa^\al \eta \pa^\al G^4\right| \lesssim\sqrt{  \mathcal{E}_{2N}^\sigma}\(\sd{2N}^\sigma +\mathcal{E}^\sigma_{2N}+   \f_{2N}\).
 \end{align}

We now turn back to estimate the $G^1$ term, and we recall that $G^1=G^{1,1}+G^{1,2}$. For the $G^{1,2}$ part, it follows directly from the estimates \eqref{p_G_e_00} of Lemma \ref{p_G2N_estimates} that
\begin{equation}\label{es09}
 \left|\int_\Omega h'(\bar{\rho}) \pa^\al \q \pa^\al G^{1,2} \right|\ls  \norm{\pa^\al \q}_{0}\norm{\pa^\al G^{1,2}}_{0} \lesssim \sqrt{\mathcal{E}_{2N}^\sigma}\sqrt{ \se{2N}^\sigma\(\sd{2N}^\sigma +\se{2N}^\sigma +   \f_{2N}\)}.
\end{equation}
Now for the $G^{1,1}$ term we must split to two cases: $\al_0\ge 1$ and $\al_0=0$. If $\al_0\ge 1$, then by the estimates \eqref{p_G_e_00} of Lemma \ref{p_G2N_estimates}, we have
\begin{equation}
 \left|\int_\Omega h'(\bar{\rho}) \pa^\al \q \pa^\al G^{1,1} \right|\ls  \norm{\pa^\al \q}_{0}\norm{\pa^\al G^{1,1}}_{0}\lesssim \sqrt{\mathcal{D}_{2N}^\sigma}\sqrt{\mathcal{E}_{2N}^\sigma\(\sd{2N}^\sigma+\se{2N}^\sigma +   \f_{2N}\)}.
\end{equation}
If $\al_0=0$, we must resort to the special estimates \eqref{rho es} of Lemma \ref{lemma7} to bound
\begin{equation}\label{gg10}
 \left|\int_\Omega h'(\bar{\rho}) \pa^\al \q \pa^\al G^{1,1} \right| \lesssim \sqrt{\mathcal{D}_{2N}^\sigma+\mathcal{E}^\sigma_{2N}}\sqrt{\mathcal{E}_{2N}^\sigma\(\sd{2N}^\sigma +\mathcal{E}^\sigma_{2N}+   \f_{2N}\)}.
\end{equation}

Finally, for the last term in \eqref{es_00}, by the trace theorem and Cauchy's inequality, since $\alpha_0\le 2N-1$ and $|\alpha|\le 4N$, we have
\begin{align}\label{m6}
\nonumber \int_{\Sigma_-} \rj g\pa^\al \eta_- \pa^\al u_3&\ls\norm{\partial^\alpha\eta_-}_{-1/2}\norm{\partial^\alpha u_{3}}_{H^{1/2}(\Sigma)}
\\&\ls
C_\varepsilon \norm{\partial_t^{\alpha_0}\eta_-}_{4N-2\alpha_0-1/2}^2+\varepsilon\norm{
\partial^\alpha u }_{1}^2
\end{align}
for any $\varepsilon>0$.

In light of \eqref{gg1}--\eqref{m6}, we may now integrate \eqref{es_00} from $0$ to $t$,  apply Korn's inequality, choose $\varepsilon$ sufficiently small, and use  Cauchy's inequality to find that
\begin{align}\label{ls0}
 \nonumber&\norm{\partial^\alpha   q (t)}_0^2+\norm{\partial^\alpha   u (t)}_0^2+\norm{\partial^\alpha  \eta_+ (t)}_0^2+\sigma\norm{\nabla_\ast\partial^\alpha   \eta (t)}_0^2
  +\int_0^t  \norm{\partial^\alpha   u}_1^2\,ds
 \\
&\quad\lesssim  {\mathcal{E}}_{2N}(0)+  \int_0^t  \sqrt{  \mathcal{E}_{2N}^\sigma}\(\sd{2N}^\sigma +\mathcal{E}^\sigma_{2N}+   \f_{2N}\)ds
 +\int_0^t \norm{\partial_t^{\alpha_0}\eta_- }_{4N-2\alpha_0-1/2}^2\,ds.
 \end{align}

Now for $\alpha_0=0$, summing \eqref{ls0} over such $\alpha$ gives
\begin{align}\label{ls1}
 \nonumber&\sum_{\substack{|\alpha|\le 4N  \\ \alpha_0=0}}\(\norm{\partial^\alpha   q (t)}_0^2+\norm{\partial^\alpha   u (t)}_0^2+\norm{\partial^\alpha  \eta_+ (t)}_0^2+\sigma\norm{\nabla_\ast\partial^\alpha   \eta (t)}_0^2\)
  +\int_0^t \sum_{\substack{|\alpha|\le 4N  \\ \alpha_0=0}} \norm{\partial^\alpha   u}_1^2\,ds
 \\
&\quad\lesssim  {\mathcal{E}}_{2N}(0)+  \int_0^t  \sqrt{  \mathcal{E}_{2N}^\sigma}\(\sd{2N}^\sigma +\mathcal{E}^\sigma_{2N}+   \f_{2N}\)ds
 +\int_0^t \norm{ \eta_- }_{4N -1/2}^2\,ds.
 \end{align}
For $ \alpha_0=j$ with $1\le j\le 2N-1$, the kinematic boundary condition,  trace theory, and the estimates \eqref{p_G_e_00} imply that
\begin{equation}\label{ls2}
\begin{split}
\norm{\partial_t^{\alpha_0}\eta_-}_{4N-2\alpha_0-1/2}^2&\le
\norm{\partial_t^{j-1}u_3}_{H^{4N-2j-1/2}(\Sigma_-)}^2+\norm{\partial_t^{j-1}G^4}_{ {4N-2j-1/2}}^2
\\&\le \norm{\nabla_\ast^{4N-2j-1}\partial_t^{j-1}u }_{1} ^2+\norm{\partial_t^{j-1}G^4}_{ {4N-2(j-1)-5/2}}^2
\\&\le \norm{\nabla_\ast^{4N-2(j-1)-3}\partial_t^{j-1}u}_{
1}
^2+\mathcal{E}_{2N}^\sigma\(\sd{2N}^\sigma+\mathcal{E}^\sigma_{2N} +   \f_{2N}\).
\end{split}
\end{equation}
Plugging \eqref{ls2} into \eqref{ls0}, we obtain
\begin{align}\label{ls3}
 \nonumber&\sum_{\substack{|\alpha|\le 4N  \\ \alpha_0=j}}\(\norm{\partial^\alpha   q (t)}_0^2+\norm{\partial^\alpha   u (t)}_0^2+\norm{\partial^\alpha  \eta_+ (t)}_0^2+\sigma\norm{\nabla_\ast\partial^\alpha   \eta (t)}_0^2\)
  +\int_0^t \sum_{\substack{|\alpha|\le 4N  \\ \alpha_0=j}} \norm{\partial^\alpha   u}_1^2\,ds
 \\
&\quad\lesssim  {\mathcal{E}}_{2N}(0)+  \int_0^t  \sqrt{  \mathcal{E}_{2N}^\sigma}\(\sd{2N}^\sigma +\mathcal{E}^\sigma_{2N}+   \f_{2N}\)ds
 +\int_0^t \norm{\nabla_\ast^{4N-2(j-1)-3}\partial_t^{j-1}u}_{
1}
^2\,ds
 \end{align}
since $ \sqrt{  \mathcal{E}_{2N}^\sigma}\le 1$. Consequently, chaining \eqref{ls3} and \eqref{ls1} together leads us to \eqref{energy 2N}.
\end{proof}

 \subsection{Energy evolution for $\eta$ in transport equation}\label{stable3}

Note that the energy estimates of $\eta_-$ are still missing in the energy evolutions presented in Sections \ref{stable1}--\ref{stable2}. We thus need to derive the estimates for $\eta_-$, and this can be done by revisiting the kinematic boundary condition, which is a transport equation for $\eta$:
 \begin{equation}\label{transport}
 \partial_t\eta+u\cdot \nabla_\ast\eta=u_3\quad
 \hbox{in }\Sigma,
 \end{equation}
where $u\cdot \nabla_\ast\eta=u_1\partial_1\eta+u_2\partial_2\eta$.
\begin{prop}\label{transportlemma}
For any $\varepsilon>0$, there exists a constant $C_\varepsilon >0$ such that
\begin{align}\label{transportes}
\nonumber
\sum_{{j}=0}^{2N} \norm{\partial_t^{j}   \eta
(t)}^2_{4N-2{j}} &\lesssim \mathcal{E}_{2N}^\sigma(0) + \int_0^t  \sqrt{  \mathcal{E}_{2N}^\sigma}\(\sd{2N}^\sigma +\mathcal{E}^\sigma_{2N}+   \f_{2N}\)ds
 \\&\quad+\varepsilon\int_0^t  \mathcal{E}_{2N}^\sigma \,ds+C_\varepsilon \int_0^t \norm{\bar{\na}_{\ast}^{4N}   u }_1^2\,ds
\end{align}
and
\begin{equation}\label{transportes2}
 \mathcal{F}_{2N}(t)  \le  \mathcal{F}_{2N}(0)+ C\int_0^t  \sqrt{  \mathcal{E}_{2N}^\sigma}    \f_{2N} ds
 +\varepsilon\int_0^t
\mathcal{F}_{2N}  \,ds+C_\varepsilon \int_0^t\norm{\nabla_\ast^{4N} u }_{1}^2\,ds.
\end{equation}
\end{prop}
\begin{proof}
We first prove the estimates \eqref{transportes}. Recall that we have written $\dt \eta= u_3+G^4$.  Applying $\partial^\alpha$ for $\alpha\in \mathbb{N}^{1+2}$ with $|\alpha|\le 4N$ to this and then taking the inner product with $\partial^\alpha\eta$, we obtain
 \begin{equation}\label{transportalpha}
 \hal \dtt\norm{\partial^\alpha\eta}_0^2 =\int_\Sigma \partial^\alpha\eta \partial^\alpha u_3 + \int_\Sigma
\partial^\alpha\eta \partial^\alpha G^4.
 \end{equation}
For the $G^4 $ term, if $\al_0\ge 1$, then $\pa^\al$ involves at least one temporal derivative,  so the estimates \eqref{p_G_e_00} of Lemma \ref{p_G2N_estimates} imply
\begin{equation}
\left| \int_\Sigma
\partial^\alpha\eta \partial^\alpha G^4\right|
 \lesssim \norm{\pa^\al \eta}_{0} \norm{\pa^\al G^4}_{0}
\lesssim  \sqrt{ \mathcal{E}^\sigma_{2N}}\sqrt{\mathcal{E}_{2N}^\sigma\(\sd{2N}^\sigma+\mathcal{E}^\sigma_{2N} +   \f_{2N}\)}.
\end{equation}
If $\al_0=0$, we use the special estimates \eqref{eta es} of Lemma \ref{lemma8} to estimate
\begin{equation}
\left| \int_\Sigma
\partial^\alpha\eta \partial^\alpha G^4\right|\lesssim\sqrt{  \mathcal{E}_{2N}^\sigma}\(\sd{2N}^\sigma +\mathcal{E}^\sigma_{2N}+   \f_{2N}\).
 \end{equation}
 On the other hand,  trace theory and Cauchy's inequality allow us to bound
\begin{equation}
  \abs{\int_\Sigma \partial^\alpha\eta \partial^\alpha u_3} \ls\norm{\partial^\alpha\eta}_{0}\norm{\partial^\alpha u_{3}}_{H^{0}(\Sigma)}
 \ls \sqrt{ \mathcal{E}^\sigma_{2N}}\norm{\partial^\alpha u}_{1}\ls \varepsilon \mathcal{E}_{2N}^\sigma  +C_\varepsilon  \norm{\bar{\na}_{\ast}^{4N}   u }_1^2
\end{equation}
for any $\varepsilon>0$.
Then the estimate \eqref{transportes} follows.

To prove \eqref{transportes2}, we define the operator $\mathcal{J}=\sqrt{1-\Delta_\ast}$. We apply $\mathcal{J}^{4N+1/2}$  to \eqref{transport}, multiply the resulting equation by $\mathcal{J}^{4N+1/2}\eta$, and then integrate  over $ \Sigma$; using the standard commutator estimate, Sobolev embeddings on $\Sigma$, and trace theory, we find that
 \begin{align}
\nonumber  &\frac{1}{2}\frac{d}{dt}\mathcal{F}_{2N}   =
 -\frac{1}{2}\int_\Sigma u\cdot \nabla_\ast|\mathcal{J}^{4N+1/2}\eta|^2+\int_\Sigma
   \left(\mathcal{J}^{4N+1/2} u_3-\left[\mathcal{J}^{4N+1/2},u\right]\cdot\nabla_\ast\eta\right) \mathcal{J}^{4N+1/2}\eta
  \\\nonumber & \quad=
 \frac{1}{2}\int_\Sigma (\partial_1u_1+\partial_2u_2)|\mathcal{J}^{4N+1/2}\eta|^2+
 \int_\Sigma  \left(\mathcal{J}^{4N+1/2} u_3-\left[\mathcal{J}^{4N+1/2},u\right]\cdot\nabla_\ast\eta\right) \mathcal{J}^{4N+1/2}\eta
    \\\nonumber&\quad\lesssim\norm{\nabla_\ast u}_{L^\infty(\Sigma)}\norm{\mathcal{J}^{4N+1/2}\eta}_0^2
+\left(\norm{\mathcal{J}^{4N+1/2} u_3}_{L^2(\Sigma)}+\norm{\nabla_\ast
u}_{L^\infty(\Sigma)}\norm{\mathcal{J}^{4N-1/2}\nabla_\ast\eta}_{0}\right.
 \\\nonumber&\qquad+\left.\norm{\mathcal{J}^{4N+1/2}
u}_{0}\norm{\nabla_\ast\eta}_{L^\infty(\Sigma)}\right)\norm{\mathcal{J}^{4N+1/2}\eta}_{0}
 \\\nonumber&\quad\lesssim\norm{ u}_{H^3(\Sigma)}\norm{ \eta}_{4N+1/2}^2 + \norm{\mathcal{J}^{4N}
u }_{1}\(1+\norm{\eta}_3\)
\norm{\eta}_{4N+1/2}  \\&\quad\lesssim
\sqrt{\mathcal{E}^\sigma_{2N}} \mathcal{F}_{2N} +\sqrt{\mathcal{F}_{2N}}\norm{\nabla_\ast^{4N} u }_{1}^2.
 \end{align}
Then the estimate \eqref{transportes2} follows by using Cauchy's inequality.
\end{proof}

\subsection{The evolution of energies controlling $\pa_3\q$} \label{sec_aux}
The energy evolutions presented in Sections \ref{stable1}--\ref{stable3} are not enough to get the full energy estimates by applying the Stokes regularity as in the incompressible case \cite{WT,WTK} since we have not controlled $\diverge u$.  Motivated by Matsumura and Nishida \cite{MN83}, to control $\diverge u$ we introduce the material derivative of $q$ in our coordinates:
\begin{equation}\label{dt}
\mathcal{Q}:= \pa_t\q- K \p_t\theta \pa_3  \q+u_j \mathcal{A}_{jk}\pa_k  \q = \pa_t\q-G^{1,1}=-\diverge(\bar\rho u)+G^{1,2}.
\end{equation}
We may then derive the following from \eqref{ns_perturb}:
\begin{equation}\label{eeqq}
\begin{split}
& \pa_3\mathcal{Q} +\bar{\rho}\pa_3(\diverge   u)=\pa_3 G^{1,2}-\diverge (\pa_3\bar{\rho} u)-\pa_3 \bar{\rho}\pa_3u_3,
\\
 &  \bar\rho \partial_t    u_3   +\bar\rho  \pa_3(h'(\bar{\rho})\q)     -\mu \Delta u_3-(\mu/3+\mu')\pa_3(\diverge u) =G^2_3.
  \end{split}
\end{equation}
By eliminating $\pa_{33} u_3$ from the equations \eqref{eeqq}, we obtain
\begin{align}\label{density 0}
\nonumber
& \displaystyle
\frac{ 4\mu/3+\mu'  }{h'(\bar\rho)\bar\rho^2}
\pa_3\left(h'(\bar\rho)\mathcal{Q} \right)+ \pa_3(h'(\bar\rho) \q)
 = \frac{ 4\mu/3+\mu'  }{\bar\rho^2} \pa_3  G^{1,2} +\frac{ 1  }{\bar\rho } G^2_3+\frac{ 4\mu/3+\mu'  }{h'(\bar\rho)\bar\rho^2}
\pa_3h'(\bar\rho) \mathcal{Q}
\\& \qquad   -  \partial_t    u_3 -\frac{ 4\mu/3+\mu'  }{\bar\rho^2}(\diverge (\pa_3\bar{\rho} u)+\pa_3 \bar{\rho}\pa_3u_3)+\frac{ \mu }{\bar\rho }(\pa_{11} u _3+\pa_{22}u_3- \pa_{31} u_1-\pa_{32}  u_2 ) .
\end{align}
In the light of \eqref{dt}, we can view \eqref{density 0} as the evolution equation for $\pa_3\q$.

We now present the energy evolution of $\p_3\q$.
\begin{prop}\label{i_rho_evolution 2N}
For $0\le  j\le 2N-1$ and   $0\le k\le 4N-2 j-1$, we have
\begin{align}\label{density es2N}
&\nonumber \sum_{k'\le k}   \norm{
  \nab_{\ast}^{4N-2j-k'-1} \pa_3^{k'+1}\pa_t^ j(h'(\bar\rho) \q)}_0^2 \\
&\nonumber\quad + \int_0^t\sum_{k'\le k} \norm{\nab_{\ast}^{4N-2j-k'-1}  \pa_3^{k'+1}\pa_t^j \left( h'(\bar\rho)\q \right)}_0^2
+ \sum_{k'\le k} \norm{\nab_{\ast}^{4N-2j-k'-1}  \pa_3^{k'+1}\pa_t^j  \mathcal{Q}}_0^2 ds
\\&\nonumber\quad
\ls \mathcal{E}_{2N}^\sigma(0)+\int_0^t \norm{\pa_t^{ j+1} u  }_{4N-2 j-1}^2 +\norm{ \bar{\na}_{\ast}^{4N}  u}_1^2
+ \sum_{k'\le k} \norm{\na_{\ast}^{4N-2 j-k'} \pa_t^j    u }_{k'+1}^2  ds
\\  &\qquad+\int_0^t \sqrt{  \mathcal{E}_{2N}^\sigma}\(\sd{2N}^\sigma +\mathcal{E}^\sigma_{2N}+   \f_{2N}\) ds.
\end{align}

 \end{prop}
\begin{proof}
We first fix $0\le  j\le 2N-1$ and then take $0 \le k\le 4N-2 j-1$ and $0\le k'\le k$.  Let $\alpha\in \mathbb{N}^{2}$ so that $|\alpha|\le 4N-2 j-1-k' $. Applying $\partial^\alpha\pa_3^{k'}\pa_t^ j$  to $\eqref{density 0}$,  multiplying the resulting equation by $\partial^\alpha\pa_3^{k'+1}\pa_t^ j(h'(\bar\rho) \q)+\partial^\alpha\pa_3^{k'+1}\pa_t^ j(h'(\bar\rho)\mathcal{Q})$, and then integrating over $\Omega$, we obtain
\begin{equation}\label{dtdensity1}
 I + II + III = IV,
\end{equation}
where
\begin{align}
 &I = \int_\Omega  \partial^\alpha\pa_3^{k'}\pa_t^ j\left(\frac{ 4\mu/3+\mu'  }{h'(\bar\rho)\bar\rho^2}
\pa_3\left(h'(\bar\rho)\mathcal{Q} \right)\right)\partial^\alpha\pa_3^{k'+1}\pa_t^ j(h'(\bar\rho) \q),
\\
& II = \int_\Omega \partial^\alpha\pa_3^{k'}\pa_t^ j\left(\frac{ 4\mu/3+\mu'  }{h'(\bar\rho)\bar\rho^2}
\pa_3\left(h'(\bar\rho)\mathcal{Q} \right)\right)\partial^\alpha\pa_3^{k'+1}\pa_t^ j\left(h'(\bar\rho) \mathcal{Q}\right),
\\
& III = \int_\Omega  \abs{\partial^\alpha\pa_3^{k'+1}\pa_t^ j(h'(\bar\rho) \q)}^2+\int_\Omega  \partial^\alpha\pa_3^{k'+1}\pa_t^ j(h'(\bar\rho) \q)\partial^\alpha\pa_3^{k'+1}\pa_t^ j\left(h'(\bar\rho)\mathcal{Q}\right),
\\
& IV = \int_\Omega \left\{\partial^\alpha\pa_3^{k'+1}\pa_t^ j(h'(\bar\rho) \q)+\partial^\alpha\pa_3^{k'+1}\pa_t^ j\left(h'(\bar\rho)\mathcal{Q} \right)\right\}\nonumber
\\
 &\qquad\times\partial^\alpha\pa_3^{k'}\pa_t^ j\left\{\frac{ 4\mu/3+\mu'  }{\bar\rho^2} \pa_3  G^{1,2} +\frac{ 1  }{\bar\rho } G^2_3+\frac{ 4\mu/3+\mu'  }{h'(\bar\rho)\bar\rho^2}
\pa_3h'(\bar\rho) \mathcal{Q} -  \partial_t    u_3\right.\nonumber
\\&\qquad\qquad
 \left. -\frac{ 4\mu/3+\mu'  }{\bar\rho^2}(\diverge (\pa_3\bar{\rho} u)+\pa_3 \bar{\rho}\pa_3u_3)+\frac{ \mu }{\bar\rho }(\pa_{11} u _3+\pa_{22}u_3- \pa_{31} u_1-\pa_{32}  u_2 ) \right\}.
\end{align}

We will now estimate $I, II, III, IV$.  First, using the Cauchy-Schwarz inequality, we may easily estimate
\begin{align}\label{rev_1}
 IV \ls &\left\{\norm{\partial^\alpha\pa_3^{k'+1}\pa_t^ j(h'(\bar\rho) \q)}_0 +\sum_{k''\le k'} \norm{\partial^\alpha\pa_3^{k''+1}\pa_t^ j\mathcal{Q}}_0\right\}\nonumber
\\
 & \times\left\{\norm{\pa_t^{ j }  G^{1,2}}_{4N-2 j} +\norm{\pa_t^{ j }  G^2}_{4N-2 j-1}+\sum_{k''\le k'}\norm{\partial^\alpha\pa_3^{k''}\pa_t^ j \mathcal{Q}}_0 \right.\nonumber
\\
 &\qquad  \left.+\norm{\pa_t^{ j+1} u  }_{4N-2 j-1} +\sum_{k''\le k'}  \norm{\partial^\alpha\pa_3^{k''}\pa_t^ j u}_1+\norm{\partial^\alpha\pa_3^{k''}\na_\ast\na\pa_t^ j u}_0  \right\}.
\end{align}
For the last term in $III$ we recall  the definition of $\mathcal{Q}$ from \eqref{dt} in order to  rewrite
\begin{align}\label{dens0}
&\nonumber \int_\Omega  \partial^\alpha\pa_3^{k'+1}\pa_t^ j(h'(\bar\rho) \q)\partial^\alpha\pa_3^{k'+1}\pa_t^ j\left(h'(\bar\rho)\mathcal{Q}\right)
\\&\nonumber\quad=\int_\Omega  \partial^\alpha\pa_3^{k'+1}\pa_t^ j(h'(\bar\rho) \q)\partial^\alpha\pa_3^{k'+1}\pa_t^ j
\left(h'(\bar\rho)(\dt\q-G^{1,1})\right)
\\&\quad=\frac{1}{2}\frac{d}{dt} \int_\Omega \abs {\partial^\alpha\pa_3^{k'+1}\pa_t^ j(h'(\bar\rho) \q)}^2
-\int_\Omega  \partial^\alpha\pa_3^{k'+1}\pa_t^ j(h'(\bar\rho) \q)\partial^\alpha\pa_3^{k'+1}\pa_t^ j
\left(h'(\bar\rho) G^{1,1} \right).
\end{align}
For $II$ we estimate by expanding with the Leibniz rule:
\begin{equation} \label{dens1}
II \ge \int_\Omega \frac{ 4\mu/3+\mu'  }{h'(\bar\rho)\bar\rho^2}\abs{\partial^\alpha\pa_3^{k'+1}\pa_t^ j \left(h'(\bar\rho)\mathcal{Q} \right)}^2
-C\norm{\partial^\alpha\pa_3^{k'+1}\pa_t^ j \left( h'(\bar\rho)\mathcal{Q} \right)}_0\sum_{k''\le k'}\norm{\partial^\alpha\pa_3^{k''}\pa_t^ j \mathcal{Q}}_0.
 \end{equation}
For $I$, we have
\begin{align} \label{dens2}
\nonumber I  &\ge \int_\Omega \frac{ 4\mu/3+\mu'  }{h'(\bar\rho)\bar\rho^2}\partial^\alpha\pa_3^{k'+1}\pa_t^ j
\left(h'(\bar\rho)\mathcal{Q} \right) \partial^\alpha\pa_3^{k'+1}\pa_t^ j(h'(\bar\rho) \q)
\\\nonumber &\quad
-C\norm{\partial^\alpha\pa_3^{k'+1}\pa_t^ j \left( h'(\bar\rho)\q \right)}_0\sum_{k''\le k'}\norm{\partial^\alpha\pa_3^{k''}\pa_t^ j \mathcal{Q}}_0
\\\nonumber &= \frac{1}{2}\frac{d}{dt} \int_\Omega  \frac{ 4\mu/3+\mu'  }{h'(\bar\rho)\bar\rho^2}
\abs{\partial^\alpha\pa_3^{k'+1}\pa_t^ j(h'(\bar\rho) \q)}^2
\\&\nonumber \quad
-\int_\Omega \frac{ 4\mu/3+\mu'  }{h'(\bar\rho)\bar\rho^2} \partial^\alpha\pa_3^{k'+1}\pa_t^ j(h'(\bar\rho) \q)\partial^\alpha\pa_3^{k'+1}\pa_t^ j
\left(h'(\bar\rho) G^{1,1} \right)
\\&\quad
-C\norm{\partial^\alpha\pa_3^{k'+1}\pa_t^ j \left( h'(\bar\rho)\q \right)}_0\sum_{k''\le k'}\norm{\partial^\alpha\pa_3^{k''}\pa_t^ j \mathcal{Q}}_0.
 \end{align}

Combining the estimates \eqref{rev_1}--\eqref{dens2} with \eqref{dtdensity1},  applying Cauchy's inequality in order to absorb the term $\norm{\partial^\alpha\pa_3^{k'+1}\pa_t^ j \left( h'(\bar\rho)\q \right)}_0$  onto the left, and then integrating in time from $0$ to $t$, we arrive at the inequality
\begin{align}\label{dens_10}
\nonumber& \norm{
 \partial^\alpha\pa_3^{k'+1}\pa_t^ j(h'(\bar\rho) \q)}_0^2
 +\int_0^t \norm{\partial^\alpha\pa_3^{k'+1}\pa_t^ j \left( h'(\bar\rho)\q \right)}_0^2+  \norm{\partial^\alpha\pa_3^{k'+1}\pa_t^ j \left( h'(\bar\rho)\mathcal{Q}\right)}_0^2 ds \\
&\;\lesssim \mathcal{E}_{2N}^\sigma(0) +\int_0^t\sum_{k''\le k'}\norm{\partial^\alpha\pa_3^{k''}\pa_t^ j \mathcal{Q}}_0^2 ds\\
&\nonumber\quad+\int_0^t \abs{\int_\Omega \left(1+\frac{ 4\mu/3+\mu'  }{h'(\bar\rho)\bar\rho^2}\right)
\partial^\alpha\pa_3^{k'+1}\pa_t^ j(h'(\bar\rho) \q)\partial^\alpha\pa_3^{k'+1}\pa_t^ j \left(h'(\bar\rho) G^{1,1} \right)}ds
\\&\nonumber\quad
+\int_0^t \norm{\pa_t^{ j }  G^{1,2}}_{4N-2 j}^2+\norm{\pa_t^{ j }  G^2}_{4N-2 j-1}^2
+ \norm{\pa_t^{ j+1} u  }_{4N-2 j-1}^2
+ \norm{\na_{\ast}^{4N-2 j-k'} \pa_t^ j    u }_{k'+1}^2 ds.
\end{align}
Owing to the Leibniz rule and the properties of $\bar{\rho}$, we may estimate
\begin{multline}
 \ns{\p_3^{k'+1} \p^\al \dt^j \mathcal{Q}}_0 \ls  \ns{h'(\bar{\rho}) \p_3^{k'+1} \p^\al \dt^j \mathcal{Q}}_0 \\
\ls  \ns{ \p_3^{k'+1} \p^\al \dt^j (h'(\bar{\rho})\mathcal{Q})}_0 + \sum_{k'' \le k'} \ns{\p_3^{k''} \p^\al \dt^j \mathcal{Q}}_0.
\end{multline}
Combining this with \eqref{dens_10} and summing over all $\alpha$ with $\abs{\alpha}\le 4N-2 j-1-k'$, we deduce that
\begin{multline}\label{Cell}
 \norm{
 \nab_{\ast}^{4N-2j-k'-1}
 \pa_3^{k'+1}\pa_t^ j(h'(\bar\rho) \q)}_0^2 \\
 +\int_0^t \norm{\nab_{\ast}^{4N-2j-k'-1} \pa_3^{k'+1}\pa_t^j \left( h'(\bar\rho)\q \right)}_0^2
+   \norm{\nab_{\ast}^{4N-2j-k'-1} \pa_3^{k'+1}\pa_t^j  \mathcal{Q}}_0^2  ds \\
\lesssim \mathcal{E}_{2N}^\sigma(0) + \int_0^t \sum_{k''\le k'}\norm{\nab_{\ast}^{4N-2j-k'-1} \pa_3^{k''}\pa_t^ j \mathcal{Q}}_0^2 ds
\\
+\int_0^t \sum_{\abs{\alpha}\le 4N-2 j-1-k'} \abs{\int_\Omega \left(1+\frac{ 4\mu/3+\mu'  }{h'(\bar\rho)\bar\rho^2}\right)
\p^\al \pa_3^{k'+1}\pa_t^ j(h'(\bar\rho) \q)\partial^\alpha\pa_3^{k'+1}\pa_t^ j \left(h'(\bar\rho) G^{1,1} \right)} ds
\\
+  \int_0^t \norm{\pa_t^{ j }  G^{1,2}}_{4N-2 j}^2+\norm{\pa_t^{ j }  G^2}_{4N-2 j-1}^2
+ \norm{\pa_t^{ j+1} u  }_{4N-2 j-1}^2
+ \norm{\na_{\ast}^{4N-2 j-k'} \pa_t^ j    u }_{k'+1}^2 ds
\end{multline}
for each $0 \le k' \le k$.

Finally, we will estimate the nonlinear terms in the right hand side of \eqref{Cell}. We use the estimates \eqref{p_G_e_00} of Lemma \ref{p_G2N_estimates} to estimate, for $0\le  j\le 2N-1$,
\begin{equation}\label{rhoes2}
\norm{\pa_t^{ j }  G^{1,2}}_{4N-2 j}^2+\norm{\pa_t^{ j }  G^2}_{4N-2 j-1}^2\lesssim   {\mathcal{E}_{2N}^\sigma }\(\mathcal{D}_{2N}^\sigma+\mathcal{E}_{2N}^\sigma+  \f_{2N}\) .
\end{equation}
Then we use Lemma \ref{G11_weighted} to bound
\begin{multline}\label{rhoes25}
  \abs{\int_\Omega \left(1+\frac{ 4\mu/3+\mu'  }{h'(\bar\rho)\bar\rho^2}\right)
  \partial^\alpha\pa_3^{k'+1}\pa_t^ j(h'(\bar\rho) \q)\partial^\alpha\pa_3^{k'+1}\pa_t^ j \left(h'(\bar\rho) G^{1,1} \right)} \\
    \ls   \sqrt{ \mathcal{D}_{2N}^\sigma+\mathcal{E}_{2N}^\sigma }\sqrt{ {\mathcal{E}_{2N}^\sigma }\(\mathcal{D}_{2N}^\sigma+\mathcal{E}_{2N}^\sigma+  \f_{2N}\) }.
\end{multline}
Plugging the nonlinear estimates \eqref{rhoes2} and \eqref{rhoes25}  into \eqref{Cell} then yields that every $0 \le k' \le k$,
\begin{align}\label{Cell1}
\nonumber& \norm{
 \nab_{\ast}^{4N-2j-k'-1}
 \pa_3^{k'+1}\pa_t^ j(h'(\bar\rho) \q)}_0^2
  \\
&\nonumber\quad +\int_0^t \norm{\nab_{\ast}^{4N-2j-k'-1} \pa_3^{k'+1}\pa_t^j \left( h'(\bar\rho)\q \right)}_0^2 +  \norm{\nab_{\ast}^{4N-2j-k'-1} \pa_3^{k'+1}\pa_t^j  \mathcal{Q}}_0^2ds\\
&\nonumber\quad \lesssim \mathcal{E}_{2N}^\sigma(0) + \int_0^t \sum_{k''\le k'}\norm{\nab_{\ast}^{4N-2j-k'-1} \pa_3^{k''}\pa_t^ j \mathcal{Q}}_0^2 ds
\\
&\quad+ \int_0^t \norm{\pa_t^{ j+1} u  }_{4N-2 j-1}^2
+ \norm{\na_{\ast}^{4N-2 j-k'} \pa_t^ j    u }_{k'+1}^2+\sqrt{  \mathcal{E}_{2N}^\sigma}\(\sd{2N}^\sigma +\mathcal{E}^\sigma_{2N}+   \f_{2N}\) ds
\end{align}
since $ \sqrt{  \mathcal{E}_{2N}^\sigma}\le 1$. We recall the notation $\mathcal{Q}$ in \eqref{dt}. We may use the estimates \eqref{p_G_e_00} of Lemma \ref{p_G2N_estimates} to obtain the bound
\begin{equation}\label{drho}
\norm{ \bar{\na}_{\ast}^{4N}  \mathcal{Q}}_0^2 \ls \norm{\bar{\na}_{\ast}^{4N}\diverge(\bar\rho u)}_0^2 +
\norm{\bar{\na}_{\ast}^{4N} G^{1,2}}_0^2
\lesssim \norm{ \bar{\na}_{\ast}^{4N}  u}_1^2  +     \mathcal{E}_{2N}^\sigma \(\mathcal{D}_{2N}^\sigma+\mathcal{E}^\sigma_{2N}+  \f_{2N}\).
\end{equation}
Then a standard induction argument on \eqref{Cell1}, together with \eqref{drho}, yield \eqref{density es2N}.
 \end{proof}

Note that the novelty of Proposition \ref{i_rho_evolution 2N} is twofold. First, it presents the energy estimates of $\pa_3q$. Second, it provides the dissipation estimates of $\pa_3\mathcal{Q}$ and hence $\pa_3\diverge(\bar\rho u)$ by \eqref{dt}. These are crucial for improving the horizontal energy and dissipation estimates derived in the previous section into the full ones in later sections, respectively.

\subsection{Combined energy evolution estimates}\label{sec_combo}

Now we chain the results in Sections \ref{stable1}, \ref{stable2} and \ref{sec_aux} with the elliptic regularity theory of a certain Stokes problem into an intermediate energy-dissipation estimate.

We first derive the elliptic estimates. We deduce from \eqref{ns_perturb} that
\begin{equation}
\begin{split}
 &\diverge ( \bar{\rho}   u)=G^{1,2}-\mathcal{Q},
 \\
 & -\frac{\mu}{\bar{\rho}}\Delta u- \frac{\mu/3+\mu'}{\bar{\rho}} \na \diverge u+ \nabla   \left(h'(\bar{\rho})\q\right) =\frac{1}{\bar{\rho}}G^2-   \partial_t    u.
  \end{split}
\end{equation}
Direct calculations give the form of the Stokes problem we shall use:
\begin{equation}\label{stoke problem}
\begin{cases}
 \displaystyle-\mu\Delta\left (\frac{u}{\bar{\rho}} \right)+ \nabla  \left(h'(\bar{\rho})\q\right)
 =\frac{1}{\bar{\rho}}G^2-   \partial_t    u-\mu\left(2\pa_3\left (\frac{1}{\bar{\rho}} \right)\pa_3u+ \pa_{33}\left (\frac{1}{\bar{\rho}} \right) u\right)
 \\\qquad\qquad\qquad\qquad\qquad\quad\ \displaystyle +\frac{\mu/3+\mu'}{\bar{\rho}}\na \left (\displaystyle\frac{1}{\bar{\rho}} \left(G^{1,2}-\displaystyle\mathcal{Q}-\pa_3\bar{\rho}u_3\right)\right)&\text{in }\Omega_\pm\\\displaystyle
 \diverge \left (\frac{u}{\bar{\rho}} \right)=\frac{1}{\bar{\rho}^2}\left(G^{1,2}-\mathcal{Q}-2\pa_3\bar{\rho}u_3\right)&\text{in }\Omega_\pm
\\u=u&\text{on }\pa\Omega_\pm.
  \end{cases}
\end{equation}

We now prove the Stokes estimates.

\begin{lem}\label{lemmau2N}
Fix $0\le  j\le 2N-1$. Then for any $1\le k\le 4N-2 j$,
\begin{align}\label{u es2N}
\nonumber
 &\norm{ \na_{\ast}^{4N-2 j-k}\pa_t^ j  u }_{k+1}^2 + \norm{\na \na_{\ast}^{4N-2 j-k}\pa_t^ j\left(h'(\bar{\rho})\q\right) }_{ k-1}^2
\\& \quad\lesssim\norm{\pa_t^{ j+1} u  }_{4N-2 j-1}^2+ \norm{\na_{\ast}^{4N-2 j-k} \pa_t^{ j } \mathcal{Q} }_{k}^2+ \norm{\bar{\na}_{\ast}^{4N}   u }_1^2
+   \mathcal{E}_{2N}^\sigma  \(\mathcal{D}_{2N}^\sigma+\mathcal{E}_{2N}^\sigma+   \f_{2N}\).
 \end{align}
\end{lem}
\begin{proof}
We first fix $0\le  j\le 2N-1$ and then take $1 \le k\le 4N-2 j$. Let $\alpha\in \mathbb{N}^{2}$ such that $|\alpha|\le 4N-2 j-k$.  We apply $\partial^\alpha\pa_t^ j$  to the equations $\eqref{stoke problem}$ in $\Omega_\pm$ respectively;  then the elliptic estimates of Lemma \ref{stokes reg} with $r=k'+1\ge2$ for any $1\le k'\le k$ and  trace theory allow us to obtain the bounds
\begin{align}\label{u claim es}
\nonumber
&\norm{ \pa^\al\pa_t^ j  u }_{k'+1}^2+ \norm{\na \pa^\al\pa_t^ j   \left(h'(\bar{\rho})\q\right)  }_{k'-1}^2
 \lesssim \norm{ \pa^\al\pa_t^ j \left (\frac{u}{\bar{\rho}} \right) }_{k'+1}^2+ \norm{\na \pa^\al\pa_t^ j  \left(h'(\bar{\rho})\q\right) }_{k'-1}^2
\\\nonumber
&\quad\lesssim \norm{\pa^\al\pa_t^{ j} G^2  }_{k'-1}^2+\norm{\pa^\al\pa_t^{ j+1} u  }_{k'-1}^2+\norm{ \pa^\al\pa_t^ j  u }_{k'}^2
\\\nonumber
&\qquad+ \norm{\pa^\al\pa_t^{ j} G^{1,2}  }_{k'}^2+\norm{\pa^\al \pa_t^{ j } \mathcal{Q} }_{k'}^2+ \norm{\pa^\al\pa_t^ j u }_{H^{k'+1/2}(\Sigma)}^2
\\\nonumber
&\quad\lesssim \norm{\pa^\al\pa_t^{ j} G^{1,2}  }_{k}^2+\norm{\pa^\al\pa_t^{ j} G^2  }_{k-1}^2+\norm{\pa^\al\pa_t^{ j+1} u  }_{k-1}^2 +\norm{\pa^\al \pa_t^{ j } \mathcal{Q} }_{k}^2+\norm{ \pa^\al\pa_t^ j  u }_{k'}^2
\\\nonumber
& \qquad+ \norm{\na_{\ast}^{k'}\pa^\al\pa_t^ j u }_{H^{1/2}(\Sigma)}^2
\\\nonumber
&\quad\lesssim \norm{\pa_t^{ j} G^{1,2}  }_{4N-2 j}^2  +\norm{ \pa_t^{ j} G^2  }_{4N-2 j-1}^2+\norm{\pa_t^{ j+1} u  }_{4N-2 j-1}^2+ \norm{\na_{\ast}^{4N-2 j-k} \pa_t^{ j } \mathcal{Q} }_{k}^2
\\
&\qquad+ \norm{\bar{\na}_{\ast}^{\,\,4N}   u }_1^2 +\norm{ \pa^\al\pa_t^ j  u }_{k'}^2.
\end{align}

A simple induction based on the above yields that
\begin{align}\label{u claim}
\nonumber
\norm{\pa^\al\pa_t^ j  u }_{k+1}^2&+ \norm{\na \pa^\al\pa_t^ j   \q }_{ k-1}^2
    \lesssim \norm{\pa_t^{ j} G^{1,2}  }_{4N-2 j}^2
 +\norm{ \pa_t^{ j} G^2  }_{4N-2 j-1}^2
\\&\quad+\norm{\pa_t^{ j+1} u  }_{4N-2 j-1}^2+ \norm{\na_{\ast}^{4N-2 j-k} \pa_t^{ j } \mathcal{Q} }_{k}^2 + \norm{\bar{\na}_{\ast}^{4N}   u }_1^2.
\end{align}

Finally, we use the estimates \eqref{p_G_e_00} of Lemma \ref{p_G2N_estimates} to have
\begin{equation}\label{u claim es2}
\norm{\pa_t^{ j} G^{1,2}  }_{4N-2 j}^2
 +\norm{ \pa_t^{ j} G^2  }_{4N-2 j-1}^2\ls \mathcal{E}_{2N}^\sigma  \(\mathcal{D}_{2N}^\sigma+\mathcal{E}_{2N}^\sigma+   \f_{2N}\)
 \end{equation}
We then  sum \eqref{u claim} over such $|\alpha|\le 4N-2 j-k$ to conclude \eqref{u es2N}.
\end{proof}

We will now combine the energy evolution estimates of Sections \ref{stable1}--\ref{stable2} with the $\p_3 q$ estimate of Section \ref{sec_aux} and the estimates of Lemma \ref{lemmau2N}. The full dissipation estimates of $u$ will be obtained, and also some estimates of $\q$ will be improved along the way. To do so, we first introduce some notation. We write
\begin{equation}\label{E_frak}
\mathfrak{E}_{2N}^\sigma =
 \norm{\bar{\nab}_{\ast}^{4N} u}_0^2 +\norm{ \bar{\nab}_{\ast}^{4N} \q}_0^2
+   \sum_{j=0}^{2N} \ns{\dt^j \eta_+}_{4N-2j} + \sigma \sum_{j=0}^{2N} \ns{\nabla_\ast\dt^j \eta}_{4N-2j}
\end{equation}
and
\begin{equation}\label{D_frak}
\mathfrak{D}_{2N} = \norm{ \bar{\nab}_{\ast}^{4N} u}_1^2
\end{equation}
for the various terms appearing in Propositions  \ref{i_temporal_evolution 2N} and \ref{i_spatial_evolution 2N}.  Similarly, for integers $0 \le j \le 2N-1$ and $0 \le k \le 4N-2j -1$ we write
\begin{equation}\label{A_frak}
 \mathfrak{A}_{2N}^{j,k} := \sum_{k'=0}^{ k}  \norm{
  \nab_{\ast}^{4N-2j-k'-1} \pa_3^{k'+1}\pa_t^ j(h'(\bar\rho) \q)}_0^2
\end{equation}
and
\begin{align}\label{B_frak}
\nonumber \mathfrak{B}_{2N}^{j,k} := &  \sum_{k'=0}^{k} \norm{\nab_{\ast}^{4N-2j-k'-1}  \pa_3^{k'+1}\pa_t^j \left( h'(\bar\rho)\q \right)}_0^2
+ \sum_{k'=1}^k \ns{ \nab \nab_\ast^{4N-2j-k'} \dt^j(h'(\bar{\rho}) \q)}_{k'-1}
\\ &
+ \norm{ \bar{\na}_{\ast}^{4N}  \mathcal{Q}}_0^2  + \sum_{k'=0}^{k} \norm{\nab_{\ast}^{4N-2j-k'-1}  \pa_3^{k'+1}\pa_t^j  \mathcal{Q}}_0^2,
\end{align}
where $\mathcal{Q}$ is defined in \eqref{dt}. In addition, we introduce the following intermediate energies:
\begin{equation}\label{Enn}
\bar{\mathcal{E}}_{2N}^\sigma := \norm{ {\bar\na}_{\ast}^{4N} u}_0^2  + \sum_{j=0}^{2N} \norm{ \pa_t^j \q}_{4N-2j}^2
+  \sum_{j=0}^{2N} \ns{\dt^j \eta_+}_{4N-2j} + \sigma \sum_{j=0}^{2N} \ns{\nabla_\ast\dt^j \eta}_{4N-2j}
\end{equation}
and
\begin{equation}\label{Dnn}
\bar{\mathcal{D}}_{2N}:= \sum_{j=0}^{2N} \ns{ \dt^j u}_{4N-2j+1} + \sum_{j=0}^{n-1} \ns{\nab \dt^j (h'(\bar{\rho}) \q)}_{4N-2j-1}.
\end{equation}
The rest of the section is devoted to the derivation of the energy bounds for $\bar{\mathcal{E}}_{2N}^\sigma$ and $\bar{\mathcal{D}}_{2N}$ based on the evolution equations for $\mathfrak{E}_{2N}^\sigma, \mathfrak{D}_{2N}, \mathfrak{A}_{2N}^{j,k},  \mathfrak{B}_{2N}^{j,k}$.

\begin{prop}\label{boostrap2N}
Let $\bar{\mathcal{E}}_{2N}^\sigma$ and $\bar{\mathcal{D}}_{2N}$ be as defined by \eqref{Enn} and \eqref{Dnn}.  Then we have
\begin{equation}\label{ebar2N}
\bar{\mathcal{E}}_{2N}^\sigma(t) + \int_0^t \bar{\mathcal{D}}_{2N} ds
\\
\ls{\mathcal{E}}_{2N}^\sigma(0)+ \int_0^t  \sqrt{\mathcal{E}_{2N}^\sigma }\(\mathcal{D}_{2N}^\sigma+ \mathcal{E}_{2N}^\sigma+ \f_{2N}\)ds+ \int_0^t \norm{\eta_-}_{4N-1/2}^2 ds.
\end{equation}
\end{prop}
\begin{proof}

First, we sum the result of Proposition  \ref{i_temporal_evolution 2N} with the result of Proposition
 \ref{i_spatial_evolution 2N}; this yields the estimate
\begin{equation}\label{boo_1}
 \Ef(t) + \int_0^t\Df  ds\ls  {\mathcal{E}}_{2N}^\sigma(0)+ \int_0^t  \sqrt{\mathcal{E}_{2N}^\sigma }\(\mathcal{D}_{2N}^\sigma+ \mathcal{E}_{2N}^\sigma+ \f_{2N}\)ds+ \int_0^t \norm{\eta_-}_{4N-1/2}^2ds,
\end{equation}
where $\Ef$ and $\Df$ are as defined in \eqref{E_frak} and \eqref{D_frak}.

Next, for $0\le  j\le 2N-1$ and  $0\le k\le 4N-2 j-1$, we may combine the results of Proposition \ref{i_rho_evolution 2N} and Lemma \ref{lemmau2N} (summed over  $1 \le k' \le k$) with \eqref{drho} to see that
\begin{align}\label{boo_3}
\nonumber   \Af +  \int_0^t\Bf  ds &\ls  {\mathcal{E}}_{2N}^\sigma(0)+\int_0^t
 \norm{\pa_t^{ j+1} u  }_{4N-2 j-1}^2
+ \sum_{k'=1}^{k}\norm{\na_{\ast}^{4N-2 j-k'} \pa_t^{ j } \mathcal{Q} }_{k'}^2 ds
  \\&\quad+ \int_0^t \Df +  \sqrt{\mathcal{E}_{2N}^\sigma }\(\mathcal{D}_{2N}^\sigma+\mathcal{E}_{2N}^\sigma+  \f_{2N}\) ds,
\end{align}
where $\Af$ and $\Bf$ are given in \eqref{A_frak} and \eqref{B_frak}. Note that Lemma \ref{lemmau2N} used here is to control the term $\norm{\na_{\ast}^{4N-2 j-k'} \pa_t^j    u }_{k'+1}^2$ in the right hand side of \eqref{density es2N}. If we
write
\begin{equation}\label{boo_5de}
 \Hf := \ns{ \bar{\na}_{\ast}^{4N}  \mathcal{Q}}_0 + \ns{\nab_\ast^{4N-2j-k-1} \dt^j \mathcal{Q}}_{k+1},
\end{equation}
then we  have
\begin{equation}\label{boo_7}
\Hf \ls  \norm{ \bar{\na}_{\ast}^{4N}  \mathcal{Q}}_0^2  + \sum_{k'=0}^{k} \norm{\nab_{\ast}^{4N-2j-k'-1}  \pa_3^{k'+1}\pa_t^j  \mathcal{Q}}_0^2\le\Bf.
\end{equation}
In turn, we have
\begin{align}\label{boo_3'}
  \nonumber  \Af +  \int_0^t\Hf ds& \ls {\mathcal{E}}_{2N}^\sigma(0)+\int_0^t
 \norm{\pa_t^{ j+1} u  }_{4N-2 j-1}^2
+ \sum_{k'=1}^{k}\norm{\na_{\ast}^{4N-2 j-k'} \pa_t^{ j } \mathcal{Q} }_{k'}^2 ds
   \\
 & \quad+ \int_0^t \Df +  \sqrt{\mathcal{E}_{2N}^\sigma }\(\mathcal{D}_{2N}^\sigma+\mathcal{E}_{2N}^\sigma+  \f_{2N}\)ds.
\end{align}
A standard induction argument on the above yields
\begin{multline}\label{boo_9}
 \sum_{k=0}^{4N-2j-1}  \Af  +\int_0^t\sum_{k=0}^{4N-2j-1}  \Hf ds
\ls  {\mathcal{E}}_{2N}^\sigma(0)+\int_0^t
\norm{\pa_t^{ j+1} u  }_{4N-2 j-1}^2+\norm{\na_{\ast}^{4N-2 j} \pa_t^{ j } \mathcal{Q} }_{0}^2 ds
\\+ \int_0^t \Df +  \sqrt{\mathcal{E}_{2N}^\sigma }\(\mathcal{D}_{2N}^\sigma+\mathcal{E}_{2N}^\sigma+  \f_{2N}\) ds
\\
\ls  {\mathcal{E}}_{2N}^\sigma(0)+\int_0^t \norm{\pa_t^{ j+1} u  }_{4N-2 j-1}^2
+ \Df +  \sqrt{\mathcal{E}_{2N}^\sigma }\(\mathcal{D}_{2N}^\sigma+\mathcal{E}_{2N}^\sigma+  \f_{2N}\) ds,
\end{multline}
where we have used \eqref{drho} to derive the second inequality.

Note now, using the definition of $\Hf$, that
\begin{equation}
 \ns{ \bar{\na}_{\ast}^{4N}  \mathcal{Q}}_0 + \ns{\dt^j \mathcal{Q}}_{4N-2j}
\ls
 \sum_{k=0}^{4N-2j-1}   \Hf  .
\end{equation}
Using Lemma \ref{lemmau2N} with $k = 4N-2j$, we then have that
\begin{align}
&\nonumber \ns{ \dt^j u}_{4N-2j+1}  + \ns{\nab \dt^j (h'(\bar{\rho}) \q)}_{4N-2j-1}
\\
&\quad\ls
 \sum_{k=0}^{4N-2j-1} \Hf  + \ns{\dt^{j+1} u}_{4N-2j-1}+ \Df +  \mathcal{E}_{2N}^\sigma  \(\mathcal{D}_{2N}^\sigma+\mathcal{E}_{2N}^\sigma+   \f_{2N}\).
\end{align}
Hence \eqref{boo_9} implies
\begin{align}\label{boo_11}
 &\nonumber\sum_{k=0}^{4N-2j-1}  \Af  +\int_0^t   \left( \ns{ \dt^j u}_{4N-2j+1}  + \ns{\nab \dt^j (h'(\bar{\rho}) \q)}_{4N-2j-1}  \right) ds
 \\
&\quad\ls  {\mathcal{E}}_{2N}^\sigma(0)+\int_0^t\norm{\pa_t^{ j+1} u  }_{4N-2 j-1}^2
+  \Df
+  \sqrt{\mathcal{E}_{2N}^\sigma }\(\mathcal{D}_{2N}^\sigma+\mathcal{E}_{2N}^\sigma+  \f_{2N}\) ds,
\end{align}
for all $0 \le j \le 2N-1$. A standard induction argument on the above yields
\begin{align}\label{boo_11'}
 &\nonumber\sum_{j=0}^{2N-1}\sum_{k=0}^{4N-2j-1}  \Af  +\int_0^t\sum_{j=0}^{2N-1}  \left( \ns{ \dt^j u}_{4N-2j+1}  + \ns{\nab \dt^j (h'(\bar{\rho}) \q)}_{4N-2j-1}  \right) ds
 \\ &\quad
\ls   {\mathcal{E}}_{2N}^\sigma(0)+\int_0^t \Df+
\sqrt{\mathcal{E}_{2N}^\sigma }\(\mathcal{D}_{2N}^\sigma+\mathcal{E}_{2N}^\sigma+  \f_{2N}\) ds.
\end{align}

Consequently, a suitable linear combination of \eqref{boo_1} and \eqref{boo_11'} gives
\begin{align}\label{boo_12}
&\nonumber \Ef(t)+ \sum_{j=0}^{2N-1}\sum_{k=0}^{4N-2j-1}  \Af  +\int_0^t   \bar{\mathcal{D}}_{2N} ds
\\&\quad\ls
    {\mathcal{E}}_{2N}^\sigma(0)+ \int_0^t  \sqrt{\mathcal{E}_{2N}^\sigma }\(\mathcal{D}_{2N}^\sigma+ \mathcal{E}_{2N}^\sigma+ \f_{2N}\) ds +\int_0^t \norm{\eta_-}_{4N-1/2}^2 ds.
\end{align}
Note that
\begin{multline}\label{ee_12}
  \sum_{j=0}^{2N-1} \sum_{k=0}^{4N-2j-1}   \Af \asymp \sum_{j=0}^{2N-1} \sum_{k=0}^{4N-2j-1}\sum_{k'=0}^k \ns{\nab_\ast^{4N-2j-k'-1} \p_3^{k'+1} \dt^j(h'(\bar{\rho})\q)}_0
\\
= \sum_{j=0}^{2N-1} \sum_{k=0}^{4N-2j-1} \ns{\nab_\ast^{4N-2j-k-1} \p_3 \dt^j(h'(\bar{\rho})\q)}_k
\asymp \sum_{j=0}^{2N-1} \ns{\p_3 \dt^j(h'(\bar{\rho})\q)}_{4N-2j-1}:=\mathcal{Z}.
\end{multline}
Since
\begin{equation}
 \p_3 \dt^j \q = \frac{1}{h'(\bar{\rho})} \left[ \p_3 \dt^j (h'(\bar{\rho}) \q) - \p_3(h'(\bar{\rho}) )\dt^j\q  \right].
\end{equation}
and $h'(\bar{\rho})$ is smooth on $[-b,0]$ and $[0,\ell]$ and bounded below from zero, we may estimate
\begin{equation}
 \ns{ \p_3 \dt^j q}_{0} \ls \ns{ \p_3 \dt^j(h'(\bar{\rho}) \q)}_0 +  \ns{\dt^j \q}_0 \ls \z + \ns{\dt^j \q}_0 \ls \z + \ns{\bar{\nab}_\ast^{4N-1} \q}_0
\end{equation}
and similarly
\begin{multline}
 \ns{ \p_3 \dt^j q}_{i} \ls \ns{ \p_3 \dt^j(h'(\bar{\rho}) \q)}_i +  \ns{\dt^j \q}_i \ls \z+ \ns{\nab_\ast^i \dt^j \q}_0 + \ns{ \p_3 \dt^j \q}_{i-1} \\
\ls \z+ \ns{\bar{\nab}_\ast^{4N-1} \q}_0 + \ns{ \p_3 \dt^j \q}_{i-1}
\end{multline}
for $i=1,\dotsc,4N-2j-1$. A standard induction argument then yields
\begin{equation}\label{ee_1}
\ns{\p_3 \dt^j  \q}_{4N-2j-1} \le \sum_{i=0}^{4N-2j-1} \ns{\p_3 \dt^j  \q}_{i}  \ls \z + \ns{\bar{\nab}_\ast^{4N-1} \q}_0.
\end{equation}
On the other hand,
\begin{equation}\label{ee_2sum}
 \ns{\dt^j \q}_{4N-2j} \le \ns{\bar{\nab}_\ast^{4N} \q}_{0} + \ns{\p_3 \dt^j \q}_{4N-2j-1},
\end{equation}
so summing  \eqref{ee_1} and \eqref{ee_2sum} yields 
\begin{equation}\label{ee_2}
 \sum_{j=0}^{2N}  \ns{\dt^j \q}_{4N-2j} \ls \z  + \ns{\bar{\nab}_\ast^{4N} \q}_{0}.
\end{equation}
We then deduce \eqref{ebar2N} from \eqref{boo_12} and \eqref{ee_2}.
\end{proof}

\subsection{Full energy estimates}\label{sec_full}

In this section, we will derive our ultimate energy estimates.
First, we combine the results in Sections \ref{stable3} and \ref{sec_combo}. We define
\begin{equation}\label{energytil}
\tilde{\mathcal{E}}_{2N}^\sigma:= \bar{\mathcal{E}}_{2N}^\sigma + \sum_{j=0}^{2N} \ns{\dt^j \eta}_{4N-2j}.
\end{equation}
Then a linear combination of the estimates \eqref{ebar2N} of Proposition \ref{boostrap2N} and the estimates \eqref{transportes}--\eqref{transportes2} of Proposition \ref{transportlemma} gives
\begin{align}\label{ebar2N1}
\nonumber
\tilde{\mathcal{E}}_{2N}^\sigma(t)+\f_{2N}(t) + \int_0^t \bar{\mathcal{D}}_{2N} ds
\le& C_\varepsilon{\mathcal{E}}_{2N}^\sigma(0)+\f_{2N}(0) + C_\varepsilon\int_0^t  \sqrt{\mathcal{E}_{2N}^\sigma }\(\mathcal{D}_{2N}^\sigma+ \mathcal{E}_{2N}^\sigma+ \f_{2N}\)ds
\\&+\varepsilon\int_0^t  \(\mathcal{E}_{2N}^\sigma+\f_{2N}\)  ds+  C_\varepsilon\int_0^t\norm{\eta_-}_{4N-1/2}^2ds
\end{align}
for any $\varepsilon>0$ and a corresponding constant $C_\varepsilon>0$.

Next, we show that ${\mathcal{E}}_{2N}^\sigma$ is comparable to $ \tilde{\mathcal{E}}_{2N}^\sigma$ and that ${\mathcal{D}}_{2N}^\sigma$ is comparable to $ \bar{\mathcal{D}}_{2N}$. We begin with the result for the energy.

\begin{prop}\label{eth}
Let ${\mathcal{E}}_{2N}^\sigma$ and $\tilde{\mathcal{E}}_{2N}^\sigma$ be as defined in \eqref{p_energy_def} and \eqref{energytil}  respectively.  It holds that
\begin{equation}\label{e2n}
{\mathcal{E}}_{2N}^\sigma \lesssim
 \tilde{\mathcal{E}}_{2N}^\sigma  +\se{2N}^\sigma\(\se{2N}^\sigma +   \f_{2N}\).
\end{equation}
\end{prop}

\begin{proof}
We compactly write
\begin{equation}\label{n1}
 {\mathcal{X}}_{2N} =  \ns{ \bar{\nab}^{4N-2} G^1}_{1}  +  \ns{ \bar{\nab}^{4N-2}  G^2}_{0} +
 \ns{ \bar{\nab}_{\ast}^{  4N-2} G^3}_{1/2} +
 \ns{\bar{\nab}_{\ast }^{  4N-1} G^4}_{1/2} .
\end{equation}

We first estimate $\dt^j u$ for $j=0,\dots,2N-1$. The key is to use the elliptic regularity theory of the following two-phase Lam\'e system derived from \eqref{ns_perturb}:
\begin{equation}\label{lame}
 \begin{cases}
 -\mu\Delta u-(\mu/3+\mu')\nabla \diverge u =G^2-\bar{\rho} \partial_t    u  - \bar{\rho}\nabla \left(h'(\bar{\rho})\q\right)    & \text{in }
\Omega  \\
 - \S(  u) e_3  = (-P'(\bar\rho)\q+\rho_1  g \eta_+ -\sigma_+ \Delta_\ast \eta_+ ) e_3 +G_+^3
 & \text{on } \Sigma_+
 \\ -\jump{\S(u)}e_3
=(\jump{P'(\bar\rho)\q}+\rj g\eta_- +\sigma_- \Delta_\ast \eta_-)e_3-G_-^3&\hbox{on }\Sigma_-
\\\jump{u}=0 &\hbox{on }\Sigma_-
\\ u_-=0 &\hbox{on }\Sigma_b.
\end{cases}
\end{equation}
We let $  j=0,\dots, 2N-1$ and then apply $\dt^ j$ to the problem \eqref{lame} and use the elliptic estimates of Lemma \ref{lame reg} with $r=4N-2 j\ge2$, by \eqref{n1} and the trace theory to obtain
\begin{align}\label{n3}
 \nonumber\norm{\pa_t^ j u}_{4N-2 j}^2 &\lesssim   \norm{\pa_t^{ j } G^2}_{4N-2 j-2}^2+\norm{\pa_t^{ j+1} u}_{4N-2 j-2}^2
 +\norm{ \pa_t^{ j } \q }_{4N-2 j-1}^2+\norm{\pa_t^{ j }  \q }_{H^{4N-2 j-3/2}(\Sigma) }^2
   \\\nonumber&\quad
+\norm{\pa_t^{ j }  \eta }_{4N-2 j-3/2 }^2+\sigma^2\norm{\pa_t^{ j }  \eta }_{4N-2 j+1/2 }^2+\norm{\pa_t^{ j } G^3}_{4N-2 j-3/2 }^2
 \\
&\lesssim \norm{\pa_t^{ j+1} u}_{4N-2 ( j+1)}^2
+    \tilde{\mathcal{E}}_{2N}^\sigma + {\mathcal{X}}_{2N}  .
\end{align}
Using a simple induction based on the estimate \eqref{n3}, utilizing the $\ns{\dt^{2N} u}_0$ estimate contained in $\tilde{\mathcal{E}}_{2N}^\sigma$ for the base case, we easily deduce that for $ j=0,\dots,2N$,
\begin{equation}\label{n5}
 \norm{ \pa_t^ j u}_{4N-2 j}^2\lesssim     \tilde{\mathcal{E}}_{2N}^\sigma + {\mathcal{X}}_{2N}  .
\end{equation}

We then estimate $\pa_t^ j\q$ and $\pa_t^ j\eta$ for $ j=1,\dots,2N$ to get an improvement. By the first equation of $\eqref{ns_perturb}$, using the estimates \eqref{n5} and \eqref{n1}, we have that for $ j=1,\dots,2N$,
\begin{align}\label{n6}
\nonumber\norm{ \pa_t^ j \q}_{4N-2 j+1}^2 &\lesssim   \norm{ \pa_t^{ j-1} u}_{4N-2 j+2}^2+\norm{\pa_t^{ j-1} G^1}_{4N-2 j+1}^2
 \\ &= \norm{ \pa_t^{ j-1} u}_{4N-2({ j-1})}^2+\norm{\pa_t^{ j-1} G^1}_{4N-2({ j-1})-1}^2
\ls   \tilde{\mathcal{E}}_{2N}^\sigma + {\mathcal{X}}_{2N} .
\end{align}
Now by the kinematic boundary condition
\begin{equation}
\partial_t\eta=u_3+G^4\text{ on }\Sigma,
\end{equation}
we have that for $j=1,\dots,2N$, by the trace theory, \eqref{n5} and \eqref{n1},
\begin{align}\label{n7}
\nonumber\ns{\partial_t^j\eta}_{4N-2j+3/2}
&\le \ns{\dt^{j-1}u_3}_{H^{4N-2j+3/2}(\Sigma)}+\ns{\dt^{j-1} G^4}_{4N-2j+3/2}
\\
&\ls \ns{\dt^{j-1}u }_{4N-2(j-1)}+\ns{\dt^{j-1} G^4}_{4N-2(j-1)-1/2}
\ls    \tilde{\mathcal{E}}_{2N}^\sigma + {\mathcal{X}}_{2N} .
\end{align}

Summing  the estimates \eqref{n5}, \eqref{n6} and \eqref{n7},  we conclude that
\begin{equation}\label{claim}
{\mathcal{E}}_{2N}^\sigma \lesssim       \tilde{\mathcal{E}}_{2N}^\sigma + {\mathcal{X}}_{2N} .
\end{equation}
Using the estimate \eqref{p_G_e_0} of Lemma \ref{p_G2N_estimates} to bound $\mathcal{X}_{2N} \lesssim\se{2N}^\sigma\(\se{2N}^\sigma +   \f_{2N}\)$, we then obtain \eqref{e2n} from \eqref{claim}.
\end{proof}

Next we consider a similar result for the dissipation.
\begin{prop}\label{dth}
Let ${\mathcal{D}}_{2N}^\sigma$ and $\bar{\mathcal{D}}_{2N}$ be as defined in \eqref{p_dissipation_def} and \eqref{Dnn} respectively. It holds that
\begin{equation}\label{d2n}
{\mathcal{D}}_{2N}^\sigma
\ls     \bar{\mathcal{D}}_{2N} + \mathcal{E}_{2N}^\sigma \(\mathcal{D}_{2N}^\sigma  + \mathcal{E}_{2N}^\sigma+{\mathcal{F}}_{2N} \).
 \end{equation}
\end{prop}

\begin{proof}
We compactly write
\begin{equation}\label{p_D_b_4}
\begin{split}
 {\mathcal{Y}}_{2N} = & \ns{ \bar{\nab}^{4N-1} G^1}_{0} + \ns{ \bar{\nab}^{4N-2}\dt G^1}_{0} +
 \ns{\bar{\nab}_{\ast}^{4N-1} G^3}_{1/2} \\&+ \ns{\bar{\nab}_{\ast }^{ 4N-1} G^4}_{1/2}
+ \ns{\bar{\nab}^{4N-2} \dt G^4}_{1/2}+ \sigma^2\ns{\nab_\ast^{4N}   G^4}_{1/2}.
\end{split}
\end{equation}

We now estimate the remaining parts of $\bar{\mathcal{D}}_{2N}$ not contained in ${\mathcal{D}}_{2N}^\sigma$.  We divide the proof into several steps.

Step 1 -- $\dt^j\q$ estimates

We first notice that by the first equation of \eqref{ns_perturb},
\begin{equation}\label{n511}
\norm { \pa_t \q
 }_{4N-1}^2\le \norm {u }_{4N}^2+\norm {
G^{1}}_{4N -1}^2\lesssim \sdb{2N}   +\y_{2N},
\end{equation}
and for $2\le  j\le 2N+1$,
\begin{align}\label{n512}
\nonumber
\norm { \pa_t^{ j}\q}_{4N-2 j+2}^2&\le \norm { \pa_t^{ j-1}u}_{4N-2 j+3}^2+\norm { \pa_t^{ j-1}G^{1}}_{4N-2 j+2}^2
\\&\le \norm { \pa_t^{ j-1}u}_{4N-2 j+3}^2+\norm { \pa_t^{ j-1}G^{1}}_{4N-2 j+2}^2\lesssim \sdb{2N}   +\y_{2N}.
\end{align}

Step 2 -- $\dt^j \eta$ estimates

We now derive estimates for time derivatives of $\eta$. For the term $\dt^j \eta$ for $j\ge 2$ we use the kinematic boundary condition
\begin{equation}\label{n61}
\partial_t\eta=u_3+G^4\text{ on }\Sigma.
\end{equation}
Indeed, for $j=2,\dots,2N+1$ trace theory and \eqref{p_D_b_4} imply that
\begin{align}\label{eta1}
\nonumber
\ns{\partial_t^j\eta}_{4N-2j+5/2} & \ls  \ns{\partial_t^{j-1}u_3}_{H^{4N-2j+5/2}(\Sigma)} +\ns{\partial_t^{j-1}G^4}_{4N-2j+5/2}
\\
&\lesssim   \ns{\partial_t^{j-1}u }_{{4N-2(j-1)+1}} +\ns{\partial_t^{j-1}G^4}_{4N-2(j-1)+1/2}
 \lesssim\sdb{2N}   +\y_{2N}.
\end{align}
For the term $\partial_t\eta$, we again use \eqref{n61}, trace theory, and \eqref{p_D_b_4} to find
\begin{align}\label{eta2}
\nonumber\sigma^2\ns{\partial_t \eta}_{4N+1/2}+\ns{\partial_t \eta}_{4N-1/2} &\lesssim (1+\sigma^2)\ns{u_3}_{H^{4N+1/2}(\Sigma)}+\sigma^2\ns{ G^4}_{4N+1/2}+\ns{ G^4}_{4N-1/2}
\\ &\lesssim  \ns{ u }_{4N+1}+\y_{2N}\lesssim   \sdb{2N}   +\y_{2N}.
\end{align}

Step 3 -- $\nab_\ast \eta$ estimates

In this step we use instead the dynamic boundary condition
\begin{equation}\label{pb1}
 -\sigma_+\Delta_\ast\eta_++\rho_1  g\eta_+= P'_+(\rho_1 )\q_+-2 \mu_+\partial_3u_{3,+}-\mu'_+\diverge u_+ -G_{3,+}^3\text{ on }\Sigma_+
\end{equation}
 and
\begin{equation} \label{pb2}
-\sigma_-\Delta_\ast\eta_-=-\jump{P'(\bar\rho)\q}+2\jump{\mu\partial_3u_3}+\jump{\mu'\diverge u}-G_{3,-}^3+\rj g\eta_- \text{ on } \Sigma_-.
\end{equation}
Notice that at this point we do not have any bound of $\q$ on the boundary $\Sigma$, but we have bounded  $\nabla (h'(\bar\rho)\q)$ in $\Omega$.  As such, we first apply $\nab_\ast$ to \eqref{pb1} and \eqref{pb2} and then we employ the standard elliptic theory.  This, trace theory, and \eqref{p_D_b_4} 
then provide the estimate
 \begin{align}\label{n51}
\nonumber& \sigma^2 \ns{\nab_\ast \eta}_{4N+1/2} + \ns{ \nab_\ast \eta_+}_{4N-3/2}
\\\nonumber&\quad
\lesssim \ns{  \nab_\ast (h'(\bar\rho)\q) }_{H^{4N-3/2}(\Sigma)}  + \ns{  \nab_\ast \nabla u }_{H^{4N-3/2}(\Sigma)} +\ns{\nab_\ast G^3_3}_{ 4N-3/2 }+ \ns{ \nab_\ast \eta_-}_{4N-3/2}
\\\nonumber&\quad\lesssim \ns{\nabla (h'(\bar\rho)\q)}_{4N-1} + \ns{u }_{4N+1}  + \ns{G^3}_{4N-1/2}+\ns{ \eta_-}_{4N-1/2}
\\&\quad\lesssim \sdb{2N}   +\y_{2N}+\ns{ \eta_-}_{4N-1/2}.
\end{align}

Consequently, summing  the estimates \eqref{n511}, \eqref{n512}, \eqref{eta1}, \eqref{eta2} and \eqref{n51},  we conclude
\begin{equation}\label{claim2}
{\mathcal{D}}_{2N}^\sigma \lesssim       \bar{\mathcal{D}}_{2N} +\y_{2N}+\ns{ \eta_-}_{4N-1/2}.
\end{equation}
Using the estimate \eqref{p_G_e_00} of Lemma \ref{p_G2N_estimates} to bound $\mathcal{Y}_{2N} \lesssim\se{2N}^\sigma\left(\se{2N}^\sigma +   \f_{2N}\right)$, we then obtain \eqref{d2n} from \eqref{claim2}.
\end{proof}


Finally we are ready to prove Theorem  \ref{energyeses}.

\begin{proof}[Proof of Theorem \ref{energyeses}]
By the estimates \eqref{e2n} of Proposition \ref{eth} and the estimates \eqref{d2n} of Proposition \ref{dth}, we can improve the inequality \eqref{ebar2N1} to be
\begin{align}\label{ebar2N2}
\nonumber
 {\mathcal{E}}_{2N}^\sigma(t)+\f_{2N}(t) + \int_0^t  {\mathcal{D}}_{2N}^\sigma ds
\le& C_\varepsilon{\mathcal{E}}_{2N}^\sigma(0)+\f_{2N}(0) + C_\varepsilon\int_0^t  \sqrt{\mathcal{E}_{2N}^\sigma }\(  {\mathcal{D}}_{2N}^\sigma+\mathcal{E}_{2N}^\sigma+ \f_{2N}\)ds
\\&+\varepsilon\int_0^t  \(\mathcal{E}_{2N}^\sigma+\f_{2N}\)  ds+  C_\varepsilon\int_0^t\norm{\eta_-}_{4N-1/2}^2ds.
\end{align}
Sobolev interpolation on $\Sigma$ allows us to bound
\begin{equation}
C_\varepsilon\norm{\eta_-}_{4N-1/2}^2\le\varepsilon \norm{\eta_-}_{4N }^2+ C_\varepsilon\norm{\eta_-}_{0}^2\le \varepsilon \mathcal{E}_{2N}^\sigma+C_\varepsilon\norm{\eta_-}_{0}^2.
\end{equation}
We can thus refine the inequality \eqref{ebar2N1} to be \eqref{fullenergy}.
\end{proof}
\section{Nonlinear instability}\label{proof}

\subsection{Restated estimates}

In the following, we take $\lambda=\Lambda$ defined by \eqref{Lambda} when $\sigma_->0$, while we take $\lambda=\Lambda_\ast$ defined by \eqref{littlelambda} when $\sigma_-=0$.  In each case, we have that $\frac{\Lambda}{2}<\lambda\le \Lambda $.

We define the norm $\Lvert3\cdot\Rvert3_{00}$ appearing in Theorem \ref{maintheorem} by
\begin{equation}\label{norm3}
 \Lvert3  (q ,  u, \eta) \Rvert3_{00} := \sqrt{\mathcal{E}_{2N}^\sigma+\mathcal{F}_{2N}}
\end{equation}
for an integer $N\ge 3$, where $\mathcal{E}_{2N}^\sigma$ and $\mathcal{F}_{2N}$ are given by \eqref{p_energy_def} and \eqref{fff}. For notational convenience, we denote
\begin{equation}\label{Uvector}
U:= (  q ,  u,  \eta).
\end{equation}

We now restate the main results of the previous Sections in our new notation.
\begin{prop}\label{prop1}
Let the norm $\Lvert3\cdot\Rvert3_{00}$ be given by \eqref{norm3}. Then we have the following.
\begin{enumerate}
\item There is a growing mode $U^\star:=(q^\star, u^\star, \eta^\star)$
 satisfying $\norm{ \eta^\star_-}_0= 1$,
$\Lvert3 U^\star \Rvert3_{00}=C_1<\infty$, and
$e^{\lambda t}U^\star$ is the solution to \eqref{linear}.
\item  Suppose that $U(t)$ is the solution to \eqref{ns_perturb}. There exists a small constant $\delta$ such that  if $\Lvert3U(t) \Rvert3_{00}\le\delta$ for all $t \in [0,T]$, then there exists $C_\delta> 0$ so
that the following inequality holds for $t \in [0,T]$:
\begin{align}
\nonumber  \Lvert3 U(t) \Rvert3_{00}^2 &\le  C_\delta\Lvert3 U(0)  \Rvert3_{00}^2
   +  \frac{\lambda}{2}\int_0^t\Lvert3U(s)  \Rvert3_{00}^2\,ds \\ &\qquad+   C_\delta\int_0^t\Lvert3U(s)  \Rvert3_{00}^3\,ds
   + C_\delta\int_0^t \norm{  \eta_-(s)}_{0}^2\,ds. \label{energyes}
 \end{align}
\item There exists $C_2> 0$ so that
\begin{align}
\nonumber \| \eta_-(t)-\iota e^{\lambda t}\eta^\star_-\|_{0} & \le
C_2 e^{ \Lambda t}\Lvert3U(0)-\iota U^\star  \Rvert3_{00} +C_2\int_0^t  \Lvert3U(s)  \Rvert3_{00}^2 ds
 \\ &\qquad+C_2\sqrt{\int_0^t  e^{2\Lambda
(t-s)} \Lvert3 U(s)  \Rvert3_{00}^2\Lvert3 U(s)-\iota e^{\lambda t} U^\star   \Rvert3_{00} ds}.
\end{align}

\end{enumerate}
 \end{prop}
\begin{proof}
Statement 1 follows from Theorem \ref{growingmode}.  Statement 2 follows from the estimates \eqref{fullenergy} of Theorem \ref{energyeses} by taking $\varepsilon=\lambda/2$ and then taking $\delta$ sufficiently small to absorb the term $\sqrt{\mathcal{E}_{2N}^\sigma } {\mathcal{D}}_{2N}^\sigma$ on the right hand side by the dissipation.

We now prove Statement 3 by using Theorem \ref{lineargrownth}. We observe that $U(t)-\iota e^{\lambda t}U^\star$ solve the problem \eqref{linear ho} with initial data $U(0)-\iota U^\star$ and the force terms $G^i$ given by \eqref{G1_def}--\eqref{G4_def}. Then Statement 3 follows from \eqref{result2} by noticing that
\begin{equation}
\mathcal{E}_{G}=\norm{G^1}_1^2 +\norm{\dt G^2}_0^2 +\norm{\dt G^3}_0^2 +\norm{ G^4}_2^2\ls  \Lvert3U  \Rvert3_{00}^4.
\end{equation}
We thus conclude the proposition.
\end{proof}

\subsection{Local well-posedness}
Thus far we have not elaborated on the local well-posedness theory for our problem that we developed in our companion paper \cite{JTW_LWP}. In the result we will refer to the ``necessary compatibility conditions'' required for the local well-posedness in our energy spaces.  These are cumbersome to write out explicitly, and we refer to \cite{JTW_LWP} for the explicit statement.

\begin{thm}\label{LWP}
Suppose that the initial data $U(0)$ satisfies the necessary compatibility conditions.  There exist $\delta_0 , T>0$ so that if
\begin{equation}\label{lwp_01}
\Lvert3   U(0) \Rvert3_{00} < \delta_0,
\end{equation}
then there exists a unique solution $U(t)$ to \eqref{ns_perturb} on $[0,T]$ that satisfies the estimate
\begin{equation}\label{lwp_02}
 \Lvert3 U(t) \Rvert3_{00} \lesssim \sqrt{1+T} \Lvert3 U(0) \Rvert3_{00}
\end{equation}
for all $t \in [0,T]$.
\end{thm}
\begin{proof}
The theorem can be deduced readily from Theorem 2.1 of \cite{JTW_LWP}. Indeed, Theorem 2.1 of \cite{JTW_LWP} is stated in more general form, where we only require $\norm{\eta_0}_{4N-1/2}$ to be small and no smallness condition is imposed on $u_0$ or $q_0$. We record this version of local well-posedness so that it can be adapted directly in our instability analysis.
\end{proof}

\subsection{Data analysis}

In order to prove our nonlinear instability result, we want to use the linear growing mode in Proposition \ref{prop1} to construct small initial data for the nonlinear problem \eqref{ns_perturb}. Since we are involved in the higher-order regularity context, we cannot simply set the initial data for the nonlinear problem to be a small constant times the linear growing modes. The reason for this is that the initial data for the nonlinear problem must satisfy certain nonlinear compatibility conditions in order for us to guarantee local existence in the  space corresponding to norm $\Lvert3\cdot\Rvert3_{00}$, which the linear growing mode  solutions do not satisfy.

To get around this obstacle, we note that the nonlinear problem   is  slightly perturbed from the linearized problem and so  their compatibility conditions for the small initial data should be close to each other. We are able to produce  a curve of small initial data satisfying the compatibility conditions for the nonlinear problem which are close to the linear growing modes.

\begin{prop}\label{intialle}
  Let $U^\star$ be the linear growing mode stated in Proposition \ref{prop1}. Then there exists  a number $\iota_0>0$ and a family of initial data
  \begin{equation}\label{initial0}
U_0^\iota
=\iota U^\star+\iota^2
\tilde{U}(\iota)
\end{equation}
for $\iota\in [0,\iota_0)$ so that the followings  hold.

1. $U_0^\iota$ satisfy the nonlinear compatibility conditions required by Theorem \ref{LWP} for a solution to the nonlinear problem \eqref{ns_perturb} to exist in the norm $\Lvert3\cdot\Rvert3_{00}$.

2. There exist $C_3, C_4>0$ independent of $\iota$ so that
\begin{equation}\label{initial1}
\Lvert3\tilde{U}(\iota) \Rvert3_{00} \le
C_3
\end{equation}
and
\begin{equation}\label{initial2}
\Lvert3U_0^\iota\Rvert3_{00}^2\le C_4\iota^2.
\end{equation}
 \end{prop}
 \begin{proof}
 See the abstract argument before  Lemma 5.3 of \cite{JT}.
\end{proof}

\subsection{Proof of Theorem \ref{maintheorem}}

With Propositions \ref{prop1}, Theorem \ref{LWP} and Proposition \ref{intialle} in hand, we can now present the
\begin{proof}[Proof of Theorem \ref{maintheorem}]
Recall the notation \eqref{Uvector}. First, we restrict to $0<\iota<\iota_0\le \theta_0$, where $\iota_0$ is as small as in Proposition \ref{intialle}  and  the value of $\theta_0$ is   sufficiently small to be determined later. For $0<\iota\le \iota_0$, we let $U_0^\iota$ be the initial data given in Proposition \ref{intialle}.  By further restricting $\iota$ we may use \eqref{initial2} to verify that \eqref{lwp_01} holds, which then allows us to use Theorem \ref{LWP} to find  $U^\iota(t)$,  solutions to the system \eqref{ns_perturb} with
 \begin{equation}
 \left.U^\iota \right|_{t=0}=
U_0^\iota
=\iota U^\star+\iota^2
\tilde{U}(\iota).
\end{equation}

Fix $\delta>0$ as small as in Proposition \ref{prop1}, and let $C_\delta>0$ be the constant appearing in Proposition \ref{prop1} for this fixed choice of $\delta$.   We then define $\tilde{\delta}=\min\{\delta,\frac{\lambda}{2C_\delta}\}$.  Denote
 \begin{equation}
 T^\ast=\sup\left\{ s : \Lvert3U^\iota(t) \Rvert3_{00}   \le\tilde{\delta}, \text{ for } 0\le t\le s\right\}
 \end{equation}
 and
 \begin{equation}
 T^{\ast\ast}=\sup\left\{ s: \norm{ \eta^\iota_-(t)}_{0}\le 2 \iota e^{\lambda t}, \text{ for } 0\le t\le s\right\}.
 \end{equation}
 With $\iota_0$ small enough, \eqref{initial2} and \eqref{lwp_02} guarantee that $T^\ast$ and $T^{\ast\ast}>0$.  Recall that $T^\iota$ is defined by \eqref{escape_time}.   Then for all $t\le \min\{T^\iota,T^\ast,T^{\ast\ast}\}$, we deduce from the estimate \eqref{energyes} of Proposition \ref{prop1}, the definitions of $T^\ast$ and $T^{\ast\ast}$, and \eqref{initial2} that
 \begin{equation}\label{ins1}
 \begin{split}\Lvert3U^\iota(t)\Rvert3_{00}^2
&\le  C_\delta\Lvert3U^\iota_0 \Rvert3_{00}^2
   +  \frac{\lambda}{2}\int_0^t\Lvert3U^\iota(s) \Rvert3_{00}^2\,ds \\&\quad+   C_\delta\int_0^t\Lvert3U^\iota(s) \Rvert3_{00}^3\,ds
   + C_\delta\int_0^t \norm{ \eta^\iota_-(s)}_{0}^2\,ds
 \\&\le \left(\frac{\lambda}{2} + \tilde{\delta} C_\delta \right) \int_0^t\Lvert3 U^\iota(s) \Rvert3_{00}^2\,ds+C_\delta C_4\iota^2+\frac{C_\delta(2\iota)^2}{2\lambda}e^{2\lambda t}
  \\&\le \lambda\int_0^t\Lvert3U^\iota(s) \Rvert3_{00}^2\,ds+C_5 \iota^2e^{2\lambda t}.
\end{split}
\end{equation}
for some constant $C_5>0$ independent of $\iota$. We may view \eqref{ins1} as a differential inequality.  Then Gronwall's lemma implies that
\begin{equation}\label{ins11}
 \begin{split}\Lvert3U^\iota(t)\Rvert3_{00}^2
&\le   C_5 \iota^2e^{2\lambda t}+C_5\iota^2e^{\lambda t}\int_0^t \lambda e^{\lambda s}\,ds
 \\&\le   C_5 \iota^2e^{2\lambda t}+ C_5\iota^2 e^{2\lambda t}  =  2C_5\iota^2e^{2\lambda t}.
\end{split}
\end{equation}
We then deduce from Proposition  \ref{prop1} and \eqref{ins11} that
\begin{multline}\label{ins12}
\| \eta_-(t)-\iota e^{\lambda t}\eta^\star_-\|_{0}  \le
C_2 e^{ \Lambda t}\Lvert3\iota^2
\tilde{U}(\iota) \Rvert3_{00} +C_2\int_0^t  \Lvert3U^\iota(s)  \Rvert3_{00}^2 ds
 \\+C_2\sqrt{\int_0^t  e^{2\Lambda
(t-s)} \Lvert3U^\iota(s) \Rvert3_{00}^2\Lvert3U^\iota(s)-\iota e^{\lambda s} U^\star   \Rvert3_{00} ds}
 \\\le
C_2C_3 e^{ \Lambda t}\iota^2+C_2\int_0^t  2C_5\iota^2e^{2\lambda s}
  +C_2\sqrt{\int_0^t  e^{2\Lambda
(t-s)}  2C_5\iota^2e^{2\lambda s} (\sqrt{2C_5}\iota e^{ \lambda s}+C_1\iota e^{ \lambda s}) ds}
 \\\le
C_6 e^{ \Lambda t}\iota^2+ C_6\iota^2 e^{2\lambda t}+C_6\iota^\frac{3}{2} e^{\frac{3}{2}\lambda t}
 \le 2C_6  \iota^2 e^{2\lambda t}+C_6\iota^\frac{3}{2} e^{\frac{3}{2}\lambda t}.
\end{multline}
Here we have used the fact that $\Lambda<2\lambda$.

Now we claim that
\begin{equation}\label{Tmin}
T^\iota= \min\{T^\iota,T^\ast,T^{\ast\ast}\}
\end{equation}
by fixing $\theta_0$ small enough, namely, setting
 \begin{equation}
 \theta_0=\min\left\{\frac{\tilde{\delta}}{2\sqrt{2C_5}},\frac{1}{8C_6},\frac{1}{16C_6^2}\right\}.
 \end{equation}
 Indeed, if $T^\ast= \min\{T^\iota,T^\ast,T^{\ast\ast}\}$, then by \eqref{ins11}, we have
  \begin{equation}
  \begin{split}
  \Lvert3U^\iota(T^\ast)\Rvert3_{00}
\le    \sqrt{2C_5}\iota e^{ \lambda T^\ast} \le   \sqrt{2C_5}\iota e^{ \lambda
T^\iota}=\sqrt{2C_5}\theta_0\le \frac{\tilde{\delta}}{2} < \tilde{\delta},
\end{split}
\end{equation}
which contradicts to the definition of $T^\ast$.  If $T^{\ast\ast} =\min\{T^\iota,T^\ast,T^{\ast\ast}\}$, then by \eqref{ins12} and the fact that $\Lambda/2 < \lambda \le\Lambda$, we have that
\begin{equation}
\begin{split}
\norm{\eta^\iota_-(T^{\ast\ast})}_{0}   &\le\iota e^{\lambda T^{\ast\ast} }\norm{\eta^\star_-}_{0}
  +\| \eta_-^\iota(T^{\ast\ast})-\iota e^{\lambda T^{\ast\ast}}\eta^\star_-\|_{0}\\&\le \iota e^{\lambda T^{\ast\ast} }\norm{\eta^\star_-}_{0}
+ 2C_6\iota^2 e^{2\lambda T^{\ast\ast}}+C_6\iota^\frac{3}{2} e^{\frac{3}{2}\lambda t}
\\
& \le  \iota e^{\lambda T^{\ast\ast}} (1 +2C_6\iota e^{ \lambda T^{\iota}} +C_6\sqrt{\iota}  e^{\hal \lambda T^\iota})
 \\&\le \iota e^{\lambda T^{\ast\ast}}(1 +2C_6\theta_0
+C_6\sqrt{\theta_0})<2\iota e^{\lambda T^{\ast\ast}},
\end{split}
\end{equation}
which contradicts to the definition of $T^{\ast\ast}$.  Hence \eqref{Tmin} must hold, proving the claim.

Now we use \eqref{ins12} again  to find that
\begin{equation}
\begin{split}
\norm{\eta^\iota_-(T^\iota)}_{0}   &\ge \iota e^{\lambda T^\iota }\norm{\eta^\star_-}_{0}
  -\| \eta_-^\iota(T^\iota)-\iota e^{\lambda T^\iota}\eta^\star_-\|_{0}
  \\&\ge\iota e^{\lambda T^\iota } -2C_6\iota^2 e^{ 2\lambda T^\iota}-C_6\iota^\frac{3}{2} e^{\frac{3}{2}\lambda t}
  \\&\ge\theta_0- 2C_6\theta_0^2- C_6\theta_0^{\frac{3}{2}} \ge
 \frac{\theta_0}{2}.
\end{split}
\end{equation}
This completes the proof of Theorem \ref{maintheorem}.
\end{proof}
\appendix

\section{Analytic tools}\label{section_appendix}

\subsection{Poisson extensions}

We will now define the appropriate Poisson integrals that allow us to extend $\eta_\pm$, defined on the surfaces $\Sigma_\pm$, to functions defined on $\Omega$, with ``good'' boundedness.

Suppose that $\Sigma_+ = \mathrm{T}^2\times \{{j}\}$, where $\mathrm{T}^2:=(2\pi L_1 \mathbb{T}) \times (2\pi L_2 \mathbb{T})$. We define the Poisson integral in $\mathrm{T}^2 \times (-\infty,{j})$ by
\begin{equation}\label{P-1def}
\mathcal{P}_{-,1}f(x) = \sum_{\xi \in    (L_1^{-1} \mathbb{Z}) \times
(L_2^{-1} \mathbb{Z}) }  \frac{e^{i \xi \cdot x' }}{2\pi \sqrt{L_1 L_2}} e^{|\xi|(x_3-{j})} \hat{f}(\xi),
\end{equation}
where for $\xi \in  (L_1^{-1} \mathbb{Z}) \times (L_2^{-1} \mathbb{Z})$ we have written
\begin{equation}
 \hat{f}(\xi) = \int_{\mathrm{T}^2} f(x')  \frac{e^{- i \xi \cdot x' }}{2\pi \sqrt{L_1 L_2}} dx'.
\end{equation}
Here ``$-$'' stands for extending downward and ``${j}$'' stands for extending at $x_3={j}$, etc. It is well-known that $\mathcal{P}_{-,{j}}:H^{s}(\Sigma_+) \rightarrow H^{s+1/2}(\mathrm{T}^2 \times (-\infty,{j}))$ is a bounded linear operator for $s>0$. However, if restricted to the domain $\Omega$, we can have the following improvements.

\begin{lem}\label{Poi}
Let $\mathcal{P}_{-,{j}}f$ be the Poisson integral of a function $f$ that is either in $\dot{H}^{q}(\Sigma_+)$ or
$\dot{H}^{q-1/2}(\Sigma_+)$ for $q \in \mathbb{N}=\{0,1,2,\dots\}$, where we have written $\dot{H}^s(\Sigma_+)$ for the homogeneous Sobolev space of order $s$.  Then
\begin{equation}
\norm{\nabla^q \mathcal{P}_{-,{j}}f }_{0} \lesssim \norm{f}_{\dot{H}^{q-1/2}(\mathrm{T}^2)}^2 \text{ and }
\norm{\nabla^q \mathcal{P}_{-,{j}}f }_{0} \lesssim \norm{f}_{\dot{H}^{q}(\mathrm{T}^2)}^2.
\end{equation}
\end{lem}
\begin{proof}
 See Lemma A.3 of \cite{GT_per}.
\end{proof}

We extend $\eta_+$ to be defined on $\Omega$ by
\begin{equation}\label{P+def}
\bar{\eta}_+(x',x_3)=\mathcal{P}_+\eta_+(x',x_3):=\mathcal{P}_{-,{j}}\eta_+(x',x_3),\text{ for } x_3\le {j}.
\end{equation}
Then Lemma \ref{Poi} implies in particular that if $\eta_+\in H^{s-1/2}(\Sigma_+)$ for $s\ge 0$, then $\bar{\eta}_+\in H^{s}(\Omega)$.

Similarly, for $\Sigma_- = \mathrm{T}^2\times \{0\}$ we define the Poisson integral in $\mathrm{T}^2 \times (-\infty,0)$ by
\begin{equation}\label{P-0def}
\mathcal{P}_{-,0}f(x) = \sum_{\xi \in    (L_1^{-1} \mathbb{Z}) \times (L_2^{-1} \mathbb{Z}) }  \frac{e^{ i \xi \cdot x' }}{2\pi \sqrt{L_1 L_2}} e^{ |\xi|x_3} \hat{f}(\xi).
\end{equation}
It is clear that $\mathcal{P}_{-,0}$ has the  same regularity properties as $\mathcal{P}_{-,{j}}$. This allows us to extend $\eta_-$ to be defined on $\Omega_-$. However, we do not extend $\eta_-$ to the upper domain $\Omega_+$ by the reflection  since this will result in the discontinuity of the partial derivatives in $x_3$ of the extension. For our purposes,  we instead to do the extension through the following. Let $0<\lambda_0<\lambda_1<\cdots<\lambda_m<\infty$ for $m\in \mathbb{N}$ and define the $(m+1) \times (m+1)$ Vandermonde matrix $V(\lambda_0,\lambda_1,\dots,\lambda_m)$ by $V(\lambda_0,\lambda_1,\dots,\lambda_m)_{ij} = (-\lambda_j)^i$ for $i,j=0,\dotsc,m$.  It is well-known that the Vandermonde matrices are invertible, so we are free to let $\alpha=(\alpha_0,\alpha_1,\dots,\alpha_m)^T$ be the solution to
\begin{equation}\label{Veq}
V(\lambda_0,\lambda_1,\dots,\lambda_m)\,\alpha=q_m,\ q_m=(1,1,\dots,1)^T.
\end{equation}
Now we define the specialized Poisson integral in $\mathrm{T}^2 \times (0,\infty)$
by
\begin{equation}\label{P+0def}
\mathcal{P}_{+,0}f(x) = \sum_{\xi \in    (L_1^{-1} \mathbb{Z}) \times
(L_2^{-1} \mathbb{Z}) }  \frac{e^{ i \xi \cdot x' }}{2\pi \sqrt{L_1 L_2}}  \sum_{j=0}^m\alpha_j
e^{- |\xi|\lambda_jx_3} \hat{f}(\xi).
\end{equation}
It is easy to check that, due to \eqref{Veq}, $\partial_3^l\mathcal{P}_{+,0}f(x',0)=
\partial_3^l\mathcal{P}_{-,0}f(x',0)$  for all $0\le l\le m$ and hence
\begin{equation}
\partial^\alpha\mathcal{P}_{+,0}f(x',0)=
\partial^\alpha\mathcal{P}_{-,0}f(x',0), \ \forall\, \alpha\in \mathbb{N}^3 \text{ with }0\le |\alpha|\le m.\end{equation}
These facts allow us to  extend $\eta_-$ to be defined on $\Omega$
by
\begin{equation}\bar{\eta}_-(x',x_3)=
\mathcal{P}_-\eta_-(x',x_3):=\left\{\begin{array}{lll}\mathcal{P}_{+,0}\eta_-(x',x_3),\quad
x_3> 0 \\
\mathcal{P}_{-,0}\eta_-(x',x_3),\quad x_3\le
0.\end{array}\right.\label{P-def}\end{equation}
It is clear now that if $\eta_-\in H^{s-1/2}(\Sigma_-)$ for $ 0\le s\le m$, then $\bar{\eta}_-\in H^{s}(\Omega)$.  Since we will only work with $s$ lying in a finite interval, we may assume that $m$ is sufficiently large in \eqref{Veq} for $\bar{\eta}_- \in H^s(\Omega)$ for all $s$ in the interval.

\subsection{Estimates of Sobolev norms}

We will need some estimates of the product of functions in Sobolev spaces.

\begin{lem}\label{sobolev}
Let $U$ denote a domain either of the form $\Omega_\pm$ or of the form $\Sigma_\pm$.
\begin{enumerate}
 \item Let $0\le r \le s_1 \le s_2$ be such that  $s_1 > n/2$.  Let $f\in H^{s_1}(U)$, $g\in H^{s_2}(U)$.  Then $fg \in H^r(U)$ and
\begin{equation}\label{i_s_p_01}
 \norm{fg}_{H^r} \lesssim \norm{f}_{H^{s_1}} \norm{g}_{H^{s_2}}.
\end{equation}

\item Let $0\le r \le s_1 \le s_2$ be such that  $s_2 >r+ n/2$.  Let $f\in H^{s_1}(U)$, $g\in H^{s_2}(U)$.  Then $fg \in H^r(U)$ and
\begin{equation}\label{i_s_p_02}
 \norm{fg}_{H^r} \lesssim \norm{f}_{H^{s_1}} \norm{g}_{H^{s_2}}.
\end{equation}
\end{enumerate}
\end{lem}
\begin{proof}
These results are standard and may be derived, for example, by use of the Fourier characterization of the $H^s$ spaces and extensions.
\end{proof}


%

\subsection{Coefficient estimates}

Here we are concerned with how the size of $\eta$ can control the ``geometric'' terms that appear in the equations.
\begin{lem}\label{eta_small}
There exists a universal $0 < \delta < 1$ so that if $\ns{\eta}_{5/2} \le \delta,$ then
\begin{equation}\label{es_01}
\begin{split}
 & \norm{J-1}_{L^\infty(\Omega)} +\norm{A}_{L^\infty(\Omega)} + \norm{B}_{L^\infty(\Omega)} \le \hal, \\
 & \norm{\n-1}_{L^\infty(\Gamma)} + \norm{K-1}_{L^\infty(\Gamma)} \le \hal, \text{ and }  \\
 & \norm{K}_{L^\infty(\Omega)} + \norm{\mathcal{A}}_{L^\infty(\Omega)} \ls 1.
 \end{split}
\end{equation}
Also, the map $\Theta$ defined by \eqref{cotr} is a diffeomorphism.
\end{lem}
\begin{proof}
The estimate \eqref{es_01} is guaranteed by Lemma 2.4 of \cite{GT_per}.
\end{proof}

\subsection{Korn inequality}

Consider the following operators acting on functions $u $:
\begin{equation*}
 \sg{u} = \nab u + \nab u^T \text{ and } \sgz{u} = \sg{u} - \frac{2 \diverge{u}}{3} I,
\end{equation*}
or in components
\begin{equation*}
 \sg{u}_{ij} = \p_i u_j + \p_j u_i \text{ and } \sgz{u}_{ij} = \p_i u_j + \p_j u_i  - \frac{2 \p_k u_k}{3} \delta_{ij}.
\end{equation*}
Note that $\trace(\sgz{u})=0$.  As such, the operator $\sgz$ is referred to as the ``deviatoric part of the symmetric gradient.''

Now we record a version of Korn's inequality involving only the deviatoric part, $\sgz$,  that we will use for layered domains $\Omega_\pm$.

\begin{prop}\label{layer_korn}
There exists a constant $C>0$ so that
\begin{equation}
 \ns{u}_{1} \le C\ns{\sgz{u}}_0
\end{equation}
for all $u \in H^1(\Omega )$ with $\jump{u}=0$ along $\Sigma$ and $u_- =0$ on $\Sigma_b$.
\end{prop}
\begin{proof}
We refer to Proposition A.8 of \cite{JTW_GWP}.
\end{proof}



\subsection{Elliptic estimates}

Here we consider the two-phase elliptic problem
\begin{equation}\label{lame eq}
\begin{cases}
-\mu \Delta u -(\mu/3+\mu') \nabla\diverge u =F^2  &\hbox{ in }\Omega
\\-\S(u_+)e_3=F^3_+   &\hbox{ on }\Sigma_+
\\  -\jump{\S(u)}e_3=-F^3_- &\hbox{ on }\Sigma_-
\\ \jump{u}=0 &\hbox{ on }\Sigma_-
\\  u_-=0 &\hbox{ on }\Sigma_b.
\end{cases}
\end{equation}
We have the following elliptic regularity result.
\begin{lem}\label{lame reg}
Let $r\ge 2$. If $F^2\in  {H}^{r-2}(\Omega), F^3\in  {H}^{r-3/2}(\Sigma)$, then the problem \eqref{lame eq} admits a unique strong solution $u\in {H}^r(\Omega)$. Moreover,
\begin{equation}
\norm{u}_{r}\lesssim \norm{F^2}_{r-2}+\norm{F^3}_{r-3/2}.
\end{equation}
\end{lem}
\begin{proof}
We refer to \cite[Theorem 3.1]{WTK} for the case of two-phase Stokes problem, but the proof is the same here. It follows by making use of the flatness of the boundaries $\Sigma_\pm$ and applying the standard classical one-phase elliptic theory with Dirichlet boundary condition.
\end{proof}

We let $G$ denote a horizontal periodic slab with its boundary $\pa G$ (not necessarily flat) consisting of two smooth pieces. We shall recall the classical regularity theory for the Stokes problem with Dirichlet boundary conditions on $\pa G$,
\begin{equation}\label{stokes eq}
\begin{cases}
-\mu \Delta u +\nabla p =f  \quad &\hbox{in }G
\\\diverge{u} =h  \quad  &\hbox{in }G
\\u=\varphi\quad &\hbox{on }\pa G.
\end{cases}
\end{equation}
The following records the regularity theory for this problem.
\begin{lem}\label{stokes reg}
Let $r\ge 2$. If $f\in H^{r-2}(G),  h\in H^{r-1}(G), \varphi\in H^{r-1/2}(\pa G)$ be given such that
\begin{equation}
\int_G h =\int_{\pa G} \varphi \cdot\nu ,
\end{equation}
then there exists unique $u\in H^r(G),  p\in H^{r-1}(G)$(up to
constants) solving \eqref{stokes eq}. Moreover,
\begin{equation}
\norm{u}_{H^r(G)}+\norm{\nabla p}_{H^{r-2}(G)}\lesssim\norm{f}_{H^{r-2}(G)}+\norm{h}_{H^{r-1}(G)}+\norm{\varphi}_{H^{r-1/2}(\pa G)}.
\end{equation}
\end{lem}
\begin{proof}
 See \cite{L}.
\end{proof}

\end{document}